\documentclass[11pt]{article}
\usepackage{amsfonts,amsmath,latexsym,verbatim,amscd,amsthm,graphics}
\usepackage{amsthm,amsfonts,amssymb,euscript}
\usepackage{graphicx}
\newtheorem{theorem}{Theorem}[section]
\newtheorem{proposition}{Proposition}[section]

\newtheorem{lemma}{Lemma}[section]
\newtheorem{definition}{Definition}[section]
\newtheorem{remark}{Remark}[section]
\numberwithin{equation}{section}
\numberwithin{lemma}{section}
\numberwithin{theorem}{section}
\numberwithin{definition}{section}
\numberwithin{proposition}{section}

\allowdisplaybreaks[1]

\newcommand{\R}{\mathbb{R}}

\begin{document}
\title{Singular Ricci solitons and their stability under the Ricci flow.}
\author{Spyros Alexakis\thanks{Mathematics Department, University of Toronto. Email: alexakis@math.toronto.edu}\and
Dezhong Chen\thanks{Market Risk Measurement, Scotiabank. Email: dezhong.chen@scotiabank.com}\and
 Grigorios Fournodavlos\thanks{Mathematics Department, University of Toronto. Email: grifour@math.toronto.edu}}
\date{}
\maketitle
%

\noindent
%
\begin{abstract}
We introduce certain spherically symmetric singular Ricci solitons and
study their stability under the Ricci flow from a dynamical PDE point of view.
The solitons in question exist for all dimensions $n+1\ge 3$, and all have a  point singularity where the curvature blows up;
their evolution under the Ricci flow is in sharp contrast to the evolution of their smooth
counterparts. In particular, the family of diffeomorphisms associated with the Ricci flow ``pushes away'' from the singularity
causing the  evolving soliton  to open up immediately  becoming an incomplete (but non-singular) metric.
In the second part of this paper we study the local-in time stability of this dynamical evolution,
under spherically symmetric perturbations of the initial soliton metric.
We prove a local well-posedness result for the Ricci flow near the singular initial data,
which in particular implies  that the ``opening up'' of the singularity persists for the perturbations also.
\end{abstract}
\tableofcontents
\section{Introduction}\label{intro}
\noindent

The question of defining solutions of geometric evolution equations with singular
initial data is an interesting challenge and has been studied in recent years for a variety of
parabolic geometric PDE. For the Ricci flow, a number of solutions
have been proposed in various settings.  Simon \cite{S1}
 obtained solutions for the Ricci flow for ${\mathcal C}^0$  initial metrics that can be uniformly approximated 
by smooth metrics with bounded sectional curvature. Koch and Lamm \cite{KL} 
showed existence and uniqueness for the Ricci-DeTurck flow for initial data that are $L^\infty$-close 
to the Euclidean metric. 
Angenent, Caputo and Knopf \cite{ACK} considered initial data 
of neck-pinch type.\footnote{In particular these initial data can form 
 in the evolution of a smooth  spherically symmetric initial metric, as demonstrated in \cite{AK1,AK2}.}
They constructed a solution to the flow starting from  this singular initial metric, 
for which the singularity is immediately smoothed out.
This can be thought of as a (very weak) notion of surgery in that the method of proof
relies on a gluing construction to show the existence of such a solution,  but not uniqueness. 
Cabezas-Rivas and Wilking \cite{CRW} have obtained solutions of the Ricci flow on open manifolds with 
nonnegative (and possibly unbounded) complex sectional curvature, using the Cheeger-Gromoll convex
 exhaustion of such manifolds. 

More results have been obtained in the K\"ahler case and in dimension 2, where 
the Ricci flow equation reduces to a scalar heat equation; we list a few examples:   Chen,
Tian and Zhang \cite{CTZ} consider the K\"ahler-Ricci flow for initial data with ${\mathcal C}^{1,1}$
potentials and construct solutions to the Ricci flow which immediately smooth out.
The argument is based on an approximation of the initial potential by smoother ones. Finally,
more results have been obtained in dimension 2 (see \cite{IMS} for a survey): Giesen and Topping \cite{GT1} 
(building on earlier work by Topping \cite{T})
have given a construction of Ricci flows on surfaces starting from any (incomplete) initial metric
whose curvature is unbounded; these solutions become instantaneously complete and are unique
in the {\it maximally stretched} class that they introduce. More recently 
yet \cite{GT2}, they constructed examples of immortal solutions 
of the flow (on surfaces) which start out with a smooth initial metric, then the 
supremum of the Gauss curvature becomes infinite for some finite amount of time 
before becoming finite again. 

This paper considers a special class of singular initial metrics and produces examples of Ricci flow whose behavior is different
from those listed above. Our initial metrics  are {\it close}
to certain singular gradient Ricci solitons that we introduce separately in the first part of this paper.
The solitons exist in all dimensions $n+1\ge 3$.
Our main result is that for small enough perturbations of the singular Ricci solitons,
the Ricci flow admits a unique solution, up to some time $T>0$, within a natural class of evolving metrics
which stay close (as measured in a certain weighted Sobolev space)
to the evolving Ricci solitons. In other words, we obtain a local well-posedness result
for the Ricci flow for initial data with the same singularity profile as our Ricci solitons.

The solitons that we introduce (and in fact the perturbations that we consider)
all have $SO(n+1,\mathbb{R})$-symmetry.
In particular, the soliton metric at the initial time $t=0$  can be written in the form:
\begin{align*}
g_{\rm sol}=dx^2+\psi(x)^2 g_{\mathbb{S}^n},
\end{align*}
where $x\in (0,+\infty)$,  for certain steady
solitons (these we call half-complete, since they are complete towards $x\to+\infty$)
and $x\in(0,\delta)$, $\delta<+\infty$ in the
remaining cases\footnote{So in the general case the solitons can be extended to a smooth metric with boundary at $x=\delta$};
 here $g_{\mathbb{S}^n}$
denotes the canonical metric of the unit $n$-sphere.
In all cases
the function $\psi(x)$ is a positive smooth function
and moreover $\psi(x)\to 0$ as $x\rightarrow 0^+$, with leading order behaviour
$\psi\sim x^{\frac{1}{\sqrt{n}}}$.
In particular, the (incomplete) metric above can be extended to a
complete  ${\mathcal C}^0$ (in fact ${\mathcal C}^{\frac{1}{\sqrt{n}}}$) metric
at $x=0$, but the extended  metric will not be
of class  ${\mathcal C}^1$.
We remark that
(in the half-complete case) our (singular) solitons are complete Riemannian manifolds  towards $+\infty$, with an asymptotic profile there
that  matches  the Bryant soliton. For the rest of this introduction we discuss only the half-complete case.
\newline

Our first observation is that the evolution of the singular solitons themselves under the Ricci flow
is in sharp contrast with the behavior of their smooth counterparts.
As we note in Section \ref{evmet} below,
the evolution of the Ricci soliton $g_{\rm sol}$ is given by a 1-parameter family of radial\footnote{``Radial''
here and furtherdown means that the diffeomorhpism, for each $t\ge0$, depends only
on the parameter  $x\in(0,\infty)$.} diffeomorphisms
$\rho_t:(0,+\infty)\times\mathbb{S}^n\to(0,+\infty)\times\mathbb{S}^n$, $t\ge0$,
where $\rho_0={\rm Id}$ and such that the pullback $g(t)=\rho_t^*(g_{\rm sol})$ solves the Ricci flow
\begin{align*}
\partial_t g=-2{\rm Ric}(g),&&g(0):=g_{\rm sol}.
\end{align*}
However, the map $\rho_t$ is {\it not} surjective in this case. In fact,
for each $t>0$, $\rho_t(0,\infty)=(m(t),+\infty)$ where $m(t)>0$ is non-decreasing in $t$.
In other words the flow $\rho_t$  {\it pushes away} from the singular point $x=0$.
Thus, for each $t>0$  $(M,g(t))$ can be extended to a
smooth manifold with boundary, where the induced metric on the
boundary is that of a round sphere of radius $\lim_{x\to 0^+}\psi(\rho_t(x))>0$.
One can then visualize the evolving soliton metric $g(t)$ backwards in time: Starting at time $t=1$
it contains the portion of the original soliton corresponding to $x>m(t)$, and its
 boundary at $x=m(t)$  shrinks down, as $t\rightarrow0^+$,
to the singular metric $g_{\rm sol}$.

The perturbation problem that we consider is still within the spherically symmetric category. In
particular, the initial metrics we consider are in the form
\begin{align*}
\tilde{g}=dx^2+\tilde{\psi}^2(x)g_{\mathbb{S}^n}
\end{align*}
which is such that if we let
\begin{align*}
\xi=\frac{\tilde{\psi}}{\psi}-1,
\end{align*}
then $\xi$  is in a weighted $H^1$-Sobolev space
\begin{align*}
\xi\in H^1(0,+\infty)&&\int_0^1\frac{\xi^2}{x^{2\alpha}}+\frac{\xi_x^2}{x^{2\alpha-2}}dx<\infty
\end{align*}
for a large enough constant $\alpha$ which we determine below.
In particular, this assumption forces the perturbation $\tilde{g}$ to agree, at $t=0$,
with $g_{\rm sol}$ (asymptotically as $x\rightarrow0^+$) to a very high order, thus controlling the
asymptotics of the curvature tensor near the singularity.
We then show that (subject to suitable Dirichlet boundary conditions at $x=0$, $t\ge0$) there exists a {\it unique} evolving metric $\tilde{g}(t)$, $t\in[0,T]$,
solving the Ricci flow equation, and which stays ``close'' (measured
in  a suitable weighted $H^1$-space, the weight depending on both the spatial and the time variable)
to the evolving soliton metric; in particular it exhibits the same ``opening up''
behavior of the initial singularity. The precise statement can
be found in Theorem \ref{gensol}.

It should be stressed at this point that our work here {\it does not} have direct
bearing on the issue of ``flowing through singularities'' that form in finite time under
the Ricci flow, (as studied, for example, in \cite{ACK}),
at least for closed manifolds. Indeed, it is well known that
for such manifolds the minimum of the scalar curvature is a non-decreasing function
under the Ricci flow; however the scalar curvature of the solitons we consider (and of their perturbations)
converges to $-\infty$ at the singular point ($x=0$). Nonetheless, the issue of
proving well-posedness for geometric evolution equations with singular initial data
is an important challenge in many parts of this subject.
In fact, from a purely PDE point of view it seems natural to first attempt this well-posedness
in the neighborhood of solitons
(or more generally stationary solutions of the flow, the stationarity being
understood in a suitable sense depending on the context).
It is hoped that the methods developed herein will
prove useful in such pursuits in the future.

\subsection{Outline of the ideas}
\noindent

Now, we briefly outline the sections of the paper and the challenges that each addresses.
In Section \ref{coh1RS} we introduce the (singular) spherically symmetric  Ricci solitons that we consider.
The study of these solitons (all of which are gradient solitons, i.e., the
one parameter family of diffeomorphisms associated to them is generated
by the gradient of a function) follows the method presented in \cite[Chapter 1]{RFTA},
originally developed by R. Bryant. In the class of spherically symmetric metrics,
the gradient Ricci soliton equation reduces to a second order ODE system,
which can be transformed into a more tractable first order system in parameters $(W,X,Y)$  via a transformation that we review in (\ref{trans}).
Knowledge of the variables $W,X,Y$ in the parameter $y$ allows us to recover the metric component
$\psi$ and the gradient $\omega$ of the potential function $\phi$ of (\ref{ode}) in the parameter $x$.
In the case of steady solitons, the system (\ref{ode2}) in fact reduces to a $2\times2$ system; see \S\ref{steady}.
We provide a description of the trajectories in the $X,Y$-plane that correspond
to our singular solitons and compare them to the Bryant soliton.
In particular, we show there exists a 1-parameter family of singular gradient steady Ricci solitons;
they are all singular at $x=0$ with the leading order asymptotics
\begin{align*}
\psi(x)\sim x^{\frac{1}{\sqrt{n}}}&&\omega(x)\sim\frac{\sqrt{n}-1}{x},&&n>1
\end{align*}
and they are complete towards $x=+\infty$, with the same asymptotic profile as the Bryant soliton.

For completeness, we also consider the case of spherically symmetric
shrinking and expanding solitons; these correspond to trajectories of the $3\times3$ system (\ref{ode2}).
In this case we can again derive the existence of solutions of the above type, but only for
$x\in (0,\delta)$, for some small $\delta>0$;
the reason is that the behavior of the trajectories for large parameter $y$ is not well understood.
While the above solitons have been constructed over the manifolds $\mathbb{R}\times\mathbb{S}^n$,
it would perhaps be natural to seek similar examples in the more general cohomogeneity-1 category,
studied by Dancer and Wang, \cite{DW1, DW2, DW3}.

In Section \ref{pert} we introduce the perturbation problem
we will be studying in the rest of the paper. We consider spherically symmetric initial metrics of the form
\begin{align*}
\tilde{g}=\tilde{\chi}^2(x)dx^2+\tilde{\psi}^2(x)g_{\mathbb{S}^n}
\end{align*}
For such initial data, the Ricci flow equation can be written (after a change of variables)
in the equivalent form (\ref{RFtildeg_t}) of a PDE coupled to an ODE.
The evolving Ricci soliton metric (defined through a radial pullback) remains spherically symmetric and
is represented by coordinate components $\chi(x,t),\psi(x,t)$,
while the stipulated Ricci flow that we wish to solve for with the above type of (perturbed) initial data
corresponds to two functions $\tilde{\chi}(x,t), \tilde{\psi}(x,t)$.
Since the singular nature of the initial data do not allow the system (\ref{RFtildeg_t}) to be attacked directly,
we introduce new variables which measure the closeness of $\tilde{\chi},\tilde{\psi}$ to $\chi,\psi$.

More precisely, we define
\begin{align*}
\zeta=\frac{\tilde{\chi}}{\chi}-1&&\xi=\frac{\tilde{\psi}}{\psi}-1.
\end{align*}
 Then the system reduces to (\ref{pdezixi}),
for which the Ricci soliton corresponds to the solution $\zeta=0, \xi=0$.
The coefficients of this system refer to the variable $\psi$ of the background evolving soliton, expressed with respect to
its arc-length parameter $s$.
What is critical here is that the coefficients are singular at $(x,t)=(0,0)$;
the precise nature of this singularity is essential in our further analysis.

 A first challenge appears at this point,
which in fact is independent of the singularities of the coefficients. Indeed,
it is related to the presence of the second order term $\xi_{ss}$ on the RHS of the first equation in (\ref{pdezixi}).
Since the first equation is only of first order in $\zeta$, this term would {\it not} make it possible to close the energy
estimates for our system. We therefore introduce a new variable defined by
\begin{align*}
\eta=\frac{(\zeta+1)^2}{(\xi+1)^{2n}}-1.
\end{align*}
 The new system (\ref{pde}) for $\eta$ and $\xi$
involves only first derivatives of $\xi$ in the evolution equation of $\eta$ and therefore can (in principle) be
approached via energy estimates. It is unclear at present whether
there is any geometric significance underlying this change of variables; it is in fact not a priori obvious
that such a simplification of the system should have been possible via a change of variables.
It is at this point that the spherical symmetry of both the background soliton and of the perturbations that we study is used in an essential way.

Thus, matters are reduced to proving well-posedness of  (\ref{pde}), in the appropriate spaces.
We follow the usual approach of performing an iteration\footnote{
In reality a contraction mapping argument, although we find it more convenient to
phrase our proof in terms of the standard Picard iteration.}, by solving a sequence of linear equations for the unknows
$(\eta^{m+1},\xi^{m+1})$ in terms of the known functions  $(\eta^m,\xi^m)$ solved for in the previous step,
and proving that the sequence $(\eta^m,\xi^m), m\in\mathbb{N}$ converges
to a solution $(\eta,\xi)$ of our original system.

\par We note that the usual approach would be to replace only the highest order terms
in the RHSs of (\ref{pde}) by the unknown function $\xi^{m+1}$ and replace all the lower-order ones by the previously-solved-for $\eta^m,\xi^m$.
However in the case at hand  this approach would fail for any function space,  due
to the nature of the singularities in the coefficients. For example, as we will see the coefficient $\frac{\psi_s^2}{\psi^2}$
in the potential terms contains a factor of $\frac{1}{s^2}$, where $s(x,t)$ is the arc-length parameter of the
background evolving soliton. It turns out that the leading order
in the asymptotic expansion of $s^2$ near $x=0,t=0$ is of the form
\begin{align*}
s^2\sim x^2+(\sqrt{n}-1)t.
\end{align*}
Consequently, the best $L^\infty_x$ bound
for $\frac{1}{s^2}$ would be $\frac{1}{s^2}\le \frac{C}{t}$, which evidently fails to be integrable in time;
this would result in an energy estimate of the form $\partial_t{\mathcal E}\le \frac{\mathcal E}{t}$ which cannot close.
The remedy for this problem is to modify the iteration procedure according to (\ref{itpde}).
In this linear iteration the unknown functions $\xi^{m+1}, \eta^{m+1}$
at the $(m+1)$-step also appear in certain lower-order terms associated to the most
singular coefficients.

Finally, we solve the system (\ref{itpde}) and prove that it defines a contraction mapping
 in certain
(time-dependent) weighted Sobolev spaces
$H^1_\alpha(s)$ containing all functions
\begin{align*}
u\in H^1(\mathbb{R}_+)&&\int\frac{u^2}{(s^2+\sigma t)^\alpha}+\frac{u^2_s}{(s^2+\sigma t)^{\alpha-1}}ds<+\infty,
\end{align*}
where we note that the weights depend on both the spatial and time variables $x,t$.
We note here that the integration is over the entire half-line
$\mathbb{R}_+$,\footnote{Recall that for the purposes of the introduction we are only considering
the stability of the half-complete (steady) Ricci solitons.} but we use the length element $ds$
which corresponds to the arc-length parameter of the background evolving Ricci soliton.
In particular $s(x,t):=\rho_t(x)$; thus for all $t>0$ $s(x,t)>0, \forall x\ge 0$.

The rather involved estimates  in Section \ref{genprob} aim precisely to show that the parameters $\alpha$ and
$\sigma>0$ {\it can be}  chosen in a way to make the estimates close.
The Picard iteration we follow is presented in
\S\ref{contrit}, where we also assert that these linear systems are well-posed in the appropriate
energy spaces, and can be solved up to a uniform time $T>0$, with a uniform bound on the energies.
Assuming that the well-posedness of the linear systems
holds, the main claim is that these iterates form a Cauchy sequence. This is done in the subsequent subsections in
\S\ref{genprob}, by  first deriving bounds for the  iterates in the weighted $L^2$ spaces, and then in the appropriately
weighted $H^1$ spaces.

In Section \ref{modprob} we then prove the well-posedness of the linear systems (\ref{itpde});
in fact we prove this for a more general type of linear equations.
The approach here is to iteratively solve for a sequence of functions which solve a coupled PDE-ODE system
(a heat equation in $\xi$ and an ODE in $\eta$).
The ODE can be solved separately and the energy estimates follow readily.
The heat equation in $\xi$  is then solved by a Galerkin-type argument, proving that it admits  an a priori
weak solution (which a posteriori is a strong one)
to which the energy estimates can be applied. This part is included for the sake of completeness, since
coupled systems of this singular nature do not appear to have been
 treated in the literature.

\subsection*{Acknowledgements}
\noindent

We wish to thank McKenzie Wang for helpful conversations on cohomogeneity-1 Ricci solitons.
The first author was partially supported by NSERC grants 488916 and 489103, and a Sloan fellowship. 
The second author was partially supported by an NSERC postdoctoral fellowship.

\section{Singular spherically symmetric Ricci Solitons}\label{coh1RS}
\noindent

We will be considering metrics  over
$M^{n+1}=(0,B)\times\mathbb{S}^n$ (where $B\in\mathbb{R}_{+}$ or $B=+\infty$),
in the form
\begin{align}\label{metric}
g=dx^2+\psi^2(x)g_{\mathbb{S}^n},
\end{align}
where $\psi$ is a positive smooth function and $g_{\mathbb{S}^n}$ denotes the canonical
metric on the unit sphere.

Our first aim for this section is to obtain such metrics
which satisfy the Ricci soliton equation and which are singular as $x\to 0^+$.
In particular we wish to construct a soliton metric which will extend continuously to $x=0$ with $\psi(x)\rightarrow0$,
as $x\rightarrow0^+$,
but will not close smoothly there (see
\cite[Lemma A.2]{RFTA}). (In the case $B<+\infty$, we will consider $\psi(x)$
having a smooth limit at $x=B$).

Recall the gradient Ricci soliton equation
\begin{align}\label{gRSeq}
Ric(g)+\nabla\nabla\phi+\lambda g=0&&\lambda\in\mathbb{R},
\end{align}
for a smooth potential function $\phi:M\to\R$, depending only on $x$, $\phi(x,p)=\phi(x)$.
Generally, for metrics of the form (\ref{metric}), we can easily
check that the Ricci tensor is given by
\begin{align}\label{Ric}
Ric(g)=-n\frac{\ddot{\psi}}{\psi}dx^2+(n-1-\psi\ddot{\psi}-(n-1)(\dot{\psi})^2)g_{\mathbb{S}^n}
\end{align}
and the Hessian of a radial function $\phi$ by
\begin{align}\label{Hess}
\nabla\nabla\phi=\ddot{\phi}dx^2+\psi\dot{\psi}\dot{\phi}g_{\mathbb{S}^n},
\end{align}
where $\dot{}=\frac{d}{dx}$. Therefore, setting $\omega=\dot{\phi}$,
equation (\ref{gRSeq}) reduces to a
coupled ODE system of the form
\begin{eqnarray}\label{ode}
\left\{\begin{array}{l}
n\ddot{\psi}-\psi\dot{\omega}=\lambda\psi\\
\psi\ddot{\psi}+(n-1)\dot{\psi}^2-(n-1)-\psi\dot{\psi}\omega=\lambda\psi^2.
\end{array}\right.
\end{eqnarray}
Following \cite[Chapter 1, \S5.2]{RFTA}, we introduce the transformation
\begin{align}\label{trans}
W=\frac{1}{-\omega+n\frac{\dot{\psi}}{\psi}},&&X=\sqrt{n}W\frac{\dot{\psi}}{\psi},
&&Y=\frac{\sqrt{n(n-1)}W}{\psi},
\end{align}
along with a new independent variable $y$ defined by
\begin{align}\label{dy}
dy=\frac{dx}{W}.
\end{align}
For the above set of variables, our ODE system (\ref{ode}) becomes
\begin{align}\label{ode2}
\bigg(\;'=\frac{d}{dy}\bigg)
\left\{\begin{array}{l}
W'=W(X^2-\lambda W^2)\\
X'=X^3-X+\frac{Y^2}{\sqrt{n}}+\lambda(\sqrt{n}-X)W^2\\
Y'=Y(X^2-\frac{X}{\sqrt{n}}-\lambda W^2)
\end{array}\right.
\end{align}
We readily check that for the steady ($\lambda=0$) and the shrinking ($\lambda<0$) cases,
the equilibrium points of the above system are
\begin{align*}
(0,0,0)&&(0,\pm1,0)&&(0,\frac{1}{\sqrt{n}},\pm\sqrt{1-\frac{1}{n}}).
\end{align*}
In the expanding ($\lambda>0$) case, we have two additional critical points
$(\pm\frac{1}{\sqrt{\lambda n}},\frac{1}{\sqrt{n}},0)$.
\begin{remark}\label{Brsol}
In the steady case, $\lambda=0$, the system (\ref{ode2}) is in fact represented by the last
two equations ($W$ is a function of $X,Y$). The well-known Bryant soliton
corresponds to the unique solution of the system that emanates from
the equilibrium point,  $(\frac{1}{\sqrt{n}},\sqrt{1-\frac{1}{n}})$
and converges to the equilibrium $(X=0,Y=0)$
\cite[Chapter 1, \S4]{RFTA}.
\end{remark}
In the present article we are concerned with the trajectories emanating from the equilibrium point $(0,1,0)$,
for all $\lambda\in\mathbb{R}$ (in our primary analysis).

The linearization of (\ref{ode2}) at $(0,1,0)$ takes the diagonal form
\begin{align}\label{linode2}
\left(\begin{array}{c}
W\\
X-1\\
Y
\end{array}\right)'=
\left(\begin{array}{ccc}
1&0&0\\
0&2&0\\
0&0&1-\frac{1}{\sqrt{n}}
\end{array}\right)
\left(\begin{array}{c}
W\\
X-1\\
Y
\end{array}\right)
\end{align}
Note that for $n>1$, all eigenvalues (diagonal entries) are positive,
which implies that $(0,1,0)$ is a source of the system. Whence,
if a trajectory of (\ref{ode2}) is initially $(y=0)$
close to $(0,1,0)$, i.e.,
\begin{align*}
|(W(0),X(0)-1,Y(0))|<\varepsilon,
\end{align*}
for $\varepsilon>0$ sufficiently small (indicated by the RHS of (\ref{ode2})),
then standard ODE theory (e.g., see \cite{CL}) yields the estimate
\begin{align}\label{asym}
|(W(y),X(y)-1,Y(y))|\leq\sqrt{3}\varepsilon e^{\mu y},&&y\leq0,
\end{align}
for some $0<\mu<1-\frac{1}{\sqrt{n}}$ (least eigenvalue). The latter estimate improves
as the initial conditions approach the equilibrium point $(0,1,0)$; in other words
one can pick $\mu$ closer to the eigenvalue $1-\frac{1}{\sqrt{n}}$ by taking
$\varepsilon$ sufficiently small.
We will show that these trajectories correspond to a singularity of the original metric (\ref{metric}) at $x=0$.

\subsection{Asymptotic analysis at the singularity}\label{asyman}
\noindent

We will be considering solutions of the system (\ref{ode2}),
with $(W(0), X(0), Y(0))$  sufficiently close to the equilibrium point $(0,1,0)$ and
with $Y(0), W(0)>0$.
(The reflection-symmetric trajectories over $\{ Y=0\} $ and $\{W=0\}$ are easily seen to correspond to
the same metric, while the trajectories with $Y(0)=0$ do not to correspond to Riemannian metrics).

We proceed to derive the asymptotic behavior, as $y\rightarrow-\infty$, of the variables
$W,X,Y$. Changing
back to $x$, using (\ref{dy}), we determine the desired asymptotic behavior of
the unknown functions in the original system (\ref{ode}), as $x\rightarrow0^+$. The final estimates will
confirm that $x=0$ is actually a singular point of the metric $g$.

Let $X(y)=1+g(y)$. Plugging into the equation of $W'$ in (\ref{ode2}) we obtain
\begin{align*}
W(y)=W(0)\exp\big\{y+\int^y_0W(z)g(z)(2+g(z))dz-\lambda\int^y_0W^3(z)dz\big\},
\end{align*}
where according to (\ref{asym}), for $y\leq0$,
\begin{align*}
|\int^y_0W(z)g(z)(2+g(z))dz|\leq3\varepsilon^2(2+\sqrt{3}\varepsilon)\frac{1-e^{2\mu y}}{2\mu}
\end{align*}
and
\begin{align*}
|-\lambda\int^y_0W^3(z)dz|\leq|\lambda|\cdot3\sqrt{3}\varepsilon^3\frac{1-e^{3\mu y}}{3\mu}.
\end{align*}
Thus,
\begin{align*}
W(y)=&C_1e^y+W(0)e^y\bigg[\exp\big\{\int^y_0W(z)g(z)(2+g(z))dz-\lambda\int^y_0W^3(z)dz\big\}\\
&-\exp\big\{-\int^0_{-\infty}W(z)g(z)(2+g(z))dz+\lambda\int^0_{-\infty}W^3(z)dz\big\}\bigg],
\end{align*}
where $C_1=W(0)\exp\big\{-\int^0_{-\infty}W(z)g(z)(2+g(z))dz+\lambda\int^0_{-\infty}W^3(z)dz\big\}>0$.
Using (\ref{asym}) again, we readily estimate the second term as above and obtain
\begin{align*}
W(y)=C_1e^y+O(e^{(2\mu+1)y})&&y\leq0.
\end{align*}
Similarly, from the equation of $Y'$ (\ref{ode2}) we derive
\begin{align*}
Y(y)=C_2e^{(1-\frac{1}{\sqrt{n}})y}+O(e^{(2\mu+1-\frac{1}{\sqrt{n}})y})&&y\leq0.
\end{align*}
for an appropriate positive ($Y(0)>0$) constant $C_2$. As for $X$, we already have
directly from (\ref{asym}) the bound
\begin{align*}
X=1+g(y)=1+O(e^{\mu y})&&y\leq0,
\end{align*}
which we can retrieve from the equation of $X'$ by integrating on $(-\infty,y)$
and using (\ref{asym}), along with the previously derived estimates for $W,Y$.
Now we wish to derive asymptotics for the variables in (\ref{ode}).
According to the transformation (\ref{trans}), for $y\leq-M$ ($M>0$ large), we have
\begin{align*}
\psi=&\frac{\sqrt{n(n-1)}W}{Y}=
\frac{\sqrt{n(n-1)}(C_1e^y+O(e^{(2\mu+1)y}))}{C_2e^{(1-\frac{1}{\sqrt{n}})y}+O(e^{(2\mu+1-\frac{1}{\sqrt{n}})y})}
=\frac{\sqrt{n(n-1)}C_1}{C_2}e^{\frac{1}{\sqrt{n}}y}+O(e^{(2\mu+\frac{1}{\sqrt{n}})y})\\
\frac{\dot{\psi}}{\psi}=&\frac{X}{\sqrt{n}W}=
\frac{1+O(e^{\mu y})}{\sqrt{n}(C_1e^y+O(e^{(2\mu+1)y}))}=
\frac{1}{\sqrt{n}C_1}e^{-y}+O(e^{(\mu-1)y})\\
\omega=&n\frac{\dot{\psi}}{\psi}-\frac{1}{W}=
\frac{\sqrt{n}}{C_1}e^{-y}+nO(e^{(\mu-1)y})-\frac{1}{C_1e^y+O(e^{(2\mu+1)y})}=
\frac{\sqrt{n}-1}{C_1}e^{-y}+O(e^{(\mu-1)y}).
\end{align*}
Also, going back to the second equation of (\ref{ode}) and dividing both sides by
$\psi^2$ yields
\begin{align*}
\frac{\ddot{\psi}}{\psi}=&-(n-1)\frac{\dot{\psi}^2}{\psi^2}+\frac{n-1}{\psi^2}
+\frac{\dot{\psi}}{\psi}\omega+\lambda\\
=&-(n-1)\big[\frac{1}{\sqrt{n}C_1}e^{-y}+O(e^{(\mu-1)y})\big]^2
+\frac{n-1}{\big[\frac{\sqrt{n(n-1)}C_1}{C_2}e^{\frac{1}{\sqrt{n}}y}+O(e^{(2\mu+\frac{1}{\sqrt{n}})y})\big]^2}\\
&+\big[\frac{1}{\sqrt{n}C_1}e^{-y}+O(e^{(\mu-1)y})\big]
\big[\frac{\sqrt{n}-1}{C_1}e^{-y}+O(e^{(\mu-1)y})
\big]+\lambda\\
=&-\frac{\sqrt{n}-1}{n}\frac{e^{-2y}}{C^2_1}+O(e^{(\mu-2)y}).
\end{align*}
Furthermore, the first equation of (\ref{ode}) gives
\begin{align*}
\dot{\omega}=n\frac{\ddot{\psi}}{\psi}-\lambda=
-(\sqrt{n}-1)\frac{e^{-2y}}{C^2_1}+nO(e^{(\mu-2)y})+\lambda=
-(\sqrt{n}-1)\frac{e^{-2y}}{C^2_1}+O(e^{(\mu-2)y}).
\end{align*}
Having derived asymptotics, as $y\rightarrow-\infty$,
for all the unknown functions appearing in the problem, we would like
to derive  corresponding asymptotics in the independent variable $x$
that we started with. For that we recall (\ref{dy}) and normalize so that $x\to 0^+$ as $y\to -\infty$ to deduce
\begin{align*}
x=\int Wdy=\int C_1e^y+O(e^{(2\mu+1)y})dy=C_1e^y+O(e^{(2\mu+1)y})&&(y\leq0),
\end{align*}
 Hence, it follows
\begin{align*}
C_1e^y=x+O(x^{2\mu+1}),
\end{align*}
for $y\leq-M$, $M>0$ large. Going back to each of the above estimates, we conclude that (for $x\ll1$)
\begin{align}\label{xasym}
\notag W=x+O(x^{2\mu+1}),&\qquad X=1+O(x^\mu),\qquad Y=\frac{\sqrt{n(n-1)}}{a}x^{1-\frac{1}{\sqrt{n}}}
+O(x^{2\mu+1-\frac{1}{\sqrt{n}}}),\\
\psi=&ax^{\frac{1}{\sqrt{n}}}+O(x^{\frac{2\mu+1}{\sqrt{n}}}),\qquad
\frac{\dot{\psi}}{\psi}=\frac{1}{\sqrt{n}}\frac{1}{x}+O(x^{\mu-1}),\\
\omega=\frac{\sqrt{n}-1}{x}+O(x^{\mu-1}),&\qquad
\frac{\ddot{\psi}}{\psi}=-\frac{\sqrt{n}-1}{n}\frac{1}{x^2}+O(x^{\mu-2}),\qquad
\dot{\omega}=-\frac{\sqrt{n}-1}{x^2}+O(x^{\mu-2}),\notag
\end{align}
where $a=\frac{\sqrt{n(n-1)}C^{1-\frac{1}{\sqrt{n}}}_1}{C_2}>0$.

Our analysis verifies in particular that the above trajectories $(W,X,Y)$, emerging
from the critical point $(0,1,0)$, correspond to metrics which do not close up smoothly
as $x\rightarrow0^+$. We have proved the following proposition.
\begin{proposition}[General existence of singular solitons]\label{incgRS}
For all initial conditions $(W(0), X(0), Y(0))$, $Y(0)>0, W(0)>0$ sufficiently close to the equilibrium point $(0,1,0)$,
the unique solution of (\ref{ode2}) exists for all values $y\in(-\infty,0]$,
and corresponds to a Ricci soliton metric in the form (\ref{metric}). Along with the associated potential
function they solve the gradient Ricci soliton equation (\ref{gRSeq}) \big[reduced to (\ref{ode}) in this context\big]
and have the asymptotic profile
\begin{align*}
\psi(x)\sim ax^{\frac{1}{\sqrt{n}}},\;a>0&&\omega(x)\sim\frac{\sqrt{n}-1}{x}&&\text{as $x\rightarrow0^+$.}
\end{align*}
These metrics are a priori defined for $x\in (0,\delta)$,
for some $\delta>0$ small, such that $\psi,\omega$ have a smooth limit, as $x\rightarrow\delta^-$.
\end{proposition}
\begin{remark}\label{singsol}
In view of the above asymptotic behaviour of $\psi$, as $x\rightarrow0^+$, we conclude that
the gradient Ricci solitons $(M^{n+1},g,\nabla\phi,\lambda)$, $n>1$, $\lambda\in\mathbb{R}$,
that we have constructed can be extended to $x=0$ as ${\mathcal C}^0$ metrics but are
singular in $C^1$ norm.
\end{remark}
\begin{remark}\label{R}
Now we are able to find the limit of the scalar curvature of $g$, as $x\rightarrow0^+$, by
computing its leading term. Contracting the Ricci tensor of $g$ (\ref{Ric}), we obtain
\begin{align*}
R(g)=&-n\frac{\ddot{\psi}}{\psi}+\frac{n}{\psi^2}(n-1-\psi\ddot{\psi}-(n-1)\dot{\psi}^2)\\
\tag{in the sense of (\ref{xasym})}
\sim&\;2(\sqrt{n}-1)\frac{1}{x^2}+\frac{n(n-1)}{a^2x^{\frac{2}{\sqrt{n}}}}
-(n-1)\frac{1}{x^2},&&\text{as $x\rightarrow0^+$}
\end{align*}
It follows that (for $n>1$)
\begin{align*}
\lim_{x\rightarrow0^+}R(g)=-\infty.&&\blacksquare
\end{align*}
\end{remark}
\begin{remark}\label{(0,-1,0)}
One could also consider the trajectories which emanate
from the other equilibrium $(0,-1,0)$ of (\ref{ode2})
(also a source);
these in fact correspond to solitons with profile
\begin{align*}
\psi(x)\sim x^{-\frac{1}{\sqrt{n}}}&&\omega(x)=-\frac{1+\sqrt{n}}{x},&&\text{as $x\rightarrow0^+$}.
\end{align*}
They are in fact defined for all dimensions $n+1\ge 2$, and in
the steady case $(\lambda=0)$, dimension $n+1=2$,
can be explicitly written out as:
\begin{align*}
\psi(x)=\frac{1}{x}&&\omega(x)=-\frac{2}{x},&&x\in(0,+\infty).
\end{align*}
Notice that these metrics are also singular at $x=0$, but their
evolution under the Ricci flow is almost the opposite
from the metrics we obtain near the equilibrium at $(0,1,0)$;
in particular they remain singular for all time.
These solitons  are however beyond the scope of this paper.
\end{remark}

\subsection{The half-complete steady solitons}\label{steady}
\noindent

In the steady case, $\lambda=0$, if we consider solutions of (\ref{ode2})
with initial point $(X(0), Y(0))$ satisfying $X^2(0)+Y^2(0)<1$, $Y(0)>0$
and lying close enough to $(1,0)$, the solutions
of (\ref{ode3}) exist for
all  $y\in(-\infty,+\infty)$. In fact, these trajectories emanating from $(1,0)$
translate back to Ricci soliton metrics of the form (\ref{metric}), which exist
(and are smooth) for all $x\in(0,+\infty)$.
We will call these Ricci solitons {\it half-complete}.
This is to distinguish them from the general Ricci solitons discussed in
the previous subsection, which are only
known to exist  for $x\in (0,\delta)$ for some small $\delta>0$.
We will focus on the half-complete solitons/trajectories in this subsection.

Fortunately, the known work in the corresponding context of smooth solitons
\cite[Chapter 1, \S4]{RFTA} can be directly adapted to our setting.
In particular,
the equation of $W'$ in (\ref{ode2}) becomes redundant
and hence (\ref{ode2}) reduces to
\begin{align}\label{ode3}
\left\{\begin{array}{l}
X'=X^3-X+\frac{Y^2}{\sqrt{n}}\\
Y'=Y(X^2-\frac{X}{\sqrt{n}})
\end{array}\right.
\end{align}
Therefore, in this framework we deal with trajectories in $\mathbb{R}^2$.
We remark that the Bryant soliton corresponds to the unique trajectory emanating from
the equilibrium point $(\frac{1}{\sqrt{n}},\sqrt{1-\frac{1}{n}})$ and converging (as $y\to+\infty$) to the origin $(0,0)$.
On the other hand,
the source considered in (\ref{linode2}) corresponds  to the point $(1,0)$.
An important tool in the analysis of these trajectories is the Lyapunov function
\begin{align}\label{Lyap}
L=X^2+Y^2,&&L'=X^2(L-1).
\end{align}
Starting ($y=0$) with $(X(0),Y(0))$ near $(1,0)$ inside the  disc $X^2+Y^2<1$
with $Y(0)>0$, we easily conclude that
the trajectory $(X(y), Y(y))$ approaches the origin $(0,0)$, as $y\rightarrow+\infty$ (at an exponential rate).
Whence, it is defined for all $y\in(-\infty,+\infty)$,
leading to singular gradient steady solitons $(M^{n+1},g,\nabla\phi)$,
as in Proposition \ref{incgRS}, with $M^{n+1}=(0,+\infty)\times\mathbb{S}^n$.
One can easily see that the set of all such trajectories fills up the domain in the unit disc bounded by the Bryant
soliton trajectory \big(which emanates from $(\frac{1}{\sqrt{n}},\sqrt{1-\frac{1}{n}})$\big)
and the positive $X$-axis lying inside the unit disc.
\begin{remark}\label{compatinfty}
Following \cite[Chapter 1, \S4]{RFTA}, we derive that the soliton metrics
corresponding to the $(X,Y)$-orbits above are complete towards
$x=+\infty$.
\end{remark}
In addition, we get \cite[\S1.4]{RFTA}
that the behaviour  of $\psi(x)$ and its derivatives, as $x\rightarrow+\infty$, is
the same with ones corresponding to the Bryant soliton.
\begin{align}\label{asyminfty}
cx^{\frac{1}{2}}\leq\psi\leq Cx^{\frac{1}{2}}&&
cx^{-\frac{1}{2}}\leq\dot{\psi}\leq Cx^{-\frac{1}{2}}
&&-Cx^{-\frac{3}{2}}\leq\ddot{\psi}\leq-cx^{-\frac{3}{2}},
\end{align}
for some positive constants $c,C$.
The above estimates suffice to derive bounds from (\ref{ode}), as $x\rightarrow+\infty$,
for all the other variables we will need below. Indeed, we readily estimate (for $x\ge M$ large)
\begin{align}\label{asyminfty2}
\frac{\ddot{\psi}}{\psi}=O(\frac{1}{x^2})&&\dot{\omega}=O(\frac{1}{x^2})&&
\frac{\dot{\psi}}{\psi}=O(\frac{1}{x})
&&\frac{1}{\psi^2}=O(\frac{1}{x})&&-C<\omega<-c
\end{align}
Combined with asymptotics (\ref{xasym}) of the trajectories
at the singular point $x=0$ we have the following proposition.
\begin{proposition}[Half-complete steady solitons]\label{steadysol}
There exists a $1$-parameter family of singular steady solitons $(M^{n+1},g,\nabla\phi)$ with
$M^{n+1}=(0,+\infty)\times\mathbb{S}^n$ and asymptotic profile at $x=0$ as in Proposition \ref{incgRS}.
These solitons correspond to the trajectories $(X,Y)$ of (\ref{ode3})
emanating from the equilibrium point $(1,0)$ and lying in the first (open) quadrant of the $XY$-plane inside
the unit disk $X^2+Y^2<1$. Further, the asymptotic behaviour of the solitons at $x=+\infty$ matches those of the
Bryant soliton.
\end{proposition}
\begin{remark}\label{explsol}
It is worth noting that in dimension five, (i.e., $n=4$)
the soliton metrics and associated diffeomorphisms can in fact be written out explicitly:
\begin{align}\label{explprof}
\psi(x)=a\sqrt{x}&&\omega(x)=\frac{1}{x}-\frac{6}{a^2},&&a>0.
\end{align}
\end{remark}
\subsection{The evolving soliton metric $g(t)$ and the action of the diffeomorphisms.}\label{evmet}
\noindent

Since the metric $g$ (\ref{metric}) satisfies the gradient Ricci soliton equation (\ref{gRSeq}),
the evolving metric $g(t)$  evolves
under the Ricci flow
\begin{align*}
\partial_tg(t)=-2Ric\big(g(t)\big)&&g_0=g,
\end{align*}
only via diffeomorphisms. In particular,
\begin{align*}
g(t)=(1+2\lambda t)\cdot\rho_t^*(g)&&t\in[0,T],
\end{align*}
up to some time $T>0$ such that
\begin{align*}
\epsilon(t):=1+2\lambda t>0,
\end{align*}
where
\begin{align*}
\rho_t(x,p)=\rho_t(x)&&\rho_0=id_M
\end{align*}
is the flow generated by the (time dependent) vector field
\begin{align*}
\frac{1}{\epsilon(t)}grad_g\phi.
\end{align*}
Thus,
\begin{align}\label{g_t}
g(t)=\epsilon(t)\big[d(\rho_t(x))^2+\psi^2(\rho_t(x))g_{\mathbb{S}^n}\big]
\end{align}
We note that since
our manifold $(M^{n+1},g)$ is not complete (at $x=0$), $\rho_t(x)$ is not necessarily
defined for all $t$,
but nevertheless it exists locally in time, $t\in(-\varepsilon_x,\varepsilon_x), x>0$.
However, as we will
see below, by studying the leading behavior of $\rho_t(x)$, as $x\rightarrow0^+$,
for $x\in (0,\delta)$  the flow
exists for all future time $t\in  [0,T)$, for some small $T>0$.
In particular, for the half-complete steady Ricci solitons, which are our main focus here, the
flow exists for all $t\ge 0$.

Henceforth, we will suppress the sphere coordinates corresponding to different points $(x,p)$,
$(x,q)$ in $M^{n+1}$ so that we may consider $\rho_t$, for each time $t$, to be a real function in $x$
\begin{align*}
\rho_t:(0,\delta)\to(0,+\infty)\text{,}\text{ } \text{or} \text{ }\rho_t:(0,+\infty)\to (0,+\infty)
\end{align*}
Also, abusing notation we identify the time derivative of $\rho_t$ with its coefficient relative to
the coordinate field $\frac{\partial}{\partial x}$ component (the others
are in fact zero), that is,
\begin{align}\label{partial_trho_t}
\partial_t\rho_t(x)=\frac{1}{\epsilon(t)}(\nabla_{\partial/\partial x}\phi)_{\rho_t(x)}
=\frac{1}{\epsilon(t)}\dot{\phi}(\rho_t(x))=\frac{1}{\epsilon(t)}\omega(\rho_t(x)).
\end{align}
According to our asymptotic analysis (recall (\ref{xasym})),
\begin{align}\label{rhoasym}
\partial_t\rho_t(x)=\frac{1}{\epsilon(t)}\frac{\sqrt{n}-1}{\rho_t(x)}
+\frac{1}{\epsilon(t)}O(\rho_t(x)^{\mu-1})
\end{align}
for $x,t$ small enough, which also yields
\begin{align*}
\frac{1}{\epsilon(t)}\frac{\sqrt{n}-1}{2\rho_t(x)}\leq\partial_t\rho_t(x)\leq
\frac{2}{\epsilon(t)}\frac{\sqrt{n}-1}{\rho_t(x)}
\end{align*}
or
\begin{align*}
\frac{\sqrt{n}-1}{\epsilon(t)}\leq2\rho_t(x)\partial_t\rho_t(x)\leq
\frac{4(\sqrt{n}-1)}{\epsilon(t)}.
\end{align*}
Integrating on a small time interval we deduce
\begin{align*}
\rho_0^2(x)+\int^t_0\frac{\sqrt{n}-1}{1+2\lambda\tau}d\tau\leq\rho_t^2(x)\leq
\rho_0^2(x)+\int^t_0\frac{4(\sqrt{n}-1)}{1+2\lambda\tau}d\tau
\end{align*}
yielding
\begin{align}\label{push}
\left\{\begin{array}{l}
x^2+(\sqrt{n}-1)t,\;\;\lambda=0\\
x^2+\frac{\sqrt{n}-1}{2\lambda}\log(1+2\lambda t),\;\;\lambda\neq0
\end{array}\right.
\leq\rho^2_t(x)\leq\left\{\begin{array}{l}
x^2+4(\sqrt{n}-1)t\\
x^2+2\frac{\sqrt{n}-1}{\lambda}\log(1+2\lambda t)
\end{array}\right.
\end{align}
and in particular
\begin{align}\label{push2}
\rho^2_t(x)\geq\left\{\begin{array}{l}
(\sqrt{n}-1)t,\;\;\lambda=0\\
\frac{\sqrt{n}-1}{2\lambda}\log(1+2\lambda t),\;\;\lambda\neq0
\end{array}\right.
\end{align}
\begin{remark}\label{flow}
Notice that
for all $\lambda\in\mathbb{R}$, the lower bounds
are strictly positive (when $n>1$); provided $1+2\lambda t>0$. Thus,
\begin{align*}
\rho_t\big((0,\delta)\big)\subseteq(m(t),+\infty)&&m(t)>0,\;t>0
\end{align*}
and in particular $\rho_t$ is {\it not} surjective.
We further note that $m(t)$ is non-decreasing.
A geometric interpretation of the latter is that the flow $\rho_t$ ``pushes''
the domain away from the singularity at $x=0$.
\end{remark}
We return now to (\ref{partial_trho_t}) and focus on the singular steady solitons
constructed in \S\ref{steady}.
Besides the estimate (\ref{rhoasym}) which holds close to $(x,t)=(0,0)$,
by (\ref{asyminfty2}) it follows that
\begin{align}\label{partial_trho_tinfty}
-C\leq\partial_t\rho_t(x):=\omega(\rho_t(x))\leq -c,
\end{align}
for $c,C>0$ and $x>M$ large, $t$ small. Whence,
integrating on $[0,t]$ we arrive at
\begin{align}\label{rho_tinfty}
x-Ct\leq\rho_t(x)\leq x-ct,
\end{align}
for an appropriate constant $C>0$.
\newline

Further,  we can give a complete description of the evolution of the singular Ricci solitons
(Proposition \ref{steadysol}).
Indeed, in this case we derive that there is a critical slice $\{x_{\rm crit}\}\times\mathbb{S}^n$
of the manifold $M^{n+1}=(0,+\infty)\times\mathbb{S}^n$,
which is invariant under the flow $\rho_t(\cdot)$. In particular:
\begin{align}\label{omegasign}
\omega(x)>0,\;x\in(0,x_{\rm crit})&&\omega(x_{\rm crit})=0&&\omega(x)<0,\;(x_{\rm crit},+\infty)
\end{align}
Thus, for any  point $x\in(0,+\infty)$:
\begin{align*}
\lim_{t\rightarrow+\infty}\rho_t(x)=x_{\rm crit}
\end{align*}
and in particular regarding Remark \ref{flow}
\begin{align*}
\lim_{t\rightarrow+\infty}m(t)=x_{\rm crit}&&\lim_{t\rightarrow+\infty}\rho_t\big((0,+\infty)\big)=[x_{\rm crit},+\infty).
\end{align*}
Since the soliton in this case evolves via the pull-back of the flow $\rho_t^*$,
this implies that the limit of the Riemannian  manifolds $(M^{n+1},g(t))$, as $t\rightarrow+\infty$, is the manifold with
boundary $M^{n+1}_{\infty}=(0,+\infty)\times\mathbb{S}^n$ associated with the pseudo-metric
$0dx^2+\psi^2(x_{\rm crit})g_{\mathbb{S}^n}$.

In order to prove the above picture, it suffices to show that (\ref{omegasign}) is valid. Then the rest
follow from a monotonicity argument and the uniqueness of solutions to ODEs.
From the estimates (\ref{xasym}), (\ref{asyminfty2}) we confirm that
$\omega$ is positive close to $x=0$ and is negative near $+\infty$, hence there exists a point $x_{\rm crit}$
where $\omega(x_{\rm crit})=0$.
It remains to show that this is the only zero of $\omega$.
We recall at this point a general identity for
solutions to the gradient Ricci soliton equation (\ref{gRSeq})
(see for instance \cite[Proposition 1.15]{RFTA}).
\begin{proposition}\label{genid}
Let $(M^m,g,\nabla\phi)$ be a gradient Ricci soliton, i.e., a solution of the equation (\ref{gRSeq}).
Then the following quantities are constant:
\begin{align*}
(i)&\;\;R+\Delta_g\phi+m\lambda=0\qquad\qquad\text{(tracing)}\\
(ii)&\;\;R+|\nabla_g\phi|^2+2\lambda\phi=C_0,
\end{align*}
where $R$ is the scalar curvature of $(M^m,g)$.
\end{proposition}

Subtracting the two identities of the preceding proposition we obtain
\begin{align*}
\Delta_g\phi-|\nabla_g\phi|^2-2\lambda\phi+m\lambda=-C_0
\end{align*}
Whence, in our context where $\lambda=0$, the previous equation amounts to
\begin{align}\label{phieqn}
\ddot{\phi}+\frac{\dot{\psi}}{\psi}\dot{\phi}-\dot{\phi}^2=-C_0,
\end{align}
{\it Claim}: $C_0>0$. According to the asymptotics of $\omega:=\dot{\phi}$ (\ref{xasym}), (\ref{asyminfty2}),
we easily deduce that $\phi$ tends to $-\infty$ at both ends of the manifold $x\rightarrow 0^+,x\rightarrow +\infty$.
This implies that $\phi$ has a global maximum $M$, realized at some point $\tilde{x}$.
By (\ref{phieqn}) we get $C_0\ge0$. However, the constant $C_0$ cannot be zero, otherwise
we would have $\phi\equiv M$ (by uniqueness of ODEs), which of course is not possible.
Our claim follows.

Thus, every critical point of $\phi$ is a {\it strict} local maximum. Therefore, $\phi$ can only
have one critical point,  $x_{\rm crit}=\tilde{x}$.$\qquad\blacksquare$

\section{The Stability problem}\label{pert}
\noindent

Our main goal in this paper is to prove (local in time) well-posedness of the Ricci
flow for spherically symmetric metrics which are
close enough  (in certain spaces we construct in \S\ref{wSsp})
to the soliton metrics (Propositions \ref{incgRS} and \ref{steadysol}) we constructed in the previous section.
We recall below a useful form of the Ricci flow equation for
spherically symmetric metrics and then proceed to introduce a transformation of our system into new variables
$\zeta,\xi$ (\ref{zixi}).
These are designed to capture the closeness of the (putative)
evolving solution under the Ricci flow to the evolution of the
background Ricci soliton.
The resulting system (after a further change of variables, replacing $\zeta$ with a new variable $\eta$)
involves a second order parabolic equation in $\xi$ coupled with a 1st order equation in $\eta$,
both of them having certain singular coefficients.

The singular coefficients are determined fully by the background evolving soliton metric.
The precise asymptotics of these coefficients are essential to our further pursuits, so we begin
by studying those in the following subsection \ref{rotRF}. Next, in \S\ref{wSsp} we set up the function spaces
in which we will be proving our well-posedness result. In
\S\ref{contrit} we set up the iterative procedure by which we will show
the existence and uniqueness of solutions to the system (\ref{pde}), which
implies the desired existence and uniqueness of solutions to the original system (\ref{RFtildeg_t}) and
ultimately (\ref{RFtilde}).

One final convention: We will be considering perturbations of both the half-complete solitons
constructed in \S\ref{steady} and the general solitons in \S\ref{asyman} and proving the well-posedness of the Ricci-flow
in both those settings.\footnote{Recall that the half-complete solitons are steady, i.e. $\lambda=0$.}
Whenever a distinction needs to be made below, the case of  perturbations of the
half-complete solitons will be denoted by {\bf HC}, while case of the perturbations of the general solitons by {\bf G}.

\subsection{The Ricci flow for the perturbed metric: A transformed system}\label{rotRF}
\noindent

Let $\tilde{g}$ be a spherically
symmetric metric on $M^{n+1}=(0,B)\times\mathbb{S}^n$ ($B=\delta>0$ or $B=+\infty$),
\begin{align}\label{tildeg}
\tilde{g}=\tilde{\chi}^2(x)dx^2+\tilde{\psi}^2(x)g_{\mathbb{S}^n},
\end{align}
where $\tilde{\chi},\tilde{\psi}$ are positive smooth functions.
Suppose $\tilde{g}(t)$ solves the Ricci flow
\begin{align}\label{RFtilde}
\partial_t\tilde{g}(t)=-2Ric\big(\tilde{g}(t)\big)&&t\in[0,T],\;\;\tilde{g}_0=\tilde{g}.
\end{align}
Also, assume that $\tilde{g}(t)$ is of the form
\begin{align}\label{tildeg_t}
\tilde{g}(t)=\tilde{\chi}^2(x,t)dx^2+\tilde{\psi}^2(x,t)g_{\mathbb{S}^n}.
\end{align}
We now let $\tilde{s}(x,t)$ be the radial arc-length parameter for the above metric at any given time $t$. In other words,
we define
\begin{align}\label{tildes}
d\tilde{s}=\tilde{\chi}(x,t)dx.
\end{align}
Now, expressing $\tilde{\psi}(\cdot, t)$ relative to the parameter $\tilde{s}$ (and slightly abusing notation),
$\tilde{g}(t)$ becomes
\begin{align*}
\tilde{g}(t)=d\tilde{s}^2+\tilde{\psi}^2(\tilde{s},t)g_{\mathbb{S}^n}.
\end{align*}
As in (\ref{Ric}), for each $t\ge 0$, the Ricci tensor of $\tilde{g}(t)$ is given by
\begin{align*}
Ric\big(\tilde{g}(t)\big)=-n\frac{\tilde{\psi}_{\tilde{s}\tilde{s}}}{\tilde{\psi}}d\tilde{s}^2
+(n-1-\tilde{\psi}\tilde{\psi}_{\tilde{s}\tilde{s}}-(n-1)\tilde{\psi}_{\tilde{s}}^2)g_{\mathbb{S}^n}.
\end{align*}
Thus, the Ricci flow (\ref{RFtilde}), under the above assumptions, reduces to
the following coupled PDE system:
\begin{align*}
\left\{\begin{array}{l}
2\tilde{\chi}\tilde{\chi}_t=-2(-n\frac{\tilde{\psi}_{\tilde{s}\tilde{s}}}{\tilde{\psi}})d\tilde{s}^2(\partial_x,\partial_x)
=2n\frac{\tilde{\psi}_{\tilde{s}\tilde{s}}}{\tilde{\psi}}\tilde{\chi}^2\\
2\tilde{\psi}\tilde{\psi}_t
=-2(n-1-\tilde{\psi}\tilde{\psi}_{\tilde{s}\tilde{s}}-(n-1)\tilde{\psi}_{\tilde{s}}^2)
\end{array}\right.
\end{align*}
or simply
\begin{align}\label{RFtildeg_t}
\left\{\begin{array}{l}
\tilde{\chi}_t=n\frac{\tilde{\psi}_{\tilde{s}\tilde{s}}}{\tilde{\psi}}\tilde{\chi}\\
\tilde{\psi}_t=\tilde{\psi}_{\tilde{s}\tilde{s}}-(n-1)\frac{1-\tilde{\psi}_{\tilde{s}}^2}{\tilde{\psi}}
\end{array}\right.&&
\tilde{\chi}(x,0)=\tilde{\chi}(x),\;\tilde{\psi}(x,0)=\tilde{\psi}(x),\;t\in[0,T].
\end{align}
(Note that the first equation involves the evolution of the radial distance function,
while the second involves the evolution of the radii of the spheres, at a given radial distance).
\newline

We return now to the metric
of the (singular) gradient Ricci soliton $(M^{n+1},g,\nabla\phi,\lambda)$ (Proposition \ref{incgRS})
and its corresponding evolving metric (\ref{g_t})
\begin{align*}
g(t)=\epsilon(t)\big[d(\rho_t(x))^2+\psi^2(\rho_t(x))g_{\mathbb{S}^n}\big],
\end{align*}
defined for all $t$ such that $\epsilon(t)=1+2\lambda t>0$.
Let
\begin{align}\label{s}
s(x,t)=\sqrt{\epsilon(t)}\,\rho_t(x),\;\;s(x,0)=x&&ds=\sqrt{\epsilon(t)}\,\partial_x\rho_t(x)dx,
\end{align}
transforming $g(t)$ into
\begin{align*}
g(t)=ds^2+\psi^2(s,t), g_{\mathbb{S}^n}=
\chi^2(x,t)dx^2+\psi^2(x,t)g_{\mathbb{S}^n},
\end{align*}
where
\begin{align}\label{chipsi}
\chi(x,t):=\sqrt{\epsilon(t)}\,\partial_x\rho_t(x)&&\psi(x,t):=\sqrt{\epsilon(t)}\,\psi(\rho_t(x)).
\end{align}
Note that $\psi(x,0)=\psi(x)$ is the original component of the metric $g$ (\ref{metric}).
Arguing similarly to the case of $\tilde{g}(t)$,
we find that the evolution of $g$ via
\begin{align*}
\partial_tg(t)=-2Ric\big(g(t)\big)&&g_0=g
\end{align*}
is equivalent to the coupled system
\begin{align}\label{RFg_t}
\left\{\begin{array}{l}
\chi_t=n\frac{\psi_{ss}}{\psi}\chi\\
\psi_t=\psi_{ss}-(n-1)\frac{1-\psi_s^2}{\psi}
\end{array}\right.&&\chi(x,0)=1,\;\psi(x,0)=\psi(x).
\end{align}

We now take a first step towards transforming our system of equations into new variables. Let
\begin{align}\label{zixi}
\zeta=\frac{\tilde{\chi}}{\chi}-1&&\xi=\frac{\tilde{\psi}}{\psi}-1.
\end{align}
The above formulas are defined for all $x\in (0,B), t\in [0,T]$.
In particular, these variables measure (in
a refined way) the difference between the unknown functions $\tilde{\chi}, \tilde{\psi}$ and
the background variables $\chi,\psi$. Note in addition that requiring  $\xi=0$ at the endpoint
$x=0,t=0$ forces $\tilde{\psi}$ to have the same leading order asymptotics at $x=0$ as the background
component $\psi$.

We next wish to convert (\ref{RFtildeg_t}) into a system of equations
for $\zeta,\xi$. We wish to obtain an evolution equation in the parameters $s$ and $t$ ($s$ being
the arc-length parameter of the background evolving Ricci soliton).
We are then forced to deal with the discrepancy between $\tilde{s},s$, which
are the arc-length parameters for the evolving metrics $g(t), \tilde{g}(t)$. We calculate:
\begin{align*}
&\partial_{\tilde{s}}\overset{(\ref{tildes})}{=}\frac{1}{\tilde{\chi}}\partial_x
=\frac{\chi}{\tilde{\chi}}\frac{1}{\chi}\partial_x
\overset{(\ref{s}),(\ref{chipsi})}{=}\frac{1}{\zeta+1}\partial_{s}\\
&\partial_{\tilde{s}}\partial_{\tilde{s}}=\frac{1}{\zeta+1}\partial_s(\frac{1}{\zeta+1}\partial_s)
=\frac{1}{(\zeta+1)^2}\partial_s\partial_s-\frac{\zeta_s}{(\zeta+1)^3}\partial_s,
\end{align*}
and hence we write
\begin{align*}
\tilde{\psi}_{\tilde{s}}=&\frac{1}{\zeta+1}\big(\psi(\xi+1)\big)_s\\
\tilde{\psi}_{\tilde{s}\tilde{s}}=&
\frac{1}{(\zeta+1)^2}\big(\psi(\xi+1)\big)_{ss}-\frac{\zeta_s}{(\zeta+1)^3}\big(\psi(\xi+1)\big)_s.
\end{align*}
Taking time derivatives in (\ref{zixi})
and combining (\ref{RFtildeg_t}), (\ref{RFg_t}),
we derive the following coupled system for $\zeta,\xi$.
\begin{align}\label{pdezixi}
\notag\zeta_t=&\;n\frac{\psi_{ss}}{\psi}\big[\frac{1}{\zeta+1}-(\zeta+1)\big]
+2n\frac{\psi_s}{\psi}\frac{\xi_s}{(\zeta+1)(\xi+1)}
+n\frac{\xi_{ss}}{(\zeta+1)(\xi+1)}
-n\frac{\psi_s}{\psi}\frac{\zeta_s}{(\zeta+1)^2}\\
&\notag-n\frac{\zeta_s\xi_s}{(\zeta+1)^2(\xi+1)}\\
\xi_t=&\;(\frac{\psi_{ss}}{\psi}+(n-1)\frac{\psi_s^2}{\psi^2})\big[\frac{\xi+1}{(\zeta+1)^2}-\xi-1\big]
+\frac{n-1}{\psi^2}(\xi+1-\frac{1}{\xi+1})\\
\notag&+2n\frac{\psi_s}{\psi}\frac{\xi_s}{(\zeta+1)^2}
+\frac{\xi_{ss}}{(\zeta+1)^2}
+(n-1)\frac{\xi_s^2}{(\zeta+1)^2(\xi+1)}
-\frac{\psi_s}{\psi}\frac{\zeta_s(\xi+1)}{(\zeta+1)^3}-\frac{\zeta_s\xi_s}{(\zeta+1)^3}
\end{align}
It is clear that solving for $\tilde{\chi},\tilde{\psi}$ in (\ref{RFtildeg_t})
amounts to solving for $\zeta,\xi$ in the preceding system. Now we would like to
simplify (\ref{pdezixi}) by removing the problematic term
$n\frac{\xi_{ss}}{(\zeta+1)(\xi+1)}$ from the first equation. (This term would
 not allow us to derive energy estimates of any kind for (\ref{pdezixi})).
In order to do so, we replace the variable $\zeta$ by
\begin{align}\label{eta}
\eta:=\frac{(\zeta+1)^2}{(\xi+1)^{2n}}-1.
\end{align}
The new system of $\eta,\xi$ reads
\begin{align}\label{pde}
\notag\eta_t=&\;
-2n(n-1)\bigg(\frac{\psi^2_s}{\psi^2}\big[\frac{1}{(\xi+1)^{2n}}-1\big]
+2\frac{\psi_s}{\psi}\frac{\xi_s}{(\xi+1)^{2n+1}}
+\frac{1-(\xi+1)^{-2}}{\psi^2}
+\frac{\xi^2_s}{(\xi+1)^{2n+2}}\bigg)\\
\notag&-2n(n-1)\frac{1-(\xi+1)^{-2}}{\psi^2}\eta
+2n(n-1)\frac{\psi_s^2}{\psi^2}\eta\\
\xi_t=&\;(\frac{\psi_{ss}}{\psi}
+(n-1)\frac{\psi^2_s}{\psi^2})\big[\frac{1}{(\eta+1)(\xi+1)^{2n-1}}-(\xi+1)\big]
+\frac{n-1}{\psi^2}(\xi+1-\frac{1}{\xi+1})\\
\notag&+n\frac{\psi_s}{\psi}\frac{\xi_s}{(\eta+1)(\xi+1)^{2n}}
+\frac{\xi_{ss}}{(\eta+1)(\xi+1)^{2n}}
-\frac{\xi^2_s}{(\eta+1)(\xi+1)^{2n+1}}\\
\notag&-\frac{1}{2}\frac{\psi_s}{\psi}\frac{\eta_s}{(\eta+1)^2(\xi+1)^{2n-1}}
-\frac{1}{2}\frac{\eta_s\xi_s}{(\eta+1)^2(\xi+1)^{2n}},
\end{align}
\begin{align*}
\eta\bigg|_{t=0}:=\eta_0&&\xi\bigg|_{t=0}:=\xi_0.
\end{align*}

Before attempting to solve the above system, we must first understand its
important features. It is crucial that we know the exact
leading asymptotics of the coefficients in (\ref{pde}), as $x,t\rightarrow0^+$.
Recalling that $s(x,t)=\sqrt{\epsilon(t)}\rho_t(x)$ we derive that
(for fixed $t$)
\begin{align}\label{psi(.,.)}
\notag\psi(s,t)\overset{(\ref{chipsi})}{=}\sqrt{\epsilon(t)}\,\psi(\rho_t(x))
\overset{(\ref{s})}{=}\sqrt{\epsilon(t)}\,\psi(\frac{s}{\sqrt{\epsilon(t)}})&&
\frac{\partial_s\psi(s,t)}{\psi(s,t)}
=\frac{1}{\sqrt{\epsilon(t)}}\frac{\dot{\psi}(\frac{s}{\sqrt{\epsilon(t)}})}{\psi(\frac{s}{\sqrt{\epsilon(t)}})}\\
\frac{\partial_s\partial_s\psi(s,t)}{\psi(s,t)}=
\frac{1}{\epsilon(t)}\frac{\ddot{\psi}(\frac{s}{\sqrt{\epsilon(t)}})}{\psi(\frac{s}{\sqrt{\epsilon(t)}})},
\end{align}
as long as $\epsilon(t)=1+2\lambda t>0$.
Here we recall again that $\psi(\cdot)$ is the function in (\ref{metric}), different from
$\psi(\cdot,\cdot)$ for $t>0$. Recall definition (\ref{s})
and the upper bound (\ref{push}) of $\rho^2_t(x)$.
Going back to (\ref{xasym}), for $x,t$ sufficiently small, we deduce
\begin{align}\label{coeff}
\frac{1}{\psi^2}=\frac{\epsilon(t)^{\frac{1}{\sqrt{n}}-1}}{a^2}\frac{1}{s^{\frac{2}{\sqrt{n}}}}
+O(s^{\frac{2\mu-2}{\sqrt{n}}})&&&
\frac{\psi_s}{\psi}=\frac{1}{\sqrt{n}}\frac{1}{s}+O(s^{\mu-1})\\
\notag\frac{\psi^2_s}{\psi^2}=\frac{1}{n}\frac{1}{s^2}+O(s^{\mu-2})
&&&\frac{\psi_{ss}}{\psi}+(n-1)\frac{\psi^2_s}{\psi^2}=
\frac{n-\sqrt{n}}{n}\frac{1}{s^2}+O(s^{\mu-2}).
\end{align}
Similar estimates hold near $(x,t)=(0,0)$ for the derivatives of the potential
\begin{align}\label{omega}
\omega(s,t)&:=\frac{1}{\sqrt{\epsilon(t)}}\omega(\rho_t(x))\overset{(\ref{s})}{=}
\frac{1}{\sqrt{\epsilon(t)}}\omega(\frac{s}{\sqrt{\epsilon(t)}})=\frac{\sqrt{n}-1}{s}+O(s^{\mu-1})\\
\notag\partial_s\omega(s,t)&\overset{(\ref{s})}{=}
\frac{1}{\sqrt{\epsilon(t)}\partial_x\rho_t(x)}\frac{\partial}{\partial x}
\big[\frac{1}{\sqrt{\epsilon(t)}}\omega(\rho_t(x))\big]=
\frac{1}{\epsilon(t)}\dot{\omega}(\frac{s}{\sqrt{\epsilon(t)}})=-\frac{\sqrt{n}-1}{s^2}+O(s^{\mu-2}).
\end{align}
Moreover, as we will see below, we are going
to need also estimates for the $s$-derivatives of the coefficients in the first
equation of (\ref{pde}), that is,
\begin{align}\label{coeff_s}
\notag\partial_s(\frac{1}{\psi^2})=&
-2\frac{1}{\psi^2}\frac{\psi_s}{\psi}\overset{(\ref{coeff})}{=}
-2\frac{\epsilon(t)^{1-\frac{1}{\sqrt{n}}}}{a^2\sqrt{n}}
\frac{1}{s^{1+\frac{2}{\sqrt{n}}}}+O(s^{b-1-\frac{2}{\sqrt{n}}})\\
\partial_s(\frac{\psi_s}{\psi})=&\frac{\psi_{ss}}{\psi}-\frac{\psi^2_s}{\psi^2}
\overset{(\ref{coeff})}{=}-\frac{1}{\sqrt{n}}\frac{1}{s^2}+O(s^{\mu-2})\\
\notag\partial_s(\frac{\psi^2_s}{\psi^2})=&
2\frac{\psi_s}{\psi}\partial_s(\frac{\psi_s}{\psi})=-\frac{2}{n}\frac{1}{s^3}
+O(s^{\mu-3}),
\end{align}
where $b:=\min\{\frac{2\mu}{\sqrt{n}},\mu\}$.
Further, from (\ref{rhoasym}) and (\ref{s}) we have
\begin{align}\label{s_t}
\partial_ts=\frac{\lambda}{\sqrt{\epsilon(t)}}\rho_t(x)+\sqrt{\epsilon(t)}\partial_t\rho_t(x)
=\frac{\lambda}{\epsilon(t)}s+\frac{\sqrt{n}-1}{s}+O(s^{\mu-1}).
\end{align}
It follows that ($x,t$ are sufficiently small)
\begin{align*}
\frac{\sqrt{n}-1}{2s}\leq\partial_ts\leq\frac{2(\sqrt{n}-1)}{s}
\end{align*}
and hence
\begin{align*}
\sqrt{n}-1\leq2s\partial_ts\leq4(\sqrt{n}-1),
\end{align*}
Thus, integrating with respect to $t$
we derive
\begin{align}\label{leqsleq}
x^2+(\sqrt{n}-1)t\leq s^2\leq x^2+4(\sqrt{n}-1)t.
\end{align}
\begin{remark}\label{intcoeff}
According to Remark \ref{flow}, the lower bound
\begin{align}\label{sgeq}
\inf_{x\in(0,\delta)}s^2:=b(t)\geq(\sqrt{n}-1)t
\end{align}
holds for small $t>0$ and $b(t)$ is non-decreasing.
Hence,we derive
\begin{align}\label{1/s}
\int^t_0\frac{1}{s}d\tau
\leq\int^t_0\frac{1}{\sqrt{(\sqrt{n}-1)\tau}}d\tau\leq2\sqrt{\frac{t}{\sqrt{n}-1}}
<+\infty
\end{align}
and similarly for the leading term of $\frac{1}{\psi^2}$
\begin{align}\label{1/s1+}
\int^t_0\frac{1}{s^\frac{2}{\sqrt{n}}}d\tau<c(t)<+\infty
\end{align}
for all $n>1$, $x\in(0,\delta)$.
However, the time integral on $[0,t]$ of the most singular coefficients (\ref{coeff}) of (\ref{pde})
blows up as $x\rightarrow0^+$. Indeed, by $(\ref{leqsleq})$ it follows that
\begin{align}\label{1/s2}
\lim_{x\rightarrow0^+}\int^t_0\frac{1}{s^2}d\tau\geq
\lim_{x\rightarrow0^+}
\frac{1}{4(\sqrt{n}-1)}\log\frac{x^2+4(\sqrt{n}-1)t}{x^2}
=+\infty.
\end{align}
%
%
%
%
%
\end{remark}
Now we specialize to the {\bf HC} case to derive  an understanding of the coefficients and parameters we are interested in
near $x=+\infty$. First, from
(\ref{partial_trho_tinfty}), (\ref{rho_tinfty}) and (\ref{s}) for $x>M$ large
and $0\leq t\leq T$ small we have
\begin{align}\label{s_tinfty}
-C\leq\partial_ts\leq-c&&c,C>0
\end{align}
and
\begin{align}\label{sinfty}
x-Ct\leq s\leq x-ct,
\end{align}
Thus, by (\ref{psi(.,.)}), additionally to
the above asymptotics (\ref{coeff}),
invoking also the estimate (\ref{asyminfty2})
we obtain
\begin{align}\label{coeffinfty}
\frac{1}{\psi^2}=O(\frac{1}{s})&&
\frac{\psi_s}{\psi}=O(\frac{1}{s})&&
\frac{\psi^2_s}{\psi^2}=(\frac{1}{s^2})
&&\frac{\psi_{ss}}{\psi}+(n-1)\frac{\psi^2_s}{\psi^2}=
O(\frac{1}{s^2})
\end{align}
and analogously to (\ref{coeff_s})
\begin{align}\label{coeff_sinfty}
\partial_s(\frac{1}{\psi^2})=O(\frac{1}{s^2})&&
\partial_s(\frac{\psi_s}{\psi})=O(\frac{1}{s^2})&&
\partial_s(\frac{\psi^2_s}{\psi^2})=O(\frac{1}{s^3}).
\end{align}
Also, corresponding to (\ref{omega}) we have
\begin{align}\label{omegainfty}
-C\leq\omega\leq-c,&&\omega_s=O(\frac{1}{s^2}).
\end{align}

Finally, it will also be useful furtherdown
to note a few properties of the {\it less}
singular coefficients in (\ref{pde}),  namely,
$\frac{1}{\psi^2}$.
By (\ref{coeff}) and (\ref{coeffinfty}) we derive that we can write
\begin{align}\label{A(t)}
\frac{1}{\psi^2}=:\frac{A(s,t)}{s},&&\partial_s(\frac{A(s,t)}{s})=
-2\frac{1}{\psi^2}\frac{\psi_s}{\psi}=\frac{A(s,t)}{s}O(\frac{1}{s}),
\end{align}
where setting
\begin{align}\label{intA(t)}
A(t):=\|A(s,t)\|_{L^\infty(s)}, &&\int^t_0A^2(\tau)d\tau=o(1),\qquad\text{as $t\rightarrow0^+$};
\end{align}
see Remark \ref{intcoeff}.
\newline

As explained earlier, the energy spaces  we will be dealing with will be defined
relative to the background arc-length parameter $s(x,t)=\rho_t(x)$. We note some useful formulas that
will be useful further down in the derivation of our energy estimates.
The first issue  is that due to the coordinate change (\ref{s}),
the vector fields $\partial_s,\partial_t$
(the first considered over $M^{n+1}$, and being the arc-length parameter of the background soliton,
while the second is defined so that $\partial_tx=0$) do not
commute. In fact, we find
\begin{align*}
\tag{see again (\ref{s})}\frac{\partial}{\partial t}\frac{\partial}{\partial s}
=&\frac{\partial}{\partial t}
\bigg(\frac{1}{\sqrt{\epsilon(t)}\,\partial_x\rho_t(x)}\frac{\partial}{\partial x}\bigg)\\=&
-\frac{\lambda}{\epsilon(t)^{\frac{3}{2}}}\frac{1}{\partial_x\rho_t(x)}\frac{\partial}{\partial x}
-\frac{1}{\sqrt{\epsilon(t)}}\frac{\partial_t(\partial_x\rho_t(x))}{(\partial_x\rho_t(x))^2}\frac{\partial}{\partial x}
+\frac{1}{\sqrt{\epsilon(t)}\,\partial_x\rho_t(x)}\frac{\partial}{\partial t}\frac{\partial}{\partial x}\\
=&-\frac{\lambda}{\epsilon(t)}\frac{\partial}{\partial s}
-\frac{\partial_x\partial_t\rho_t(x)}{\partial_x\rho_t(x)}\frac{\partial}{\partial s}
+\frac{1}{\sqrt{\epsilon(t)}\,\partial_x\rho_t(x)}\frac{\partial}{\partial x}\frac{\partial}{\partial t}\\
\tag{using (\ref{partial_trho_t})}=&-\frac{\lambda}{\epsilon(t)}\frac{\partial}{\partial s}
-\frac{\partial_x\big[\frac{1}{\epsilon(t)}\omega(\rho_t(x))\big]}{\partial_x\rho_t(x)}\frac{\partial}{\partial s}
+\frac{\partial}{\partial s}\frac{\partial}{\partial t}\\
=&-\frac{\lambda+\dot{\omega}(\rho_t(x))}{\epsilon(t)}\frac{\partial}{\partial s}
+\frac{\partial}{\partial s}\frac{\partial}{\partial t}\\
\tag{from the first equation of (\ref{ode})}=&
-\frac{1}{\epsilon(t)}\cdot n\frac{\ddot{\psi}}{\psi}\bigg|_{\rho_t(x)}\frac{\partial}{\partial s}
+\frac{\partial}{\partial s}\frac{\partial}{\partial t}
\end{align*}
Thus, using (\ref{xasym}) and (\ref{asyminfty}) we derive
\begin{align}\label{[s,t]}
\notag[\partial_t,\partial_s]&\overset{(\ref{s})}{=}
-\frac{1}{\epsilon(t)}\cdot n\frac{\ddot{\psi}}{\psi}\bigg|_{\frac{s}{\sqrt{\epsilon(t)}}}\partial_s\\
&=\left\{\begin{array}{ll}
[(\sqrt{n}-1)s^{-2}+O(s^{\mu-2})]\partial_s,&x,t\ll1,\;{\bf G}\\
O(s^{-2})\partial_s, &(x,t)\in(0,+\infty)\times[0,T],\;{\bf HC}
\end{array}\right.
\end{align}
Since the Sobolev spaces we will be considering will be with respect to the arc-length 1-form $ds$
we must also calculate the evolution of this form.
\begin{align*}
\tag{recall (\ref{s})}\partial_tds=&\;\partial_t(\sqrt{\epsilon(t)}\,\partial_x\rho_t(x)dx)
=\frac{\lambda}{\sqrt{\epsilon(t)}}\partial_x\rho_t(x)dx
+\sqrt{\epsilon(t)}\,\partial_t\partial_x\rho_t(x)dx\\
\tag{using (\ref{partial_trho_t})}=&\;\frac{\lambda}{\epsilon(t)}ds
+\sqrt{\epsilon(t)}\,\partial_x\big[\frac{1}{\epsilon(t)}\omega(\rho_t(x))\big]dx\\
=&\;\frac{\lambda+\dot{\omega}(\rho_t(x))}{\epsilon(t)}ds
\tag{see the first equation of (\ref{ode})}=\frac{1}{\epsilon(t)}\cdot n\frac{\ddot{\psi}}{\psi}\bigg|_{\rho_t(x)}ds
\end{align*}
Employing (\ref{xasym}) and (\ref{asyminfty}) once more we obtain ($t\ge0$ small)
\begin{align}\label{partial_tds}
\notag\partial_tds=&\frac{1}{\epsilon(t)}\cdot n\frac{\ddot{\psi}}{\psi}\bigg|_{\frac{s}{\sqrt{\epsilon(t)}}}ds\\
=&\left\{\begin{array}{ll}
[-(\sqrt{n}-1)s^{-2}+O(s^{\mu-2})]\partial_s,&x,t\ll1,\;{\bf G}\\
O(s^{-2})\partial_s, &(x,t)\in(0,+\infty)\times[0,T],\;{\bf HC}
\end{array}\right.
\end{align}

\subsection{The weighted Sobolev spaces and the main result}\label{wSsp}
\noindent

As explained the singularities of the coefficients, at $(x=0,t=0)$, in
the system (\ref{pde}), along with the asymptotic expansions we have derived
force us to study well-posedness in {\it weighted} Sobolev spaces.
The weights will be adapted to the singularity at $x=0,t=0$.
To construct these, we define:
\begin{definition}\label{weight}
Let
\begin{align}\label{x_0}
x_0:=\left\{\begin{array}{ll}
\delta,\;{\bf G}\\
1,\;{\bf HC}
\end{array}\right.
\end{align}
Given $\sigma>0$ (to be determined later) and $T>0$ small we define
\begin{align}\label{l}
\ell^2=\left\{\begin{array}{lll}
s^2+\sigma t,&(x,t)\in(0,x_0)\times[0,T],\;{\bf G}\\
\varphi(s,t),&(x,t)\in[x_0,x_0+1)\times[0,T],\;{\bf HC}\\
1,&(x,t)\in[x_0+1,+\infty)\times[0,T],\;{\bf HC}
\end{array}\right.
\end{align}
where $\varphi(\cdot,t)$ is smooth cut off function interpolating between
$\displaystyle\lim_{x\rightarrow x_0^-}\ell^2(x,t)$ and $1$,
for each $t\in[0,T]$, having bounded derivatives.
\end{definition}
We recall at this point that for the general solitons ({\bf G}) we consider
manifolds in the form
\begin{align}\label{locx}
(0,\delta)\times\mathbb{S}^n,
\end{align}
whereas in the {\bf HC} case we study the stability of the whole
half-complete gradient steady Ricci solitons presented in \S\ref{steady}; see again Propositions
\ref{incgRS} and \ref{steadysol}. Further,
we shall investigate the evolution of the Ricci flow (\ref{RFg_t}), (\ref{RFtildeg_t})
only for positive time $t\in[0,T]$ such that
\begin{align}\label{loct}
\epsilon(t):=1+2\lambda t>\frac{1}{2}.
\end{align}
A note is in order here: Since the Sobolev spaces we will be considering
are relative to the length element $ds$, it is worth
recalling the behaviour of $s(x,t):=\rho_t(x)$ at the endpoints $x=0, x=+\infty$
for {\bf HC} and $x=0, x=\delta$ for {\bf G}. We will always consider the values
of the arc length parameter $s$ at these points to be the endpoints of integration at each time $t\ge0$.
\begin{definition}\label{endpts}
Let
\begin{align}\label{ints1}
s_{min}(t):=\lim_{x\rightarrow0^+}s(x,t)&&s_0(t):=s(x_0,t)&&{\bf G}
\end{align}
and
\begin{align}\label{ints2}
s_{max}(t):=\lim_{x\rightarrow+\infty}s(x,t)\overset{(\ref{sinfty})}{=}+\infty
&&{\bf HC},
\end{align}
for each $t\geq0$.
\end{definition}
\begin{definition}\label{wSob}
For any given $t\in [0,T]$ we consider the parameter $s=\rho_t(x)$ and
let $H^k_\alpha(s)$, $\alpha\geq1$,\footnote{We are suppressing the parameter $t$  in this notation.}
be the Hilbert space of all functions
\begin{align*}
u\in H^k(s)
\end{align*}
endowed with the norm
\begin{align*}
\|u\|_{H^k_\alpha(s)}^2=\int \frac{u^2}{\ell^{2\alpha}}+\dots
+\frac{(\partial^k_su)^2}{\ell^{2\alpha-2k}} ds<+\infty.
\end{align*}
In the special case $k=0$, we denote $H^0_\alpha(s)$ by $L^2_{\alpha}(s)$. We also let $H^1_{\alpha,0}(s)$
be the closure, in $H^1_\alpha(s)$, of all
compactly supported smooth functions in $(0,+\infty)$ $\big[(0,\delta)$, {\bf G}\big].
As usual, we define the spaces
\begin{align*}
L^2(0,T;H^k_\alpha(s))&&L^\infty(0,T;H^k_\alpha(s))
\end{align*}
of measurable functions
\begin{align*}
u:[0,T]\to H^k(s)
\end{align*}
having finite norms
\begin{align*}
\|u\|^2_{L^2(0,T;H^k_\alpha(s))}:=\int^T_0\|u\|^2_{H^k_\alpha(s)}dt,
&&\|u\|_{L^\infty(0,T;H^k_\alpha(s))}:=\operatornamewithlimits{ess\,sup}_{t\in[0,T]}\|u\|_{H^k_\alpha(s)}
\end{align*}
respectively.
Often, when the context is clear, we will suppress the parameter $s$ in the notation.
\end{definition}
We are going to need the following properties of the weight $\ell$.
By (\ref{l}), for
fixed time $t\in[0,T]$, we have
\begin{align}\label{l_s}
\partial_s\ell=\frac{s}{\ell}\textbf{1}_{(0,x_0)}
+\partial_s\varphi_1\textbf{1}_{[x_0,x_0+1)}:=
\left\{\begin{array}{lll}
\frac{s}{\ell},&x\in(0,x_0)\\
O(1),&x\in[x_0,x_0+1),\;{\bf HC}\\
0,&[x_0+1,+\infty),\;{\bf HC}
\end{array}\right.
\end{align}
Similarly, for fixed $x$, employing (\ref{s_t}) we derive
\begin{align}\label{l_t}
\notag\partial_t\ell=&\;\frac{\partial_ts^2+2\sigma}{2\ell}\textbf{1}_{(0,x_0)}
+\partial_t\varphi_1\textbf{1}_{[x_0,x_0+1)}
=\frac{2s\partial_ts+2\sigma}{2\ell}\textbf{1}_{(0,x_0)}+O(1)\textbf{1}_{[x_0,x_0+1)}\\
=&\;\big[\frac{O(1)}{\ell}
+\frac{\sigma}{\ell}\big]\textbf{1}_{(0,x_0)}
+O(1)\textbf{1}_{[x_0,x_0+1)}
\end{align}
Also, combining (\ref{sgeq}), (\ref{sinfty}) we obtain the following comparison
estimate of $s,\ell$
\begin{align}\label{l/s}
0<c\leq\frac{\ell^2}{s^2}=
\left\{\begin{array}{ll}
1+\frac{2\sigma t}{s^2}\\
\frac{O(1)}{s^2}
\end{array}\right.
\leq
\left\{\begin{array}{ll}
1+\frac{2\sigma}{\sqrt{n}-1},&x\in(0,x_0)\\
C, &x\in[x_0,+\infty),\;{\bf HC}
\end{array}\right.
&&(n>1)
\end{align}
for each $0\leq t\leq T$.

Now we can proceed to the
well-posedness problem for (\ref{pde}). First, we
introduce the function spaces in which we will be constructing the solutions to
 our problem and then state the main theorem.
\begin{definition}\label{sol}
Assume initially ($t=0$)
\begin{align*}
\eta_0\in H^1_{\alpha}&&\xi_0\in H^1_{\alpha,0},
\end{align*}
Then, two functions $\eta,\xi$ defined over $M^{n+1}\times  [0,T]$
are a (strong)  solution of (\ref{pde}), if
\begin{align*}
\eta\in L^\infty(0,T;H^1_\alpha)\cap L^2(0,T;H^1_{\alpha+1})&&
\xi\in L^\infty(0,T;H^1_{\alpha,0})\cap L^2(0,T;H^2_{\alpha+1})\\
\eta_t\in L^2(0,T;L^2_\alpha)&&\xi_t\in L^2(0,T;L^2_{\alpha-1})
\end{align*}
and moreover satisfy the system (\ref{pde}) in the usual sense of test functions.
\end{definition}
We remark here the fact that once we have such a solution to (\ref{pde}), then we
straightforwardly derive that this solution $(\eta,\xi)$ corresponds to a
solution of (\ref{RFtildeg_t}), which in fact will be smooth over $M^{n+1}\times (0,T]$,
given the parabolicity of the Ricci flow.
\newline

We define the energy
\begin{align}\label{energy}
\mathcal{E}(u,v;T)=\|u\|^2_{L^\infty(0,T;H^1_\alpha)}+
\|u\|^2_{L^2(0,T;H^1_{\alpha+1})}+
\|v\|^2_{L^\infty(0,T;H^1_\alpha)}+
\|v\|^2_{L^2(0,T;H^2_{\alpha+1})}
\end{align}
and for brevity let
\begin{align}\label{inen}
\mathcal{E}_0=\|\eta_0\|^2_{H^1_\alpha}+
\|\xi_0\|^2_{H^1_\alpha}.
\end{align}
\begin{theorem}\label{gensol}
There exist $\alpha,\sigma$ appropriately large and $\mathcal{E}_0,T>0$ sufficiently
small, such that (\ref{pde}) has a unique
solution up to time $T$
\begin{align}\label{gensp}
\notag\eta\in L^\infty(0,T;H^1_\alpha)\cap L^2(0,T;H^1_{\alpha+1})&&
\xi\in L^\infty(0,T;H^1_{\alpha,0})\cap L^2(0,T;H^2_{\alpha+1})\\
\eta_t\in L^2(0,T;L^2_\alpha)&&\xi_t\in L^2(0,T;L^2_{\alpha-1})\\
\notag\eta\bigg|_{t=0}=\eta_0&&
\xi\bigg|_{t=0}=\xi_0
\end{align}
subject to the Dirichlet boundary condition
\begin{align}\label{genbdc}
\xi\bigg|_{(s_{min},t)}=0&&\big(\text{and}\;\;\xi\bigg|_{(s_0,t)}=0\;\;\text{in case {\bf G}}\big),
\end{align}
for each $t\in[0,T]$. Furthermore the solution satisfies
\begin{align}\label{genest}
\mathcal{E}(\eta,\xi;T)\leq2\widetilde{C}\mathcal{E}_0.
\end{align}
\end{theorem}
\begin{remark}
In fact it is easy to observe that in the half-complete case  the above implies that for any $T>0$ and $\alpha,\sigma>0$ 
large enough, the initial energy $\mathcal{E}_0$ can be picked sufficiently small 
to have existence up to time $T>0$. On the other hand,  the system (\ref{pde}) 
does not seem to be globally stable (i.e. up to $T=\infty$) in the half-complete case, 
 due to the behavior of the coefficients at $x=\infty$. 

\end{remark}
\section{The Contraction Mapping}\label{genprob}
\noindent

We will prove Theorem \ref{gensol} via an iteration  scheme,
which is essentially a contraction mapping argument;
see (\ref{contr}) below.
We note that throughout the subsequent estimates we will use the symbol $C$ to denote
a positive constant depending only on $n$.

\subsection{The iteration scheme and some basic estimates}\label{contrit}
\noindent

We will construct a sequence $\big\{\eta^m,\xi^m\big\}^\infty_{m=0}$ satisfying
\begin{align}\label{itpde}
\notag\eta^{m+1}_t=&\;
2n(n-1)\bigg(\frac{\psi^2_s}{\psi^2}\frac{2n\xi^{m+1}
+\sum_{j=2}^{2n}\binom{2n}{j}|\xi^m|^j}{(\xi^m+1)^{2n}}
-2\frac{\psi_s}{\psi}\frac{\xi^{m+1}_s}{(\xi^m+1)^{2n+1}}\\
&-\frac{A(s,t)}{s}\xi^m\frac{\xi^m+2}{(\xi^m+1)^2}
\notag-\frac{|\xi^m_s|^2}{(\xi^m+1)^{2n+2}}
-\frac{A(s,t)}{s}\xi^m\frac{\xi^m+2}{(1+\xi^m)^2}\eta^{m+1}
+\frac{\psi_s^2}{\psi^2}\eta^{m+1}\bigg)\\
\xi^{m+1}_t=&\;(\frac{\psi_{ss}}{\psi}+(n-1)\frac{\psi^2_s}{\psi^2})
\bigg[\frac{-\eta^{m+1}-2n(\eta^m+1)\xi^{m+1}}{(\eta^m+1)(\xi^m+1)^{2n-1}}
-\frac{\sum_{j=2}^{2n}\binom{2n}{j}|\xi^m|^j}{(\xi^m+1)^{2n-1}}\bigg]\\
&\notag+(n-1)\frac{A(s,t)}{s}\xi^m\frac{\xi^m+2}{\xi^m+1}
+n\frac{\psi_s}{\psi}\frac{\xi^{m+1}_s}{(\eta^m+1)(\xi^m+1)^{2n}}
+\frac{\xi^{m+1}_{ss}}{(\eta^m+1)(\xi^m+1)^{2n}}\\
&-\frac{|\xi^m_s|^2}{(\eta^m+1)(\xi^m+1)^{2n+1}}
\notag-\frac{1}{2}\frac{\psi_s}{\psi}\frac{\eta^{m+1}_s}{(\eta^m+1)^2(\xi^m+1)^{2n-1}}
-\frac{1}{2}\frac{\eta^m_s\xi^m_s}{(\eta^m+1)^2(\xi^m+1)^{2n}},
\end{align}
where we set $\eta^0=\xi^0=0$ and initially
\begin{align}\label{init}
\eta^{m+1}\bigg|_{t=0}=\eta_0\qquad
\xi^{m+1}\bigg|_{t=0}=\xi_0&&m=0,1,\ldots
\end{align}
Further, $ \xi^{m+1}$ is required to satisfy
the Dirichlet boundary condition
\begin{align}\label{itbdc}
\xi^{m+1}\bigg|_{(s_{min},t)}=0&&\big(\text{and}\;\;\xi^{m+1}\bigg|_{(s_0,t)}=0,\;\;\text{in case {\bf G}}\big)
\end{align}
for all $m\in\mathbb{N}$, $t\ge0$; recall (\ref{ints1}), (\ref{ints2}).
Our main claim from this point onwards will
be to show that the iterates converge to a solution of (\ref{pde}). Our method of proof
will readily imply that any solution of (\ref{pde}) in the sense of Definition \ref{sol} will be unique.
\newline

\par Note that (\ref{itpde}) is linear at each step $m+1$, {\it yet the lower-order terms in the
RHSs associated to the most singular coefficients involve the unknown functions $\eta^{m+1},\xi^{m+1}$}.
In view of the asymptotic behavior (\ref{coeff}) of these coefficients, as $x,t\rightarrow0^+$,
and the non-integrability (in time) of  $\|\frac{1}{s^2}\|_{L^\infty(s)}$ (Remark \ref{intcoeff}),
one {\it could not} hope to close  energy estimates for the above system, if one had not set up an iteration
of this kind.

\par Our irst claim is:
\begin{proposition}\label{itsol}
There exist $\alpha,\sigma$ appropriately large and $\mathcal{E}_0,T>0$ sufficiently
small, such that (\ref{itpde}) has a unique solution at each step $m+1\in\{1,2,\ldots\}$,
in the sense of Definition \ref{sol}. In particular:
\begin{align}\label{itsp}
\eta^{m+1}\in L^\infty(0,T;H^1_\alpha)\cap L^2(0,T;H^1_{\alpha+1}),\;
\xi^{m+1}\in L^\infty(0,T;H^1_{\alpha,0})\cap L^2(0,T;H^2_{\alpha+1})\\
\notag\eta^{m+1}_t\in L^2(0,T;L^2_\alpha)\qquad\qquad\qquad\qquad\;\xi^{m+1}_t\in L^2(0,T;L^2_{\alpha-1})
\end{align}
and
\begin{align}\label{itsmooth}
\eta^{m+1},\xi^{m+1}\in
C^\infty\big((0,+\infty)\times[0,T]\big)&&\bigg[\text{or}\;\;C^\infty\big((0,\delta)\times[0,T]\big),\;\;\text{for}\;\;{\bf G}\bigg],
\end{align}
$m\in\mathbb{N}$, $\eta^0=\xi^0=0$, subject to (\ref{init}), (\ref{itbdc}). Moreover,
the sequence $\big\{\eta^m,\xi^m\big\}^\infty_{m=0}$
has uniformly bounded energy
\begin{align}\label{itenest}
\mathcal{E}(\eta^m,\xi^m;T)\leq2\widetilde{C}\mathcal{E}_0&&m=0,1,\ldots
\end{align}
for some constant $\widetilde{C}>0$.
\end{proposition}
We postpone the proof
of Proposition \ref{itsol} for the moment to see how we can get from this point
to the solution of the non-linear system (\ref{pde}).
(Actually, we will prove Proposition \ref{itsol} in the next subsection by
solving  more general linear systems of this type).  It is useful at
this point to note the $L^\infty$ bounds that the functions $\eta^m\xi^m$ belonging to
our weighted Sobolev spaces satisfy.
From this point onwards $\eta^m,\xi^m$, $m=1,2,\dots,$ will be solutions of (\ref{itpde})
(replacing $m+1$ by $m$) which lie in the energy spaces (\ref{itsp}).

\begin{lemma}\label{lempm}
The iterates $\eta^m,\xi^m,\xi^m_s$, $n\in\mathbb{N}$,
obey the following pointwise bounds:
\begin{align}\label{pm}
\|\frac{\eta^m}{\ell^k}\|_{L^\infty(s,t)}^2\leq C\widetilde{C}(k+1)\mathcal{E}_0&&
\|\frac{\xi^m}{\ell^k}\|_{L^\infty(s,t)}^2\leq C\widetilde{C}(k+1)\mathcal{E}_0
\end{align}
and
\begin{align}\label{pm_s}
\|\frac{\xi^m_s}{\ell^k}\|_{L^\infty(s)}^2\leq
C\sqrt{\widetilde{C}\mathcal{E}_0}\big(\|\frac{\xi^m_{ss}}{\ell^{\alpha-1}}\|_{L^2}
+k\|\frac{\xi^m_s}{\ell^\alpha}\|_{L^2}\big),
\end{align}
for every $k=0,\ldots\alpha-1$, $\alpha\ge1$, and a.e. $t\in[0,T]$.
\end{lemma}
\begin{proof}
For each of $\eta^m,\xi^m$, $m\in\mathbb{N}$, the energy estimate
(\ref{itenest}) furnishes a point $s_\star(t)$, such that
\begin{align*}
|\frac{\eta^m}{\ell^\alpha}(s_\star,t)|^2
+|\frac{\xi^m}{\ell^\alpha}(s_\star,t)|^2
\leq M\mathcal{E}_0,
\end{align*}
for some $M>0$ and a.e. $0\leq t\leq T$. Thus, for $\eta^m$ and a.e. $t$ we have
\begin{align*}
\big||\frac{\eta^m}{\ell^k}(s,t)|^2-|\frac{\eta^m}{\ell^k}(s_\star,t)|^2\big|=
\bigg|\int^s_{s_\star}\partial_s|\frac{\eta^m}{\ell^k}|^2ds\bigg|,
\end{align*}
thus
\begin{align*}
\tag{recall (\ref{l_s})}|\frac{\eta^m}{\ell^k}(s,t)|^2\leq&\;
\bigg|\int^s_{s_\star}2\frac{\eta^m}{\ell^k}\big(\frac{\eta^m_s}{\ell^k}-k\frac{\eta^m}{\ell^{k+1}}\partial_s\ell\big)ds\bigg|
+|\frac{\eta^m}{\ell^k}(s_\star,t)|^2\\
\leq&\;C\|\frac{\eta^m}{\ell^k}\|_{L^2}\big(\|\frac{\eta^m_s}{\ell^k}\|_{L^2}
+k\|\frac{\eta^m}{\ell^{k+1}}\|_{L^2}\big)+2M\mathcal{E}_0
\end{align*}
Employing (\ref{itenest}) once more and similarly for $\xi^m$
we obtain (\ref{pm}), provided $k\leq\alpha-1$

We treat $\xi^m_s$ in a different manner for a.e. $t\in[0,T]$, $m\in\mathbb{N}$.
By assumption (\ref{itbdc})
\begin{align}\label{smaxdecay}
\xi^m(s_{min},t)=0,\;{\bf HC}&& \big[\text{$\xi^m(s_0,t)=0$,\;{\bf G}}\big].
\end{align}
Thus, there exists a point $s_\star(t)\in(s_{min},s_{max}]$ such that $\xi^m_s(s_\star,t)=0$.
It follows that
\begin{align*}
|\frac{\xi^m_s}{\ell^k}(s,t)|^2=&
\bigg|\int^s_{s_\star}\partial_s|\frac{\xi^m_s}{\ell^k}|^2ds\bigg|
=\bigg|\int^s_{s_\star}2\frac{\xi^m_s}{\ell^k}
\big(\frac{\xi^m_{ss}}{\ell^k}-k\frac{\xi^m_s}{\ell^{k+1}}\partial_s\ell ds\bigg|\\
\leq&\;C\|\frac{\xi^m_s}{\ell^k}\|_{L^2}
\big(\|\frac{\xi^m_{ss}}{\ell^k}\|_{L^2}
+(k-1)\|\frac{\xi^m_s}{\ell^{k+1}}\|_{L^2}\big)\\
\tag{by (\ref{itenest})}
\leq&\;C\sqrt{\widetilde{C}\mathcal{E}_0}\big(\|\frac{\xi^m_{ss}}{\ell^{\alpha-1}}\|_{L^2}
+(k-1)\|\frac{\xi^m_s}{\ell^\alpha}\|_{L^2}\big),
\end{align*}
for every $k=0,\ldots,\alpha-1$.
\end{proof}
{\bf The iteration (\ref{itpde}) yields a contraction mapping:}

\begin{definition}\label{B,D}
We introduce generic notation
\begin{align*}
B,D
\end{align*}
to denote rational functions in $\eta^m,\xi^m$, $m=0,1,\ldots$,
satisfying the following conditions:
\begin{itemize}
\item The denomerators of $B,D$ have non-zero constant terms.
\item The constant term in the numerator of $B$ is non-zero,
whereas the one in the numerator of $D$ vanishes.
\end{itemize}
\end{definition}
\begin{lemma}\label{estB,D}
If $B,D$ are functions as above and
$\mathcal{E}_0$ is sufficiently small, then the following estimates
hold:
\begin{align}\label{pB,D}
0<c<\|B(s,t)\|_{L^\infty(s)}<C&&\|\frac{D}{\ell^k}\|^2_{L^\infty(s)}\leq C\widetilde{C}\mathcal{E}_0,
\end{align}
where $k=0,\ldots,\alpha-1$ and
\begin{align}\label{B_s,D_s}
\|\frac{B_s}{\ell^{\alpha-1}}\|^2_{L^2(s)}
+\|\frac{D_s}{\ell^{\alpha-1}}\|^2_{L^2(s)}\leq C\widetilde{C}\mathcal{E}_0,
\end{align}
for a.e. $0\leq t\leq T$ and appropriate positive constants $c,C$;
these depend on the coefficients of the rational functions $B,D$.
\end{lemma}
We remark that we shall make use of the lower bound of $|B|$ only to confirm the parabolicity condition
(see below the second equation of (\ref{dpde})).
\begin{proof}
The proof is immediate from the pointwise estimates
(\ref{pm}) and the energy estimate (\ref{itenest}).
\end{proof}
Let
\begin{align}\label{deta,dxi}
d\eta^{m+1}=\eta^{m+1}-\eta^m,\;\;d\xi^{m+1}=\xi^{m+1}-\xi^m&&m=0,1,\ldots
\end{align}
Consider now the two  systems (\ref{itpde})
corresponding to the steps $m+1$ and $m$. Subtracting
these two systems, it is straightforward to check
that we arrive at a system:
\begin{align}\label{dpde}
\notag d\eta^{m+1}_t=&\;\frac{\psi_s^2}{\psi^2}Bd\xi^{m+1}
+\frac{\psi_s^2}{\psi^2}Dd\xi^m
+\frac{A(s,t)}{s}Bd\xi^m+\frac{\psi_s}{\psi}Bd\xi^{m+1}_s
+\frac{\psi_s}{\psi}B\xi^m_sd\xi^m\\
&+Bd\xi^m_s(\xi^m_s+\xi^{m-1}_s)
\notag+|\xi^{m-1}_s|^2d\xi^mB
+\frac{A(s,t)}{s}Dd\eta^{m+1}
+2n(n-1)\frac{\psi_s^2}{\psi^2}d\eta^{m+1}\\
d\xi^{m+1}_t=&\;(\frac{\psi_{ss}}{\psi}+(n-1)\frac{\psi_s^2}{\psi^2})
\big[Bd\eta^{m+1}+Bd\xi^{m+1}+Dd\xi^m+Dd\eta^m\big]\\
\notag&+\frac{A(s,t)}{s}Bd\xi^m+\frac{\psi_{s}}{\psi}Bd\xi^{m+1}_s
+\frac{\psi_{s}}{\psi}B\xi^m_s(d\eta^m+d\xi^m)
+|B|d\xi^{m+1}_{ss}\\
&\notag+\xi^m_{ss}(d\eta^m+d\xi^m)
+Bd\xi^m_s(\xi^m_s+\xi^{m-1}_s)
+|\xi^{m-1}_s|^2B(d\eta^m+d\xi^m)
+\frac{\psi_s}{\psi}Bd\eta^{m+1}_s\\
&+\frac{\psi_s}{\psi}\eta^m_sB(d\eta^m+d\xi^m)
\notag+B(\xi^m_sd\eta^m_s+\eta^{m-1}_sd\xi^m_s)
+\eta^{m-1}_s\xi^{m-1}_sB(d\eta^m+d\xi^m),
\end{align}
where we note that each $B,D$ appearing in the preceding system may differ from the other.
Also, by $|B|$ in the second equation of (\ref{dpde}) above
we denote a positive function in the $B$ class, having an $L^\infty(s)$ lower bound
$c=\frac{1}{2}$; one can readily check this from (\ref{itpde}), (\ref{itenest}),
provided the initial energy $\mathcal{E}_0$ is sufficiently small.
\newline

The main goal of this section is to show that the sequence of solutions $d\eta^{m+1}, d\xi^{m+1}$ to (\ref{dpde})
yield a Cauchy sequence in the energy spaces we have introduced. Specifically:

\begin{proposition}\label{contrmap}
There exist appropriately large parameters $\alpha,\sigma$ such that for sufficiently small $\mathcal{E}_0,T$
the following estimate holds (for a.e. $t\in[0,T]$):
\begin{align*}
&\big(\|d\eta^{m+1}\|^2_{H^1_\alpha}+\|d\xi^{m+1}\|^2_{H^1_\alpha}\big)(t)
+\frac{1}{2}\alpha\sigma\int^t_0\big(\|d\eta^{m+1}\|^2_{H^1_{\alpha+1}}+\|d\xi^{m+1}\|^2_{H^1_{\alpha+1}}\big)d\tau
+\frac{1}{2}\int^t_0\|\frac{d\xi^{m+1}_{ss}}{\ell^{\alpha-1}}\|^2_{L^2}d\tau\\
\leq&\int^t_0B_1(\tau)\big(\|d\eta^{m+1}\|^2_{H^1_\alpha}+\|d\xi^{m+1}\|^2_{H^1_\alpha}\big)d\tau
+\int^t_0B_2(\tau)\sqrt{\|d\eta^{m+1}\|^2_{H^1_\alpha}+\|d\xi^{m+1}\|^2_{H^1_\alpha}}d\tau\\
&+C\bigg(\sigma\int^T_0A^2(\tau)d\tau+\widetilde{C}\mathcal{E}_0\bigg)\mathcal{E}(d\eta^m,d\xi^m;T)
+C\sigma(\widetilde{C}\mathcal{E}_0)^2\sqrt{T}\mathcal{E}(d\eta^{m-1},d\xi^{m-1};T),
\end{align*}
where the functions $B_1(t),B_2(t)$, $t\in[0,T]$ are integrable, and $A^2(\tau)$ as defined in (\ref{intA(t)}).
\end{proposition}
In fact, we will prove part of this  estimate in \S\ref{L2detadxi}
(see (\ref{sumdetadxi2})) building up for the second half
in the end of \S\ref{L2deta_sdxi_s} (\ref{sumdeta_sdxi_s4}). We will show in the end of
this section that this proposition implies our claim, via  a standard Gronwall inequality.
\newline

Our goal for this and the next two subsections is to prove Proposition \ref{contrmap}.
Prior to embarking on the proof, we outline the main strategy
to deal with the singularities of the coefficients in the equations (\ref{dpde}).
The main issue is that while
we start with the energies $\|\cdot\|_{H^1_{\alpha}}$
for $d\xi^m,d\eta^m$, we are forced to control their energies in the norms
$\|\cdot\|_{H^1_{\alpha+1}}$ which involve  {\it more singular} weights; in particular
the new weights contain an extra factor of $\frac{1}{\ell^2}$, for each of the spaces $H^k_\alpha$  involved.
These appear in the LHS of the estimate in Proposition \ref{contrmap}.
All energies $\|\cdot\|_{H^1_{\alpha+1}}, \|\cdot\|_{L^2_{\alpha+1}}$ will be called {\it critical energies}, and the weights
$\frac{1}{\ell^{\alpha+1}}$ or $\frac{1}{\ell^\alpha}$ multiplying the functions
$d\xi^{K}, d\eta^{K}$, or $d\xi^{K}_s, d\eta^{K}_s$ {\it critical weights}.\footnote{$K=m+1,m$ or $m-1$ below.}
Thus, one theme of our analysis in this section
will involve splitting the {\it extra} powers of $1/s$ among the various terms in our
estimates in order to control the resulting critical energies as required by
Proposition \ref{contrmap}.\footnote{The estimate (\ref{l/s}) gives us bounds for $\frac{\ell}{s}$.}
\newline

We are going to need pointwise estimates for the differences.
\begin{lemma}\label{lemdpm}
For every $m\in\mathbb{N}$ and a.e. $t\in[0,T]$ the following $L^\infty$ estimates hold:
\begin{align}\label{pdxi}
\|\frac{d\xi^m}{\ell^k}\|_{L^\infty(s)}^2
\leq C\big(\|\frac{d\xi^m_s}{\ell^k}\|^2_{L^2}
+(k+1)\|\frac{d\xi^m}{\ell^{k+1}}\|^2_{L^2}\big),
\end{align}
\begin{align}\label{pdxi_s}
\|\frac{d\xi^m_s}{\ell^k}\|^2_{L^\infty(s)}\leq
C\big(\|\frac{d\xi^m_{ss}}{\ell^{\alpha-1}}\|^2_{L^2}
+(k+1)\|\frac{d\xi^m_s}{\ell^\alpha}\|^2_{L^2}\big)
\end{align}
and
\begin{align}\label{pdeta}
\|\frac{d\eta^m}{\ell^k}\|^2_{L^\infty(s)}\leq&\;
C\big(\|\frac{d\eta^m_s}{\ell^{\alpha-1}}\|^2_{L^2}+(k+1)\|\frac{d\eta^m}{\ell^\alpha}\|^2_{L^2}\big)\\
\notag&+CT\int^t_0\|d\xi^m\|^2_{H^2_{\alpha+1}}d\tau
+C\widetilde{C}\mathcal{E}_0\sqrt{T}\int^T_0\|d\xi^{m-1}\|^2_{H^2_{\alpha+1}}d\tau,
\end{align}
$k=0,\ldots,\alpha-1$. Further, in the {\bf HC} case we have the following decay at infinity
(for a.e. $t\in[0,T]$)
\begin{align}\label{smaxdecay}
\lim_{s\rightarrow+\infty}d\xi^m=0&&\lim_{s\rightarrow+\infty}d\xi^m_s=0,&&
m=0,1,\ldots
\end{align}
\end{lemma}
\begin{proof}
We recall the boundary condition (\ref{itbdc}) and (\ref{l_s}) to obtain
\begin{align*}
|\frac{d\xi^m}{\ell^k}(s,t)|^2=&
\int^s_{s_{min}}\partial_s|\frac{d\xi^m}{\ell^k}|^2ds
=\int^s_{s_{min}}2\frac{d\xi^m}{\ell^k}\big(\frac{d\xi^m_s}{\ell^k}
-k\frac{d\xi^m}{\ell^{k+1}}\partial_s\ell\big) ds\\
\leq&\; C\big(\|\frac{d\xi^m_s}{\ell^k}\|^2_{L^2}
+(k+1)\|\frac{d\xi^m}{\ell^{k+1}}\|^2_{L^2}\big)
\end{align*}
for all $m\in\mathbb{N}$ and a.e. $0\leq t\leq T$.
We treat $d\xi^m_s$ in a similar manner to (\ref{pm_s}). By
(\ref{itbdc}) $d\xi^m(s_{min},t)=0$ \big[$d\xi^m(s_0,t)=0$, {\bf G}\big]. The latter furnishes a point
$s_\star(t)\in(s_{min},s_{max}]$ such that $d\xi^m_s(s_\star,t)=0$, which implies as above
\begin{align*}
|\frac{d\xi^m_s}{\ell^k}(s,t)|^2\leq
C\big(\|\frac{d\xi^m_{ss}}{\ell^{\alpha-1}}\|^2_{L^2}
+(k+1)\|\frac{d\xi^m_s}{\ell^\alpha}\|^2_{L^2}\big),
\end{align*}
provided $k\leq\alpha$.
In order to obtain pointwise estimates for $d\eta^m$
we use the equation of $d\eta^m_t$, analogous to that of $d\eta^{m+1}_t$ (\ref{dpde}).
Setting $s=s_0:=s(x_0,t)$ we have
\begin{align*}
d\eta^m_t&-\big[\frac{A(s_0,t)}{s_0}D+2n(n-1)\frac{\psi_s^2}{\psi^2}\big]d\eta^m\bigg|_{(s_0,t)}
=\bigg[\frac{\psi_s^2}{\psi^2}Bd\xi^m
+\frac{\psi_s^2}{\psi^2}Dd\xi^{m-1}
+\frac{A(s_0,t)}{s_0}Bd\xi^{m-1}+\\
&+\frac{\psi_s}{\psi}Bd\xi^m_s
+\frac{\psi_s}{\psi}B\xi^{m-1}_sd\xi^{m-1}
+Bd\xi^{m-1}_s(\xi^{m-1}_s+\xi^{m-2}_s)
+|\xi^{m-2}_s|^2d\xi^{m-1}B\bigg]\bigg|_{(s_0,t)}
\end{align*}
Note that the coefficients of the preceding equation, at $s_0:=s(x_0,t)$, are bounded functions in $t$.
Whence, integrating with respect to time and taking into account the
pointwise estimates in Lemmas \ref{lempm}, \ref{estB,D}
along with the ones we just proved above we deduce ($\alpha\geq1$, a.e. $t\in[0,T]$)
\begin{align}\label{pdetas_0}
\notag|\frac{d\eta^m}{\ell^k}(s_0,t)|^2
\leq&\;CT\int^t_0\|d\xi^m\|^2_{H^2_{\alpha+1}}d\tau+CT\int^t_0\|d\xi^{m-1}\|^2_{H^1_{\alpha+1}}d\tau\\
\notag&+C\sqrt{\widetilde{C}\mathcal{E}_0}\int^t_0\big(\|\frac{\xi^{m-1}_{ss}}{\ell^{\alpha-1}}\|_{L^2}
+\|\frac{\xi^{m-2}_{ss}}{\ell^{\alpha-1}}\|_{L^2}\big)d\tau
\int^t_0\|d\xi^{m-1}\|^2_{H^2_{\alpha+1}}d\tau\\
\notag&+C\widetilde{C}\mathcal{E}_0T\int^t_0\|\frac{\xi^{m-2}_{ss}}{\ell^{\alpha-1}}\|^2_{L^2}dt
\int^t_0\|d\xi^{m-1}\|^2_{H^1_{\alpha+1}}d\tau\\
\overset{(\ref{itenest})}{\leq}&\;CT\int^t_0\|d\xi^m\|^2_{H^2_{\alpha+1}}d\tau
+C\widetilde{C}\mathcal{E}_0\sqrt{T}\int^t_0\|d\xi^{m-1}\|^2_{H^2_{\alpha+1}}d\tau.
\end{align}
Then integrating in the $s$-direction yields
\begin{align*}
|\frac{d\eta^m}{\ell^k}(s,t)|^2\leq&\;
\bigg|\int^{s_0}_s\partial_s|\frac{d\eta^m}{\ell^k}|^2ds\bigg|+|\frac{d\eta^m}{\ell^k}(s_0,t)|^2\\
\leq&\;C(\|\frac{d\eta^m_s}{\ell^{\alpha-1}}\|^2_{L^2}+(k+1)\|\frac{d\eta^m}{\ell^\alpha}\|^2_{L^2})
+CT\int^t_0\|d\xi^m\|^2_{H^2_{\alpha+1}}d\tau\\
&+C\widetilde{C}\mathcal{E}_0\sqrt{T}\int^T_0\|d\xi^{m-1}\|^2_{H^2_{\alpha+1}}d\tau,
\end{align*}
for every $m\in\mathbb{N}$, a.e. $t\in[0,T]$ and $k=0,\ldots,\alpha-1$, as required. The decay stated
in the end of the present lemma follows directly from (\ref{itsp}).
\end{proof}
Now, let
\begin{align}\label{dFm1}
dF^m_1:=&\;\frac{\psi_s^2}{\psi^2}Dd\xi^m+\frac{A}{s}Bd\xi^m
+\frac{\psi_s}{\psi}B\xi^m_sd\xi^m
+Bd\xi^m_s(\xi^m_s+\xi^{m-1}_s)+|\xi^{m-1}_s|^2d\xi^mB
\end{align}
and
\begin{align}\label{dFm2}
\notag dF^m_2:=&\;(\frac{\psi_{ss}}{\psi}+(n-1)\frac{\psi_s^2}{\psi^2})
(Dd\xi^m+Dd\eta^m)+\frac{A}{s}Bd\xi^m
+\frac{\psi_{s}}{\psi}B\xi^m_s(d\eta^m+d\xi^m)\\
&+\xi^m_{ss}(d\eta^m+d\xi^m)
+Bd\xi^m_s(\xi^m_s+\xi^{m-1}_s)
+|\xi^{m-1}_s|^2B(d\eta^m+d\xi^m)\\
&+\frac{\psi_s}{\psi}\eta^m_sB(d\eta^m+d\xi^m)
\notag+B(\xi^m_sd\eta^m_s+\eta^{m-1}_sd\xi^m_s)
+\eta^{m-1}_s\xi^{m-1}_sB(d\eta^m+d\xi^m)
\end{align}
In this notation, (\ref{dpde}) can be rewritten in the form
\begin{align}\label{dpde2}
\notag d\eta^{m+1}_t=&\;\frac{\psi_s^2}{\psi^2}Bd\xi^{m+1}
+\frac{\psi_s}{\psi}Bd\xi^{m+1}_s
+\frac{A}{s}Dd\eta^{m+1}
+2n(n-1)\frac{\psi_s^2}{\psi^2}d\eta^{m+1}+dF^m_1\\
d\xi^{m+1}_t=&\;(\frac{\psi_{ss}}{\psi}+(n-1)\frac{\psi_s^2}{\psi^2})
(Bd\eta^{m+1}+Bd\xi^{m+1})
+\frac{\psi_{s}}{\psi}Bd\xi^{m+1}_s
+|B|d\xi^{m+1}_{ss}\\
&\notag+\frac{\psi_s}{\psi}Bd\eta^{m+1}_s+dF^m_2
\end{align}
We establish some estimates for the functions $dF^m_1,dF^m_2$
that we will use in proving Proposition \ref{contrmap}.
\begin{lemma}\label{lemdFm}
For any function $u\in L^2(0,T;L^2_k(s))$,
$k=\alpha,\alpha+1$, $\alpha\geq2$, a.e. $t\in[0,T]$, and sufficiently small $\mathcal{E}_0,T$
the following estimates hold (for $k=\alpha+1$ in the first and third equation and $k=\alpha$ in the second one):
\begin{align}\label{intdFm1est}
\notag\int^t_0\int\frac{u\cdot dF^m_1}{\ell^{2\alpha}}dsd\tau
\leq&\;C\sigma\int^t_0\|\frac{u}{\ell^{\alpha+1}}\|^2_{L^2(s_{min},s_0)}d\tau
+C\int^t_0\|\frac{u}{\ell^\alpha}\|^2_{L^2}d\tau\\
&+C\bigg(\int^T_0A^2(\tau)d\tau+\widetilde{C}\mathcal{E}_0\sqrt{T}\bigg)\|d\xi^m\|^2_{L^\infty(0,T;H^1_\alpha)},
\end{align}
\begin{align}\label{intdFm1_sest}
\notag&\int^t_0\int\frac{u\cdot\partial_s(dF^m_1)}{\ell^{2\alpha-2}}dsd\tau
\leq C\sigma\int^t_0\|\frac{u}{\ell^\alpha}\|^2_{L^2(s_{min},s_0)}d\tau
+C\int^t_0\|\frac{u}{\ell^{\alpha-1}}\|^2_{L^2}d\tau\\
&+\frac{C}{\varepsilon_1}\sqrt{\widetilde{C}\mathcal{E}_0}\int^t_0
\big(\|\frac{\xi^m_{ss}}{\ell^{\alpha-1}}\|_{L^2}+\|\frac{\xi^{m-1}_{ss}}{\ell^{\alpha-1}}\|_{L^2}\big)
\|\frac{u}{\ell^{\alpha-1}}\|^2_{L^2}d\tau\\
\notag&+C\int^t_0\big(\|\frac{\xi^m_{ss}}{\ell^{\alpha-1}}\|_{L^2}+\|\frac{\xi^{m-1}_{ss}}{\ell^{\alpha-1}}\|_{L^2}\big)
\|d\xi^m\|_{H^2_{\alpha+1}}\|\frac{u}{\ell^{\alpha-1}}\|_{L^2}d\tau\\
\notag&+(C\widetilde{C}\mathcal{E}_0+\varepsilon_1)\int^T_0\|d\xi^m\|^2_{H^2_{\alpha+1}}d\tau
+C\bigg(\sigma\int^T_0A^2(\tau)d\tau+\widetilde{C}\mathcal{E}_0\bigg)\|d\xi^m\|^2_{L^\infty(0,T;H^1_\alpha)}
\end{align}
and
\begin{align}\label{intdFm2est}
\notag\int^t_0\int\frac{u\cdot dF^m_2}{\ell^{2\alpha}}dsd\tau\leq&\;
\tilde{\varepsilon}\int^t_0\|\frac{u}{\ell^{\alpha+1}}\|^2_{L^2}d\tau
+\frac{C}{\tilde{\varepsilon}}(\widetilde{C}\mathcal{E}_0)^2\sqrt{T}
\int^T_0\|d\xi^{m-1}\|^2_{H^2_{\alpha+1}}d\tau\\
&+\frac{C}{\tilde{\varepsilon}}\bigg(\sigma\int^T_0A^2(\tau)d\tau+\widetilde{C}\mathcal{E}_0\bigg)
\mathcal{E}(d\eta^m,d\xi^m;T),
\end{align}
where $C$ is a constant depending only on $n$.
\end{lemma}
\begin{proof}
Recall the estimates on the coefficients
(\ref{coeff}), (\ref{coeffinfty}), (\ref{A(t)}).
Plugging (\ref{dFm1}) in the integral below
we obtain the following (for a.e. $t\in[0,T]$).
\begin{align}\label{dFm1est}
&\int\frac{u\cdot dF^m_1}{\ell^{2\alpha}}ds\\
\notag=&\int\frac{u}{\ell^{2\alpha}}\bigg[
\frac{\psi_s^2}{\psi^2}Dd\xi^m+\frac{A}{s}Bd\xi^m
+\frac{\psi_s}{\psi}B\xi^m_sd\xi^m
+Bd\xi^m_s(\xi^m_s+\xi^{m-1}_s)+|\xi^{m-1}_s|^2d\xi^mB
\bigg]ds\\
\tag{using the estimate (\ref{pB,D}) for the fraction $\frac{D}{\ell}$ and $B$}\leq&\;
\|\frac{u}{s\ell^\alpha}\|^2_{L^2}
+C\widetilde{C}\mathcal{E}_0\|\frac{d\xi^m}{s\ell^{\alpha-1}}\|^2_{L^2}\\
\tag{employ the $L^\infty$ estimate (\ref{pm_s}), $k=0$}
&+\|\frac{u}{s\ell^\alpha}\|^2_{L^2}
+CA^2(t)\|\frac{d\xi^m}{\ell^\alpha}\|^2_{L^2}
+\|\frac{u}{s\ell^\alpha}\|^2_{L^2}\\
&+C\sqrt{\widetilde{C}\mathcal{E}_0}\|\frac{\xi^m_{ss}}{\ell^{\alpha-1}}\|_{L^2}\|\frac{d\xi^m}{\ell^\alpha}\|^2_{L^2}
\notag+\|\frac{u}{s\ell^\alpha}\|^2_{L^2}
+C\sqrt{\widetilde{C}\mathcal{E}_0}\big(\|\frac{\xi^m_{ss}}{\ell^{\alpha-1}}\|_{L^2}
+\|\frac{\xi^{m-1}_{ss}}{\ell^{\alpha-1}}\|_{L^2}\big)
\|\frac{d\xi^m_s}{\ell^{\alpha-1}}\|^2_{L^2}\\
&+\|\frac{u}{s\ell^\alpha}\|^2_{L^2}
\notag+C\sqrt{\widetilde{C}\mathcal{E}_0}\|\frac{\xi^m_{ss}}{\ell^{\alpha-1}}\|_{L^2}
\|d\xi^m\|^2_{L^\infty(s)}\|\frac{\xi^{m-1}_s}{\ell^{\alpha-1}}\|^2_{L^2}\\
\tag{by (\ref{l/s}); see $^\star$ below}
\notag\leq&\;C\sigma
\|\frac{u}{\ell^{\alpha+1}}\|^2_{L^2(s_{min},s_0)}
+C\|\frac{u}{\ell^\alpha}\|^2_{L^2(s_0,+\infty)}
+C\widetilde{C}\mathcal{E}_0\sigma\|\frac{d\xi^m}{\ell^\alpha}\|^2_{L^2}\\
&\notag+C\big(A^2(t)+\sqrt{\widetilde{C}\mathcal{E}_0}\|\frac{\xi^m_{ss}}{\ell^{\alpha-1}}\|_{L^2}\big)
\|\frac{d\xi^m}{\ell^\alpha}\|^2_{L^2}
+C\sqrt{\widetilde{C}\mathcal{E}_0}\big(\|\frac{\xi^m_{ss}}{\ell^{\alpha-1}}\|_{L^2}
+\|\frac{\xi^{m-1}_{ss}}{\ell^{\alpha-1}}\|_{L^2}\big)
\|\frac{d\xi^m_s}{\ell^{\alpha-1}}\|^2_{L^2}\\
\tag{utilizing (\ref{itenest}) and (\ref{pdxi})}
&+C(\widetilde{C}\mathcal{E}_0)^{\frac{3}{2}}\|\frac{\xi^m_{ss}}{\ell^{\alpha-1}}\|_{L^2}
\big(\|\frac{d\xi^m}{\ell^\alpha}\|^2_{L^2}+\|\frac{d\xi^m_s}{\ell^{\alpha-1}}\|^2_{L^2}\big),
\end{align}
$^\star$since $\ell(s,t)=O(1)$ when $s\in(s_0,+\infty)$; recall Definition \ref{weight}
and the considerations on values of $s$  at the endpoints of integration (\ref{ints1}), (\ref{ints2}). Integrating
on $[0,t]$, $0\leq t\leq T$ ($\sqrt{T}<\frac{1}{\sigma}$) and utilizing (\ref{itenest}) once more, we derive (\ref{intdFm1est}),
provided $\widetilde{C}\mathcal{E}_0<1$.

We prove now the corresponding estimates for $\partial_s(dF^m_1)$
\begin{align}\label{partial_sdFm1}
\bullet\;\;\;\int\frac{u\cdot\partial_s(dF^m_1)}{\ell^{2\alpha-2}}ds
\end{align}
This time we plug in the RHS of (\ref{dFm1}) and
examine each generated term separately.
For the first three terms, we recall additionally the estimates
on the derivatives of the coefficients (\ref{coeff_s}), (\ref{coeff_sinfty})
and (\ref{A(t)}). Then for a.e. $t\in[0,T]$ we have
\begin{align*}
\tag{\ref{partial_sdFm1}a}
&\int\frac{u}{\ell^{2\alpha-2}}\partial_s\bigg[
\frac{\psi_s^2}{\psi^2}Dd\xi^m
\bigg]ds\\
=&\int\frac{u}{\ell^{2\alpha-2}}\bigg[
\partial_s(\frac{\psi_s^2}{\psi^2})Dd\xi^m
+\frac{\psi_s^2}{\psi^2}D_sd\xi^m
+\frac{\psi_s^2}{\psi^2}Dd\xi^m_s
\bigg]ds\\
\tag{by the pointwise estimate of $D$ (\ref{pB,D}), $k=1$}
\leq&\;\|\frac{u}{s\ell^{\alpha-1}}\|^2_{L^2}
+C\widetilde{C}\mathcal{E}_0\|\frac{d\xi^m}{s^2\ell^{\alpha-2}}\|^2_{L^2}\\
&+\|\frac{u}{s\ell^{\alpha-1}}\|^2_{L^2}+C\|\frac{d\xi^m}{s}\|^2_{L^\infty(s)}\cdot\|\frac{D_s}{\ell^{\alpha-1}}\|^2_{L^2}
+\|\frac{u}{s\ell^{\alpha-1}}\|^2_{L^2}+C\widetilde{C}\mathcal{E}_0\|\frac{d\xi^m_s}{s\ell^{\alpha-2}}\|^2_{L^2}\\
\tag{using (\ref{l/s}); recall (\ref{ints2}) and (\ref{l})}
\leq&\;C\sigma\|\frac{u}{\ell^\alpha}\|^2_{L^2(s_{min},s_0)}
+C\|\frac{u}{\ell^{\alpha-1}}\|^2_{L^2(s_0,+\infty)}\\
\tag{by (\ref{B_s,D_s}) and the $L^\infty$ (\ref{pdxi})}
&+C\widetilde{C}\mathcal{E}_0\sigma^2\|\frac{d\xi^m}{\ell^\alpha}\|^2_{L^2}
+C\widetilde{C}\mathcal{E}_0\sigma\|\frac{d\xi^m_s}{\ell^{\alpha-1}}\|^2_{L^2}
\end{align*}
Similarly, utilizing the estimates of $B$ (\ref{pB,D}), (\ref{B_s,D_s})
we obtain
\begin{align*}
\tag{\ref{partial_sdFm1}b}
&\int\frac{u}{\ell^{2\alpha-2}}\partial_s\bigg[\frac{A}{s}Bd\xi^m
\bigg]ds\\
\tag{now recall consideration (\ref{A(t)})}
=&\int\frac{u}{\ell^{2\alpha-2}}\bigg[
\partial_s(\frac{A}{s})Bd\xi^m
+\frac{A}{s}B_sd\xi^m
+\frac{A}{s}Bd\xi^m_s
\bigg]ds\\
\leq&\;\|\frac{u}{s\ell^{\alpha-1}}\|^2_{L^2}
+CA^2(t)\|\frac{d\xi^m}{s\ell^{\alpha-1}}\|^2_{L^2}
+\|\frac{u}{s\ell^{\alpha-1}}\|^2_{L^2}
+CA^2(t)\|d\xi^m\|^2_{L^\infty(s)}\cdot\|\frac{B_s}{\ell^{\alpha-1}}\|^2_{L^2}\\
&+\|\frac{u}{s\ell^{\alpha-1}}\|^2_{L^2}+CA^2(t)\|\frac{d\xi^m_s}{\ell^{\alpha-1}}\|^2_{L^2}\\
\tag{by (\ref{l/s})}
\leq&\;C\sigma\|\frac{u}{\ell^\alpha}\|^2_{L^2(s_{min},s_0)}
+C\|\frac{u}{\ell^{\alpha-1}}\|^2_{L^2(s_0,+\infty)}
+CA^2(t)\sigma\|\frac{d\xi^m}{\ell^\alpha}\|^2_{L^2}\\
\tag{invoking the $L^\infty$ estimate (\ref{pdxi})}
&+C\widetilde{C}\mathcal{E}_0A^2(t)\|d\xi^m\|^2_{H^1_\alpha}
+CA^2(t)\|\frac{d\xi^m_s}{\ell^{\alpha-1}}\|^2_{L^2}
\end{align*}
Next, we estimate
\begin{align*}
\tag{\ref{partial_sdFm1}c}
&\int\frac{u}{\ell^{2\alpha-2}}\partial_s\bigg[
\frac{\psi_s}{\psi}B\xi^m_sd\xi^m
\bigg]ds\\
=&\;\int\frac{u}{\ell^{2\alpha-2}}\bigg[
\partial_s(\frac{\psi_s}{\psi})B\xi^m_sd\xi^m
+\frac{\psi_s}{\psi}B_s\xi^m_sd\xi^m
+\frac{\psi_s}{\psi}B\xi^m_{ss}d\xi^m
+\frac{\psi_s}{\psi}B\xi^m_sd\xi^m_s
\bigg]ds\\
\tag{by (\ref{pm_s}), (\ref{pB,D})}
\leq&\;\|\frac{u}{s\ell^{\alpha-1}}\|^2_{L^2}
+C\sqrt{\widetilde{C}\mathcal{E}_0}\|\frac{\xi^m_{ss}}{\ell^{\alpha-1}}\|_{L^2}\|\frac{d\xi^m}{s\ell^{\alpha-1}}\|^2_{L^2}\\
&+\|\frac{u}{s\ell^{\alpha-1}}\|^2_{L^2}
+C\sqrt{\widetilde{C}\mathcal{E}_0}\|\frac{\xi^m_{ss}}{\ell^{\alpha-1}}\|_{L^2}
\|d\xi^m\|^2_{L^\infty(s)}\|\frac{B_s}{\ell^{\alpha-1}}\|^2_{L^2}\\
&+\|\frac{u}{s\ell^{\alpha-1}}\|^2_{L^2}
+C\|d\xi^m\|^2_{L^\infty(s)}\|\frac{\xi^m_{ss}}{\ell^{\alpha-1}}\|^2_{L^2}
+\|\frac{u}{s\ell^{\alpha-1}}\|^2_{L^2}
+C\sqrt{\widetilde{C}\mathcal{E}_0}\|\frac{\xi^m_{ss}}{\ell^{\alpha-1}}\|_{L^2}\|\frac{d\xi^m_s}{\ell^{\alpha-1}}\|^2_{L^2}\\
\tag{using (\ref{l/s})}
\leq&\;C\sigma\|\frac{u}{\ell^\alpha}\|^2_{L^2(s_{min},s_0)}
+C\|\frac{u}{\ell^{\alpha-1}}\|^2_{L^2(s_0,+\infty)}
+C\sqrt{\widetilde{C}\mathcal{E}_0}\sigma\|\frac{\xi^m_{ss}}{\ell^{\alpha-1}}\|_{L^2}\|\frac{d\xi^m}{\ell^\alpha}\|^2_{L^2}\\
\tag{from (\ref{B_s,D_s}) and the pointwise estimate (\ref{pdxi})}
&+C(\widetilde{C}\mathcal{E}_0)^{\frac{3}{2}}\|\frac{\xi^m_{ss}}{\ell^{\alpha-1}}\|_{L^2}\|d\xi^m\|^2_{H^1_\alpha}\\
&+C\|\frac{\xi^m_{ss}}{\ell^{\alpha-1}}\|^2_{L^2}\|d\xi^m\|^2_{H^1_\alpha}
+C\sqrt{\widetilde{C}\mathcal{E}_0}\|\frac{\xi^m_{ss}}{\ell^{\alpha-1}}\|_{L^2}\|\frac{d\xi^m_s}{\ell^{\alpha-1}}\|^2_{L^2}
\end{align*}
Employing again the pointwise estimate of $\xi^i_s$ (\ref{pm_s}), $i=m-1,m$, $k=0$,
and the properties of $B$ (\ref{pB,D}), (\ref{B_s,D_s}) we derive
\begin{align*}
\tag{\ref{partial_sdFm1}d}
&\int\frac{u}{\ell^{2\alpha-2}}\partial_s\bigg[
Bd\xi^m_s(\xi^m_s+\xi^{m-1}_s)+|\xi^{m-1}_s|^2d\xi^mB
\bigg]ds\\
=&\int\frac{u}{\ell^{2\alpha-2}}\bigg[
B_sd\xi^m_s(\xi^m_s+\xi^{m-1}_s)
+Bd\xi^m_{ss}(\xi^m_s+\xi^{m-1}_s)
+Bd\xi^m_s(\xi^m_{ss}+\xi^{m-1}_{ss})\\
&+2\xi^{m-1}_s\xi^{m-1}_{ss}d\xi^mB+|\xi^{m-1}_s|^2d\xi^m_sB
+|\xi^{m-1}_s|^2d\xi^mB_s
\bigg]ds\\
\leq&\;C\sqrt{\widetilde{C}\mathcal{E}_0}\big(\|\frac{\xi^m_{ss}}{\ell^{\alpha-1}}\|_{L^2}
+\|\frac{\xi^{m-1}_{ss}}{\ell^{\alpha-1}}\|_{L^2}\big)
\|\frac{u}{\ell^{\alpha-1}}\|^2_{L^2}
+\|d\xi^m_s\|^2_{L^\infty(s)}\|\frac{B_s}{\ell^{\alpha-1}}\|^2_{L^2}\\
&+\varepsilon_1\|\frac{d\xi^m_{ss}}{\ell^{\alpha-1}}\|^2_{L^2}
+\frac{C}{\varepsilon_1}\sqrt{\widetilde{C}\mathcal{E}_0}\big(\|\frac{\xi^m_{ss}}{\ell^{\alpha-1}}\|_{L^2}
+\|\frac{\xi^{m-1}_{ss}}{\ell^{\alpha-1}}\|_{L^2}\big)\|\frac{u}{\ell^{\alpha-1}}\|^2_{L^2}\\
&+\|\frac{u}{\ell^{\alpha-1}}\|_{L^2}\|d\xi^m_s\|_{L^\infty(s)}\big(\|\frac{\xi^m_{ss}}{\ell^{\alpha-1}}\|_{L^2}
+\|\frac{\xi^{m-1}_{ss}}{\ell^{\alpha-1}}\|_{L^2}\big)
+C\sqrt{\widetilde{C}\mathcal{E}_0}\|\frac{\xi^{m-1}_{ss}}{\ell^{\alpha-1}}\|_{L^2}\|\frac{u}{\ell^{\alpha-1}}\|^2_{L^2}\\
&+C\|d\xi^m\|^2_{L^\infty(s)}\|\frac{\xi^{m-1}_{ss}}{\ell^{\alpha-1}}\|^2_{L^2}
+C\sqrt{\widetilde{C}\mathcal{E}_0}\|\frac{\xi^{m-1}_{ss}}{\ell^{\alpha-1}}\|_{L^2}\|\frac{u}{\ell^{\alpha-1}}\|^2_{L^2}
+C\sqrt{\widetilde{C}\mathcal{E}_0}\|\frac{\xi^{m-1}_{ss}}{\ell^{\alpha-1}}\|_{L^2}\|\frac{d\xi^m_s}{\ell^{\alpha-1}}\|^2_{L^2}\\
&+C\sqrt{\widetilde{C}\mathcal{E}_0}\|\frac{\xi^{m-1}_{ss}}{\ell^{\alpha-1}}\|_{L^2}\|\frac{u}{\ell^{\alpha-1}}\|^2_{L^2}
+C\sqrt{\widetilde{C}\mathcal{E}_0}\|\frac{\xi^{m-1}_{ss}}{\ell^{\alpha-1}}\|_{L^2}
\|d\xi^m\|^2_{L^\infty(s)}\|\frac{B_s}{\ell^{\alpha-1}}\|^2_{L^2}\\
\tag{$0<\varepsilon_1<1$}
\leq&\;\frac{C}{\varepsilon_1}\sqrt{\widetilde{C}\mathcal{E}_0}\big(\|\frac{\xi^m_{ss}}{\ell^{\alpha-1}}\|_{L^2}
+\|\frac{\xi^{m-1}_{ss}}{\ell^{\alpha-1}}\|_{L^2}\big)\|\frac{u}{\ell^{\alpha-1}}\|^2_{L^2}
+(C\widetilde{C}\mathcal{E}_0+\varepsilon_1)\|d\xi^m\|^2_{H^2_{\alpha+1}}\\
\tag{by the $L^\infty$ estimate (\ref{pdxi_s})}
&+C\big(\|\frac{\xi^m_{ss}}{\ell^{\alpha-1}}\|_{L^2}
+\|\frac{\xi^{m-1}_{ss}}{\ell^{\alpha-1}}\|_{L^2}\big)
\|d\xi^m\|_{H^2_{\alpha+1}}\|\frac{u}{\ell^{\alpha-1}}\|_{L^2}\\
&+C\|\frac{\xi^{m-1}_{ss}}{\ell^{\alpha-1}}\|^2_{L^2}\|d\xi^m\|^2_{H^1_\alpha}
+C\sqrt{\widetilde{C}\mathcal{E}_0}\|\frac{\xi^{m-1}_{ss}}{\ell^{\alpha-1}}\|_{L^2}\|\frac{d\xi^m_s}{\ell^{\alpha-1}}\|^2_{L^2}
+C(\widetilde{C}\mathcal{E}_0)^{\frac{3}{2}}\|\frac{\xi^{m-1}_{ss}}{\ell^{\alpha-1}}\|_{L^2}
\|d\xi^m\|^2_{H^1_\alpha},
\end{align*}
applying (\ref{pdxi}) in the end. (We remark here that the control of the term $Bd\xi^m_s(\xi^m_{ss}+\xi^{m-1}_{ss})$ in the above is one
of the most delicate that we have to perform here, essentially due to
the fact that the energies only involve the first derivative of $\eta^m$;
this will be used in subsection \S\ref{L2deta_sdxi_s} below).\footnote{In fact
if this term in the equation had been slightly more nonlinear, our result would not hold.}

Combining (\ref{partial_sdFm1}a)-(\ref{partial_sdFm1}d), we achieve the following estimate
of the term (\ref{partial_sdFm1}) (for a.e. $t\in[0,T]$).
\begin{align}\label{partial_sdFm1est}
&\int\frac{u\cdot\partial_s(dF^m_1)}{\ell^{2\alpha-2}}ds\\
\notag\leq&\;C\sigma\|\frac{u}{\ell^\alpha}\|^2_{L^2(s_{min},s_0)}
+C\|\frac{u}{\ell^{\alpha-1}}\|^2_{L^2(s_0,+\infty)}
+\frac{C}{\varepsilon_1}\sqrt{\widetilde{C}\mathcal{E}_0}\big(\|\frac{\xi^m_{ss}}{\ell^{\alpha-1}}\|_{L^2}
+\|\frac{\xi^{m-1}_{ss}}{\ell^{\alpha-1}}\|_{L^2}\big)\|\frac{u}{\ell^{\alpha-1}}\|^2_{L^2}\\
&\notag+C\big(\|\frac{\xi^m_{ss}}{\ell^{\alpha-1}}\|_{L^2}
+\|\frac{\xi^{m-1}_{ss}}{\ell^{\alpha-1}}\|_{L^2}\big)
\|d\xi^m\|_{H^2_{\alpha+1}}\|\frac{u}{\ell^{\alpha-1}}\|_{L^2}
+(C\widetilde{C}\mathcal{E}_0+\varepsilon_1)\|d\xi^m\|^2_{H^2_{\alpha+1}}\\
&\notag+C\big(\widetilde{C}\mathcal{E}_0\sigma^2+A^2(t)\sigma
+\sqrt{\widetilde{C}\mathcal{E}_0}\sigma\|\frac{\xi^m_{ss}}{\ell^{\alpha-1}}\|_{L^2}\big)\|\frac{d\xi^m}{\ell^\alpha}\|^2_{L^2}\\
\notag&+C\bigg[\widetilde{C}\mathcal{E}_0\sigma+A^2(t)
+\sqrt{\widetilde{C}\mathcal{E}_0}\big(\|\frac{\xi^m_{ss}}{\ell^{\alpha-1}}\|_{L^2}
+\|\frac{\xi^{m-1}_{ss}}{\ell^{\alpha-1}}\|_{L^2}\big)\bigg]
\|\frac{d\xi^m_s}{\ell^{\alpha-1}}\|^2_{L^2}\\
&\notag+C\big(\|\frac{\xi^m_{ss}}{\ell^{\alpha-1}}\|^2_{L^2}
+\|\frac{\xi^{m-1}_{ss}}{\ell^{\alpha-1}}\|^2_{L^2}\big)\|d\xi^m\|^2_{H^1_\alpha}
+C(\widetilde{C}\mathcal{E}_0)^{\frac{3}{2}}\big(\|\frac{\xi^m_{ss}}{\ell^{\alpha-1}}\|_{L^2}
+\|\frac{\xi^{m-1}_{ss}}{\ell^{\alpha-1}}\|_{L^2}\big)\|d\xi^m\|^2_{H^1_\alpha}
\end{align}
Thus, integrating on $[0,t]$, $t\leq T<\frac{1}{\sigma^2}$, and using the energy estimate
(\ref{itenest}), $\widetilde{C}\mathcal{E}_0<1$, we obtain the desired bound (\ref{intdFm1_sest}).

Lastly, we derive estimates for $dF^m_2$ (\ref{dFm2}) in a similar manner
to (\ref{dFm1est}). Recall again the estimates on the coefficients
(\ref{coeff}), (\ref{coeffinfty}), (\ref{A(t)}).
\begin{align}\label{dFm2est}
&\int\frac{u\cdot dF^m_2}{\ell^{2\alpha}}ds
\leq\tilde{\varepsilon}\|\frac{u}{\ell^{\alpha+1}}\|^2_{L^2}
+\frac{1}{\tilde{\varepsilon}}\|\frac{dF^m_2}{\ell^{\alpha-1}}\|^2_{L^2}\\
\tag{applying (\ref{pB,D}); for $\frac{D}{\ell}$ especially}
\leq&\;\tilde{\varepsilon}\|\frac{u}{\ell^{\alpha+1}}\|^2_{L^2}
+\frac{C}{\tilde{\varepsilon}}\widetilde{C}\mathcal{E}_0\big(\|\frac{d\xi^m}{s^2\ell^{\alpha-2}}\|^2_{L^2}
+\|\frac{d\eta^m}{s^2\ell^{\alpha-2}}\|^2_{L^2}\big)\\
\tag{by (\ref{pm_s})}
&+\frac{C}{\tilde{\varepsilon}}A^2(t)\|\frac{d\xi^m}{s\ell^{\alpha-1}}\|^2_{L^2}
+\frac{C}{\tilde{\varepsilon}}\sqrt{\widetilde{C}\mathcal{E}_0}
\|\frac{\xi^m_{ss}}{\ell^{\alpha-1}}\|_{L^2}
\big(\|\frac{d\eta^m}{s\ell^{\alpha-1}}\|^2_{L^2}
+\|\frac{d\xi^m}{s\ell^{\alpha-1}}\|^2_{L^2}\big)\\
\notag&+\frac{C}{\tilde{\varepsilon}}\big(\|d\eta^m\|^2_{L^\infty(s)}+\|d\xi^m\|^2_{L^\infty(s)}\big)
\|\frac{\xi^m_{ss}}{\ell^{\alpha-1}}\|^2_{L^2}
+\frac{C}{\tilde{\varepsilon}}\sqrt{\widetilde{C}\mathcal{E}_0}\big(\|\frac{\xi^m_{ss}}{\ell^{\alpha-1}}\|_{L^2}
+\|\frac{\xi^{m-1}_{ss}}{\ell^{\alpha-1}}\|_{L^2}\big)\|\frac{d\xi^m_s}{\ell^{\alpha-1}}\|^2_{L^2}\\
&\notag+\frac{C}{\tilde{\varepsilon}}\sqrt{\widetilde{C}\mathcal{E}_0}\|\frac{\xi^{m-1}_{ss}}{\ell^{\alpha-1}}\|_{L^2}
\big(\|d\eta^m\|^2_{L^\infty(s)}+\|d\xi^m\|^2_{L^\infty(s)}\big)\cdot\|\frac{\xi^{m-1}_s}{\ell^{\alpha-1}}\|^2_{L^2}\\
&+\frac{C}{\tilde{\varepsilon}}
\big(\|\frac{d\eta^m}{s}\|^2_{L^\infty(s)}+\|\frac{d\xi^m}{s}\|^2_{L^\infty(s)}\big)
\|\frac{\eta^m_s}{\ell^{\alpha-1}}\|^2_{L^2}
\notag+\frac{C}{\tilde{\varepsilon}}\sqrt{\widetilde{C}\mathcal{E}_0}\|\frac{\xi^m_{ss}}{\ell^{\alpha-1}}\|_{L^2}
\|\frac{d\eta^m_s}{\ell^{\alpha-1}}\|^2_{L^2}\\
&+\frac{C}{\tilde{\varepsilon}}\|d\xi^m_s\|^2_{L^\infty(s)}\|\frac{\eta^{m-1}_s}{\ell^{\alpha-1}}\|^2_{L^2}
\notag+\frac{C}{\tilde{\varepsilon}}\sqrt{\widetilde{C}\mathcal{E}_0}
\|\frac{\xi^{m-1}_{ss}}{\ell^{\alpha-1}}\|_{L^2}
\big(\|d\eta^m\|^2_{L^\infty(s)}+\|d\xi^m\|^2_{L^\infty(s)}\big)\|\frac{\eta^{m-1}_s}{\ell^{\alpha-1}}\|^2_{L^2}\\
\tag{by (\ref{l/s})}
\notag\leq&\;\tilde{\varepsilon}\|\frac{u}{\ell^{\alpha+1}}\|^2_{L^2}
+\frac{C}{\tilde{\varepsilon}}\widetilde{C}\mathcal{E}_0\sigma^2
\big(\|\frac{d\eta^m}{\ell^\alpha}\|^2_{L^2}
+\|\frac{d\xi^m}{\ell^\alpha}\|^2_{L^2}\big)
+\frac{C}{\tilde{\varepsilon}}\sigma A^2(t)\|\frac{d\xi^m}{\ell^\alpha}\|^2_{L^2}\\
\notag&+\frac{C}{\tilde{\varepsilon}}\sqrt{\widetilde{C}\mathcal{E}_0}
\big(\|\frac{\xi^m_{ss}}{\ell^{\alpha-1}}\|_{L^2}
+\|\frac{\xi^{m-1}_{ss}}{\ell^{\alpha-1}}\|_{L^2}\big)
\big(\sigma\|\frac{d\eta^m}{\ell^\alpha}\|^2_{L^2}
+\sigma\|\frac{d\xi^m}{\ell^\alpha}\|^2_{L^2}
+\|\frac{d\xi^m_s}{\ell^{\alpha-1}}\|^2_{L^2}
+\|\frac{d\eta^m_s}{\ell^{\alpha-1}}\|^2_{L^2}\big)\\
\notag&+\frac{C}{\tilde{\varepsilon}}\big(\|\frac{\xi^m_{ss}}{\ell^{\alpha-1}}\|^2_{L^2}
+(\widetilde{C}\mathcal{E}_0)^{\frac{3}{2}}\|\frac{\xi^{m-1}_{ss}}{\ell^{\alpha-1}}\|_{L^2}
+\sigma\widetilde{C}\mathcal{E}_0\big)
\bigg[\|d\eta^m\|^2_{H^1_\alpha}+\|d\xi^m\|^2_{H^1_\alpha}\\
\notag&+T\int^T_0\|d\xi^m\|^2_{H^2_{\alpha+1}}dt
+\widetilde{C}\mathcal{E}_0\sqrt{T}\int^T_0\|d\xi^{m-1}\|^2_{H^2_{\alpha+1}}dt
\bigg]
+\frac{C}{\tilde{\varepsilon}}\widetilde{C}\mathcal{E}_0\big(\|\frac{d\xi^m_{ss}}{\ell^{\alpha-1}}\|^2_{L^2}
+\|\frac{d\xi^m_s}{\ell^\alpha}\|^2_{L^2}\big),
\end{align}
employing the pointwise estimates (\ref{pdxi}), (\ref{pdxi_s}), (\ref{pdeta}). Integrating
up to time $0<t\leq T<\frac{1}{\sigma^2}$ and using the energy estimate (\ref{itenest}),
$\widetilde{C}\mathcal{E}_0<1$,
we arrive at (\ref{intdFm2est}). This completes the proof the lemma.
\end{proof}

\subsection{The $L^2_\alpha$ estimates for $d\eta^{m+1},d\xi^{m+1}$}\label{L2detadxi}
\noindent

The main energy estimates that will prove Proposition \ref{contrmap} are obtained in the present
subsection and the subsequent one. These are designed to give two Gronwall-type inequalities
which together (and combined with Lemma \ref{lemdFm})
imply Proposition \ref{contrmap}). Since the reasoning behind the estimates
that we are about to derive now and in \S\ref{L2deta_sdxi_s}
is quite similar, we discuss here the main objective that guides our analysis.
As usual, we consider the time-derivatives of the
$L^2_\alpha$ norms of $d\eta^{m+1},d\xi^{m+1}$ and seek to derive  a bound
\begin{align*}
\partial_t\big(\|d\eta^{m+1}\|_{L^2_\alpha}^2+\|d\xi^{m+1}\|_{L^2_\alpha}^2\big)
\le B(t)\big[\|d\eta^{m+1}\|_{L^2_\alpha}^2+ \|d\xi^{m+1}\|_{L^2_\alpha}^2\big]
+F(d\eta^m,d\xi^m,d\eta^m_s,...)
\end{align*}
where $B(t)$ is integrable in time.
It is precisely at this point that the most singular
terms\footnote{These are the terms $d\eta^{m+1}, d\xi^{m+1}$ and
$d\eta_s^{m+1}, d\xi_s^{m+1}$ multiplied by a coefficient
behaving like $\frac{1}{s^2}$ and $\frac{1}{s}$ (respectively) as $x,t\rightarrow0^+$.}
in the system (\ref{itpde}) are very dangerous.
A priori, calculating the time derivatives of $d\eta^{m+1}$, $d\xi^{m+1}$
would give us an estimate as above with $B(t)$ behaving like $1/t$.
It is exactly for this reason that we modified the iteration procedure (\ref{itpde})
to include  the unknown functions $d\eta^{m+1}, d\xi^{m+1}$ in specific lower-order terms of the RHSs.
As importantly, it is for this reason that the weight function $\ell$ was introduced.
The correct choice of weight function is crucial at this step,
yielding us an inequality of the form:
\begin{align*}
\partial_t\big(\|d\eta^{m+1}\|_{L^2_\alpha}^2+\|d\xi^{m+1}\|_{L^2_\alpha}^2\big)
\le&\;B(t)\big[\|d\eta^{m+1}\|_{L^2_\alpha}^2+\|d\xi^{m+1}\|_{L^2_\alpha}^2\big]
+F(d\eta^m,d\xi^m,d\eta^m_s,...)\\
&+N\big[\|d\eta^{m+1}\|_{L^2_{\alpha+1}}^2+
\|d\xi^{m+1}\|_{L^2_{\alpha+1}}^2+\|d\xi^{m+1}_s\|^2_{L^2_\alpha}\big],
\end{align*}
whereas now $B(t)$ is integrable and $N$ is {\it negative}, allowing hence the estimate to close.
This negative term we obtain here will be essential in closing the estimates for
Proposition \ref{contrmap} in the next subsection.
\newline

We commence with the $L^2_\alpha$ estimates from $d\eta^{m+1}$. Taking the time derivative
of the $L^2_\alpha$ norm of $d\eta^{m+1}$ and
using (\ref{l_t}), (\ref{partial_tds}) we derive (for a.e. $0\leq t\leq T$)
\begin{align}\label{L2deta}
\notag\frac{1}{2}\frac{d}{dt}\|d\eta^{m+1}\|^2_{L^2_\alpha}=&
\int\frac{d\eta^{m+1}d\eta^{m+1}_t}{\ell^{2\alpha}}ds
-\alpha\int\frac{|d\eta^{m+1}|^2}{\ell^{2\alpha+1}}\partial_t\ell ds
+\frac{1}{2}\int\frac{|d\eta^{m+1}|^2}{\ell^{2\alpha}}\partial_tds\\
\leq&\int\frac{d\eta^{m+1}d\eta^{m+1}_t}{\ell^{2\alpha}}ds
-\alpha\sigma\|\frac{d\eta^{m+1}}{\ell^{\alpha+1}}\|^2_{L^2(s_{min},s_0)}\\
\notag&+C\alpha\|\frac{d\eta^{m+1}}{\ell^{\alpha+1}}\|^2_{L^2(s_{min},s_0)}
+C\alpha\|\frac{d\eta^{m+1}}{\ell^\alpha}\|^2_{L^2}
+C\|\frac{d\eta^{m+1}}{s\ell^\alpha}\|^2_{L^2}
\end{align}
We estimate the last term using the comparison estimate (\ref{l/s})
\begin{align*}
\|\frac{d\eta^{m+1}}{s\ell^\alpha}\|^2_{L^2}\leq
C\sigma\|\frac{d\eta^{m+1}}{\ell^{\alpha+1}}\|^2_{L^2(s_{min,s_0})}
+C\|\frac{d\eta^{m+1}}{\ell^\alpha}\|^2_{L^2(s_0,+\infty)},
\end{align*}
since $\ell=O(1)$ in $(s_0,+\infty)$, {\bf HC}; see (\ref{l}) and (\ref{ints2}).
Using the estimates of the coefficients
(\ref{coeff}), (\ref{coeffinfty}), (\ref{A(t)}) we have
\begin{align*}
&\tag{plugging in the RHS of $d\eta^{m+1}_t$ (\ref{dpde2})}
\int\frac{d\eta^{m+1}d\eta^{m+1}_t}{\ell^{2\alpha}}ds\\
=&\int\frac{d\eta^{m+1}}{\ell^{2\alpha}}\bigg[
\frac{\psi_s^2}{\psi^2}Bd\xi^{m+1}
+\frac{\psi_s}{\psi}Bd\xi^{m+1}_s
+\frac{A}{s}Dd\eta^{m+1}
+2n(n-1)\frac{\psi_s^2}{\psi^2}d\eta^{m+1}+dF^m_1
\bigg]ds\\
\leq&\;\|\frac{d\eta^{m+1}}{s\ell^\alpha}\|^2_{L^2}
+C\|\frac{d\xi^{m+1}}{s\ell^\alpha}\|^2_{L^2}
+\varepsilon\|\frac{d\xi^{m+1}_s}{\ell^\alpha}\|^2_{L^2}
+\frac{C}{\varepsilon}\|\frac{d\eta^{m+1}}{s\ell^\alpha}\|^2_{L^2}\\
&+\|\frac{d\eta^{m+1}}{s\ell^\alpha}\|^2_{L^2}
+C\widetilde{C}\mathcal{E}_0A^2(t)\|\frac{d\eta^{m+1}}{\ell^\alpha}\|^2_{L^2}
\tag{by the $L^\infty$ estimates of $B,D$ (\ref{pB,D})}\\
&+C\|\frac{d\eta^{m+1}}{s\ell^\alpha}\|^2_{L^2}
+\int\frac{d\eta^{m+1}dF^m_1}{\ell^{2\alpha}}ds\\
\tag{employing (\ref{l/s}), $0<\varepsilon<1$}
\leq&\;\frac{C}{\varepsilon}\sigma\|\frac{d\eta^{m+1}}{\ell^{\alpha+1}}\|^2_{L^2(s_{min},s_0)}
+\frac{C}{\varepsilon}\|\frac{d\eta^{m+1}}{\ell^\alpha}\|^2_{L^2(s_0,+\infty)}\\
&+C\sigma\|\frac{d\xi^{m+1}}{\ell^{\alpha+1}}\|^2_{L^2(s_{min},s_0)}
+C\|\frac{d\xi^{m+1}}{\ell^\alpha}\|^2_{L^2(s_0,+\infty)}
+\varepsilon\|\frac{d\xi^{m+1}_s}{\ell^\alpha}\|^2_{L^2}\\
&+C\widetilde{C}\mathcal{E}_0A^2(t)\|\frac{d\eta^{m+1}}{\ell^\alpha}\|^2_{L^2}
+\int\frac{d\eta^{m+1}dF^m_1}{\ell^{2\alpha}}ds
\end{align*}
We next consider $\frac{d}{dt}\|d\xi^{m+1}\|^2_{L^2_\alpha}$. Arguing similarly to (\ref{L2deta}), we deduce
\begin{align}\label{L2dxi}
\frac{1}{2}\frac{d}{dt}\|d\xi^{m+1}\|^2_{L^2_\alpha}\leq&
\int\frac{d\xi^{m+1}d\xi^{m+1}_t}{\ell^{2\alpha}}ds
-\alpha\sigma\|\frac{d\xi^{m+1}}{\ell^{\alpha+1}}\|^2_{L^2(s_{min},s_0)}\\
\notag&+C\alpha\|\frac{d\xi^{m+1}}{\ell^{\alpha+1}}\|^2_{L^2(s_{min},s_0)}
+C\alpha\|\frac{d\xi^{m+1}}{\ell^\alpha}\|^2_{L^2}
+C\|\frac{d\xi^{m+1}}{s\ell^\alpha}\|^2_{L^2}
\end{align}
Again the last term is controlled by (\ref{l/s}),
so we need only estimate
\begin{align*}
\tag{plugging in the RHS of $d\xi^{m+1}_t$ (\ref{dpde2})}
&\int\frac{d\xi^{m+1}d\xi^{m+1}_t}{\ell^{2\alpha}}ds\\
=&\int\frac{d\xi^{m+1}}{\ell^{2\alpha}}\bigg[
(\frac{\psi_{ss}}{\psi}+(n-1)\frac{\psi_s^2}{\psi^2})
(Bd\eta^{m+1}+Bd\xi^{m+1})
+\frac{\psi_{s}}{\psi}Bd\xi^{m+1}_s\\
\tag{recall (\ref{coeff}), (\ref{coeffinfty})}
&+|B|d\xi^{m+1}_{ss}
+\frac{\psi_s}{\psi}Bd\eta^{m+1}_s
+dF^m_2\bigg]ds\\
\tag{by (\ref{pB,D}) for $B$}
\leq&\;\|\frac{d\xi^{m+1}}{s\ell^\alpha}\|^2_{L^2}
+C\|\frac{d\eta^{m+1}}{s\ell^\alpha}\|^2_{L^2}
+C\|\frac{d\xi^{m+1}}{s\ell^\alpha}\|^2_{L^2}
+\varepsilon\|\frac{d\xi^{m+1}_s}{\ell^\alpha}\|^2_{L^2}\\
&+\frac{C}{\varepsilon}\|\frac{d\xi^{m+1}}{s\ell^\alpha}\|^2_{L^2}
+\int\frac{|B|d\xi^{m+1}d\xi^{m+1}_{ss}}{\ell^{2\alpha}}ds
+\int\frac{\psi_s}{\psi}\frac{Bd\xi^{m+1}d\eta^{m+1}_s}{\ell^{2\alpha}}ds
+\int\frac{d\xi^{m+1}dF^m_2}{\ell^{2\alpha}}ds\\
\tag{using (\ref{l/s})}
\leq&\;\frac{C}{\varepsilon}\sigma\|\frac{d\xi^{m+1}}{\ell^{\alpha+1}}\|^2_{L^2(s_{min},s_0)}
+\frac{C}{\varepsilon}\|\frac{d\xi^{m+1}}{\ell^\alpha}\|^2_{L^2(s_0,+\infty)}
+\varepsilon\|\frac{d\xi^{m+1}_s}{\ell^\alpha}\|^2_{L^2}\\
&+C\sigma\|\frac{d\eta^{m+1}}{\ell^{\alpha+1}}\|^2_{L^2(s_{min},s_0)}
+C\|\frac{d\eta^{m+1}}{\ell^\alpha}\|^2_{L^2(s_0,+\infty)}
+\int\frac{d\xi^{m+1}dF^m_2}{\ell^{2\alpha}}ds\\
&+\int\frac{|B|d\xi^{m+1}d\xi^{m+1}_{ss}}{\ell^{2\alpha}}ds
+\int\frac{\psi_s}{\psi}\frac{Bd\xi^{m+1}d\eta^{m+1}_s}{\ell^{2\alpha}}ds
\end{align*}
We treat the last two terms separately.
First, integrating by parts and using the boundary condition (\ref{itbdc}) and (\ref{smaxdecay}) we derive
\begin{align*}
&\int\frac{|B|d\xi^{m+1}d\xi^{m+1}_{ss}}{\ell^{2\alpha}}ds\\
=&-\int|B|\frac{|d\xi^{m+1}_s|^2}{\ell^{2\alpha}}ds
-\int\frac{B_sd\xi^{m+1}d\xi^{m+1}_s}{\ell^{2\alpha}}ds
+2\alpha\int\frac{|B|d\xi^{m+1}d\xi^{m+1}_s}{\ell^{2\alpha+1}}\partial_s\ell ds\\
\tag{see (\ref{pB,D}), $c=\frac{1}{2}$}
\leq&-\frac{1}{2}\|\frac{d\xi^{m+1}_s}{\ell^\alpha}\|^2_{L^2}
+\varepsilon\|\frac{d\xi^{m+1}_s}{\ell^\alpha}\|^2_{L^2}
+\frac{C}{\varepsilon}\|\frac{d\xi^{m+1}}{\ell}\|^2_{L^\infty(s)}\|\frac{B_s}{\ell^{\alpha-1}}\|^2_{L^2}\\
\tag{recall (\ref{l_s})}
&+\varepsilon\|\frac{d\xi^{m+1}_s}{\ell^\alpha}\|^2_{L^2}
+\frac{C\alpha^2}{\varepsilon}\|\frac{d\xi^{m+1}}{\ell^{\alpha+1}}\|^2_{L^2(s_{min},s_0)}
+\frac{C\alpha^2}{\varepsilon}\|\frac{d\xi^{m+1}}{\ell^\alpha}\|^2_{L^2}\\
\tag{by (\ref{B_s,D_s}) and the $L^\infty$ estimate (\ref{pdxi})}
\leq&-\frac{1}{2}\|\frac{d\xi^{m+1}_s}{\ell^\alpha}\|^2_{L^2}
+(2\varepsilon+\frac{C}{\varepsilon}\widetilde{C}\mathcal{E}_0)\|\frac{d\xi^{m+1}_s}{\ell^\alpha}\|^2_{L^2}\\
\tag{$\alpha\geq1$}
&+(\frac{C}{\varepsilon}\widetilde{C}\mathcal{E}_0+\frac{C\alpha^2}{\varepsilon})
\|\frac{d\xi^{m+1}}{\ell^{\alpha+1}}\|^2_{L^2(s_{min},s_0)}
+\frac{C\alpha^2}{\varepsilon}\|\frac{d\xi^{m+1}}{\ell^\alpha}\|^2_{L^2}
\end{align*}
Similarly, by (\ref{itbdc}), (\ref{pm}), (\ref{smaxdecay}) using in addition the estimates of the coefficients
(\ref{coeff}), (\ref{coeff_s}), (\ref{coeffinfty}), (\ref{coeff_sinfty}) we obtain
\footnote{The possibility to control this next term using an
integration by parts  is essential in order to close
our estimates for the $L^2_\alpha$ norms of $d\xi^{m+1},d\eta^{m+1}$, without recourse to the higher derivatives.}
\begin{align*}
&\int\frac{\psi_s}{\psi}\frac{Bd\xi^{m+1}d\eta^{m+1}_s}{\ell^{2\alpha}}ds\\
=&-\int\partial_s(\frac{\psi_s}{\psi})\frac{Bd\xi^{m+1}d\eta^{m+1}}{\ell^{2\alpha}}ds
-\int\frac{\psi_s}{\psi}\frac{B_sd\xi^{m+1}d\eta^{m+1}}{\ell^{2\alpha}}ds
-\int\frac{\psi_s}{\psi}\frac{Bd\xi^{m+1}_sd\eta^{m+1}}{\ell^{2\alpha}}ds\\
\tag{recall (\ref{l_s}), (\ref{ints2})}
&+2\alpha\int\frac{\psi_s}{\psi}\frac{Bd\xi^{m+1}d\eta^{m+1}}{\ell^{2\alpha+1}}\partial_s\ell ds\\
\leq&\;\|\frac{d\eta^{m+1}}{s\ell^\alpha}\|^2_{L^2}
+C\|\frac{d\xi^{m+1}}{s\ell^\alpha}\|^2_{L^2}
+\|\frac{d\eta^{m+1}}{s\ell^\alpha}\|^2_{L^2}
+C\|\frac{d\xi^{m+1}}{\ell}\|^2_{L^\infty(s)}\|\frac{B_s}{\ell^{\alpha-1}}\|^2_{L^2}\\
&+\varepsilon\|\frac{d\xi^{m+1}_s}{\ell^\alpha}\|^2_{L^2}
+\frac{C}{\varepsilon}\|\frac{d\eta^{m+1}}{s\ell^\alpha}\|^2_{L^2}
+\|\frac{d\eta^{m+1}}{s\ell^\alpha}\|^2_{L^2}
+\alpha^2\|\frac{d\xi^{m+1}}{\ell^{\alpha+1}}\|^2_{L^2(s_{min},s_0)}
+\alpha^2\|\frac{d\xi^{m+1}}{\ell^\alpha}\|^2_{L^2}\\
\tag{employing (\ref{l/s}), $0<\varepsilon<1$}
\leq&\;\frac{C}{\varepsilon}\sigma\|\frac{d\eta^{m+1}}{\ell^{\alpha+1}}\|^2_{L^2(s_{min},s_0)}
+\frac{C}{\varepsilon}\|\frac{d\eta^{m+1}}{\ell^\alpha}\|^2_{L^2(s_0,+\infty)}\\
\tag{by the $L^\infty$ estimate (\ref{pdxi}) and (\ref{B_s,D_s})}
&+(C\sigma+C\widetilde{C}\mathcal{E}_0+\alpha^2)\|\frac{d\xi^{m+1}}{\ell^{\alpha+1}}\|^2_{L^2(s_{min},s_0)}\\
\tag{$\alpha\ge1$}
&+(\varepsilon+C\widetilde{C}\mathcal{E}_0)\|\frac{d\xi^{m+1}_s}{\ell^\alpha}\|^2_{L^2}
+C\alpha^2\|\frac{d\xi^{m+1}}{\ell^\alpha}\|^2_{L^2}
\end{align*}

Summing  (\ref{L2deta}), (\ref{L2dxi})
and taking into account the above estimates we derive that for a.e. $t\in[0,T]$:
\begin{align}\label{sumdetadxi1}
\notag&\frac{1}{2}\frac{d}{dt}(\|d\eta^{m+1}\|^2_{L^2_\alpha}
+\|d\xi^{m+1}\|^2_{L^2_\alpha})+\alpha\sigma\big(\|\frac{d\eta^{m+1}}{\ell^{\alpha+1}}\|^2_{L^2(s_{min},s_0)}
+\|\frac{d\xi^{m+1}}{\ell^{\alpha+1}}\|^2_{L^2(s_{min},s_0)}\big)\\
\leq&\;C\big(\frac{\alpha^2}{\varepsilon}+\widetilde{C}\mathcal{E}_0A^2(t)\big)(\|d\eta^{m+1}\|^2_{L^2_\alpha}
+\|d\xi^{m+1}\|^2_{L^2_\alpha})\\
&+C(\alpha+\frac{\sigma}{\varepsilon})\|\frac{d\eta^{m+1}}{\ell^{\alpha+1}}\|^2_{L^2(s_{min},s_0)}
\notag+\frac{C}{\varepsilon}(\alpha^2+\sigma)\|\frac{d\xi^{m+1}}{\ell^{\alpha+1}}\|^2_{L^2(s_{min},s_0)}\\
&+(4\varepsilon+\frac{C}{\varepsilon}\widetilde{C}\mathcal{E}_0-\frac{1}{2})\|\frac{d\xi^{m+1}_s}{\ell^\alpha}\|^2_{L^2}
\notag+\int\frac{d\eta^{m+1}dF^m_1}{\ell^{2\alpha}}ds+\int\frac{d\xi^{m+1}dF^m_2}{\ell^{2\alpha}}ds
\end{align}
Now we wish to integrate on $[0,t]$, $t\leq T$ small, use the estimates
(\ref{intdFm1est}), (\ref{intdFm2est}) in Lemma \ref{lemdFm} and then apply the
Gronwall inequality. However, in order to do this we have to absorb the terms with critical weights
into the LHS. Let $\varepsilon>0$, $\mathcal{E}_0$ be sufficiently small such that
\begin{align*}
10\varepsilon+\frac{C}{\varepsilon}\widetilde{C}\mathcal{E}_0<\frac{1}{4}
\end{align*}
and $\alpha$, $\sigma$ be appropriately large satisfying
\begin{align*}
\frac{1}{4}\alpha>\frac{C}{\varepsilon}&&\frac{1}{4}\sigma>\frac{C}{\varepsilon}\alpha.
\end{align*}
Then, integrating on a small time interval $[0,t]$ and
invoking (\ref{intdFm1est}), (\ref{intdFm2est}) for
$u$ equal to $d\eta^{m+1}$ and $d\xi^{m+1}$, $\tilde{\varepsilon}=1$ respectively,
we deduce
\begin{align}\label{sumdetadxi2}
&\frac{1}{2}(\|d\eta^{m+1}\|^2_{L^2_\alpha}
+\|d\xi^{m+1}\|^2_{L^2_\alpha})\\
\notag&+\frac{\alpha\sigma}{2}\int^t_0\big(\|\frac{d\eta^{m+1}}{\ell^{\alpha+1}}\|^2_{L^2(s_{min},s_0)}
+\|\frac{d\xi^{m+1}}{\ell^{\alpha+1}}\|^2_{L^2(s_{min},s_0)}\big)d\tau
+\frac{1}{4}\int^t_0\|\frac{d\xi^{m+1}_s}{\ell^\alpha}\|^2_{L^2}d\tau\\
\notag\leq&\;C\int^t_0\big[\alpha^2+\widetilde{C}\mathcal{E}_0A^2(t)\big](\|d\eta^{m+1}\|^2_{L^2_\alpha}
+\|d\xi^{m+1}\|^2_{L^2_\alpha})d\tau\\
\notag&+C\bigg(\sigma\int^T_0A^2(t)dt+\widetilde{C}\mathcal{E}_0\bigg)
\mathcal{E}(d\eta^m,d\xi^m;T)+C(\widetilde{C}\mathcal{E}_0)^2\sqrt{T}
\mathcal{E}(d\eta^{m-1},d\xi^{m-1};T)
\end{align}
Thus, by Gronwall's inequality and (\ref{intA(t)})
for $t\in[0,T]$, $T>0$ sufficiently small, such that
\begin{align}\label{intA(t)small}
\sigma\int^T_0A^2(t)dt\leq\widetilde{C}\mathcal{E}_0&&\alpha^2T<1
\end{align}
we conclude
\begin{align}\label{Linftydetadxi}
\operatornamewithlimits{ess\,sup}_{t\in[0,T]}(\|d\eta^{m+1}\|^2_{L^2_\alpha}
+&\|d\xi^{m+1}\|^2_{L^2_\alpha})\\
\notag\leq&\;C\widetilde{C}\mathcal{E}_0\big[\mathcal{E}(d\xi^m,d\eta^m;T)
+\widetilde{C}\mathcal{E}_0\sqrt{T}\cdot\mathcal{E}(d\xi^{m-1},d\eta^{m-1};T)\big]
\end{align}
and
\begin{align}\label{L2L2detadxi}
\alpha\sigma\int^T_0\big(\|\frac{d\eta^{m+1}}{\ell^{\alpha+1}}\|^2_{L^2(s_{min},s_0)}
+&\|\frac{d\xi^{m+1}}{\ell^{\alpha+1}}\|^2_{L^2(s_{min},s_0)}\big)d\tau\\
\notag\leq&\;C\widetilde{C}\mathcal{E}_0\big[\mathcal{E}(d\xi^m,d\eta^m;T)
+\widetilde{C}\mathcal{E}_0\sqrt{T}\cdot\mathcal{E}(d\xi^{m-1},d\eta^{m-1};T)\big].
\end{align}
\subsection{The $L^2_{\alpha-1}$ estimates for $d\eta^{m+1}_s,d\xi^{m+1}_s$}\label{L2deta_sdxi_s}
\noindent

Here we seek estimates for the $s$-derivatives of the terms $d\xi^{m+1},d\eta^{m+1}$.
We remark that the estimate (\ref{L2L2detadxi})  we just derived will be essential in absorbing
certain terms that appear in this subsection to close the estimates.
In particular, it is essential that the terms $\int_0^T \| d\eta^{m+1}\|^2_{L^2_{\alpha+1}}+\| d\xi^{m+1}\|^2_{L^2_{\alpha+1}}dt$
with critical weights appearing in the RHS of (\ref{sumdeta_sdxi_s2}) have
{\it already} been controlled in the prior subsection by the energies in the RHS
of the above estimate (\ref{L2L2detadxi}).
\newline

Recall (\ref{l_t}), (\ref{partial_tds}) and (\ref{[s,t]})
to obtain (for a.e. $0\leq t\leq T$)
\begin{align}\label{L2deta_s}
\notag\frac{1}{2}\frac{d}{dt}\|d\eta^{m+1}_s\|^2_{L^2_{\alpha-1}}
=&\int\frac{d\eta^{m+1}_s\partial_td\eta^{m+1}_s}{\ell^{2\alpha-2}}ds
-(\alpha-1)\int\frac{|d\eta^{m+1}_s|^2}{\ell^{2\alpha-1}}\partial_t\ell ds
+\frac{1}{2}\int\frac{|d\eta^{m+1}_s|^2}{\ell^{2\alpha-2}}\partial_tds\\
\leq&\int\frac{d\eta^{m+1}_s\partial_sd\eta^{m+1}_t}{\ell^{2\alpha-2}}ds
-(\alpha-1)\sigma\|\frac{d\eta^{m+1}_s}{\ell^\alpha}\|^2_{L^2(s_{min},s_0)}\\
\notag&+C(\alpha-1)\|\frac{d\eta^{m+1}_s}{\ell^\alpha}\|^2_{L^2(s_{min},s_0)}
+C(\alpha-1)\|\frac{d\eta^{m+1}_s}{\ell^{\alpha-1}}\|^2_{L^2}
+C\|\frac{d\eta^{m+1}_s}{s\ell^{\alpha-1}}\|^2_{L^2}
\end{align}
As usual, from (\ref{l/s})
\begin{align*}
\|\frac{d\eta^{m+1}_s}{s\ell^{\alpha-1}}\|^2_{L^2}\leq
C\sigma\|\frac{d\eta^{m+1}_s}{\ell^\alpha}\|^2_{L^2(s_{min},s_0)}
+C\|\frac{d\eta^{m+1}_s}{\ell^{\alpha-1}}\|^2_{L^2(s_0,+\infty)},
\end{align*}
since $\ell=O(1)$, $s\in(s_0,+\infty)$, {\bf HC}; see (\ref{l}), (\ref{ints2}).
In order to estimate the term
\begin{align}\label{partial_sdeta_t}
\bullet\;\;\;\int\frac{d\eta^{m+1}_s}{\ell^{2\alpha-2}}\partial_sd\eta^{m+1}_tds
\end{align}
we plug in the RHS of the equation of $d\eta^{m+1}_t$ (\ref{dpde2})
and treat each arising term separately. For all three of the subsequent terms we
use the estimates on the coefficients (\ref{coeff}), (\ref{coeff_s}),
(\ref{coeffinfty}), (\ref{coeff_sinfty}) and (\ref{A(t)}).
\begin{align*}
\tag{\ref{partial_sdeta_t}a}
&\int\frac{d\eta^{m+1}_s}{\ell^{2\alpha-2}}\partial_s\bigg[
\frac{\psi_s^2}{\psi^2}Bd\xi^{m+1}
\bigg]ds\\
\tag{$B\in L^\infty(s)$ (\ref{pB,D})}
=&\int\frac{d\eta^{m+1}_s}{\ell^{2\alpha-2}}\bigg[
\partial_s(\frac{\psi_s^2}{\psi^2})Bd\xi^{m+1}
+\frac{\psi_s^2}{\psi^2}B_sd\xi^{m+1}
+\frac{\psi_s^2}{\psi^2}Bd\xi^{m+1}_s
\bigg]ds\\
\leq&\;\|\frac{d\eta^{m+1}_s}{s\ell^{\alpha-1}}\|^2_{L^2}
+C\|\frac{d\xi^{m+1}}{s^2\ell^{\alpha-1}}\|^2_{L^2}
+\|\frac{d\eta^{m+1}_s}{s\ell^{\alpha-1}}\|^2_{L^2}
+C\|\frac{d\xi^{m+1}}{s}\|^2_{L^\infty(s)}\|\frac{B_s}{\ell^{\alpha-1}}\|^2_{L^2}\\
&+\|\frac{d\eta^{m+1}_s}{s\ell^{\alpha-1}}\|^2_{L^2}
+C\|\frac{d\xi^{m+1}_s}{s\ell^{\alpha-1}}\|^2_{L^2}\\
\tag{employing (\ref{l/s})}
\leq&\;C\sigma\|\frac{d\eta^{m+1}_s}{\ell^\alpha}\|^2_{L^2(s_{min},s_0)}
+C\|\frac{d\eta^{m+1}_s}{\ell^{\alpha-1}}\|^2_{L^2(s_0,+\infty)}\\
\tag{$\sigma>1$}
&+C\sigma^2\|\frac{d\xi^{m+1}}{\ell^{\alpha+1}}\|^2_{L^2(s_{min},s_0)}
+C\|\frac{d\xi^{m+1}}{\ell^\alpha}\|^2_{L^2(s_0,+\infty)}
+C\sigma\|\frac{d\xi^{m+1}_s}{\ell^\alpha}\|^2_{L^2(s_{min},s_0)}\\
\tag{by the $L^\infty$ estimate (\ref{pdxi}) and (\ref{B_s,D_s}) for $B_s$, $\widetilde{C}\mathcal{E}_0<1$}
&+C\|\frac{d\xi^{m+1}_s}{\ell^{\alpha-1}}\|^2_{L^2(s_0,+\infty)}
\end{align*}
Similarly, we obtain
\begin{align*}
\tag{\ref{partial_sdeta_t}b}
&\int\frac{d\eta^{m+1}_s}{\ell^{2\alpha-2}}\partial_s\bigg[
\frac{\psi_s}{\psi}Bd\xi^{m+1}_s
\bigg]ds\\
=&\int\frac{d\eta^{m+1}_s}{\ell^{2\alpha-2}}\bigg[
\partial_s(\frac{\psi_s}{\psi})Bd\xi^{m+1}_s
+\frac{\psi_s}{\psi}B_sd\xi^{m+1}_s
+\frac{\psi_s}{\psi}Bd\xi^{m+1}_{ss}
\bigg]ds\\
\leq&\;\|\frac{d\eta^{m+1}_s}{s\ell^{\alpha-1}}\|^2_{L^2}
+C\|\frac{d\xi^{m+1}_s}{s\ell^{\alpha-1}}\|^2_{L^2}
+\|\frac{d\eta^{m+1}_s}{s\ell^{\alpha-1}}\|^2_{L^2}
+C\|d\xi^{m+1}_s\|^2_{L^\infty(s)}\|\frac{B_s}{\ell^{\alpha-1}}\|^2_{L^2}\\
\tag{recall the $L^\infty$ estimate (\ref{pdxi_s})}
&+\varepsilon\|\frac{d\xi^{m+1}_{ss}}{\ell^{\alpha-1}}\|^2_{L^2}
+\frac{C}{\varepsilon}\|\frac{d\eta^{m+1}_s}{s\ell^{\alpha-1}}\|^2_{L^2}\\
\leq&\;\frac{C}{\varepsilon}\sigma\|\frac{d\eta^{m+1}_s}{\ell^\alpha}\|^2_{L^2(s_{min},s_0)}
+\frac{C}{\varepsilon}\|\frac{d\eta^{m+1}_s}{\ell^{\alpha-1}}\|^2_{L^2(s_0,+\infty)}
\tag{by (\ref{l/s}), $0<\varepsilon<1$}\\
&+C\sigma\|\frac{d\xi^{m+1}_s}{\ell^\alpha}\|^2_{L^2(s_{min},s_0)}
+C\|\frac{d\xi^{m+1}_s}{\ell^{\alpha-1}}\|^2_{L^2(s_0,+\infty)}
+(\varepsilon+C\widetilde{C}\mathcal{E}_0)\|\frac{d\xi^{m+1}_{ss}}{\ell^{\alpha-1}}\|^2_{L^2}.
\end{align*}
Continuing analogously,
we derive
\begin{align*}
\tag{\ref{partial_sdeta_t}c}
&\int\frac{d\eta^{m+1}_s}{\ell^{2\alpha-2}}\partial_s\bigg[
\frac{A}{s}Dd\eta^{m+1}+
2n(n-1)\frac{\psi_s^2}{\psi^2}d\eta^{m+1}
\bigg]ds\\
=&\int\frac{d\eta^{m+1}_s}{\ell^{2\alpha-2}}\bigg[
\partial_s(\frac{A}{s})Dd\eta^{m+1}
+\frac{A}{s}D_sd\eta^{m+1}
+\frac{A}{s}d\eta^{m+1}_s
+2n(n-1)\partial_s(\frac{\psi_s^2}{\psi^2})d\eta^{m+1}\\
\tag{recall the estimates on $D$ (\ref{pB,D}), (\ref{B_s,D_s})}
&+2n(n-1)\frac{\psi_s^2}{\psi^2}d\eta^{m+1}_s
\bigg]ds\\
\leq&\;\|\frac{d\eta^{m+1}_s}{s\ell^{\alpha-1}}\|^2_{L^2}
+C\widetilde{C}\mathcal{E}_0A^2(t)\|\frac{d\eta^{m+1}}{s\ell^{\alpha-1}}\|^2_{L^2}
+\|\frac{d\eta^{m+1}_s}{s\ell^{\alpha-1}}\|^2_{L^2}
+CA^2(t)\|d\eta^{m+1}\|^2_{L^\infty(s)}\|\frac{D_s}{\ell^{\alpha-1}}\|^2_{L^2}\\
&+\|\frac{d\eta^{m+1}_s}{s\ell^{\alpha-1}}\|^2_{L^2}
+CA^2(t)\|\frac{d\eta^{m+1}_s}{\ell^{\alpha-1}}\|^2_{L^2}
+\|\frac{d\eta^{m+1}_s}{s\ell^{\alpha-1}}\|^2_{L^2}
+C\|\frac{d\eta^{m+1}}{s^2\ell^{\alpha-1}}\|^2_{L^2}
+C\|\frac{d\eta^{m+1}_s}{s\ell^{\alpha-1}}\|^2_{L^2}\\
\tag{by (\ref{l/s})}
\leq&\;C\sigma\|\frac{d\eta^{m+1}_s}{\ell^\alpha}\|^2_{L^2(s_{min},s_0)}
+C\|\frac{d\eta^{m+1}_s}{\ell^{\alpha-1}}\|^2_{L^2(s_0,+\infty)}
+C\sigma^2\|\frac{d\eta^{m+1}}{\ell^{\alpha+1}}\|^2_{L^2(s_{min},s_0)}\\
&C\|\frac{d\eta^{m+1}}{\ell^\alpha}\|^2_{L^2(s_0,+\infty)}
+\sigma C\widetilde{C}\mathcal{E}_0A^2(t)\|\frac{d\eta^{m+1}}{\ell^\alpha}\|^2_{L^2}
+CA^2(t)\|\frac{d\eta^{m+1}_s}{\ell^{\alpha-1}}\|^2_{L^2}\\
\tag{utilizing the $L^\infty$ estimate (\ref{pdeta})}
&+C\widetilde{C}\mathcal{E}_0A^2(t)
\bigg[\|\frac{d\eta^{m+1}}{\ell^\alpha}\|^2_{L^2}+\|\frac{d\eta^{m+1}_s}{\ell^{\alpha-1}}\|^2_{L^2}\\
&+T\int^t_0\|d\xi^{m+1}\|^2_{H^2_{\alpha+1}}d\tau
+\widetilde{C}\mathcal{E}_0\sqrt{T}\int^T_0\|d\xi^m\|^2_{H^2_{\alpha+1}}d\tau\bigg]
\end{align*}
Combining (\ref{partial_sdeta_t}a)-(\ref{partial_sdeta_t}c) with the first equation
of (\ref{dpde2}), we obtain the
following estimate for the term (\ref{partial_sdeta_t}).
\begin{align}\label{partial_sdeta_test}
&\int\frac{d\eta^{m+1}_s}{\ell^{2\alpha-2}}\partial_sd\eta^{m+1}_tds\\
\notag\leq&\;\frac{C}{\varepsilon}\sigma\|\frac{d\eta^{m+1}_s}{\ell^\alpha}\|^2_{L^2(s_{min},s_0)}
+C\sigma\|\frac{d\xi^{m+1}_s}{\ell^\alpha}\|^2_{L^2(s_{min},s_0)}
+C\sigma^2\|\frac{d\eta^{m+1}}{\ell^{\alpha+1}}\|^2_{L^2(s_{min},s_0)}\\
&\notag+C\sigma^2\|\frac{d\xi^{m+1}}{\ell^{\alpha+1}}\|^2_{L^2(s_{min},s_0)}
+C\big(\frac{1}{\varepsilon}+A^2(t)\big)\|\frac{d\eta^{m+1}_s}{\ell^{\alpha-1}}\|^2_{L^2}
+\big(\sigma C\widetilde{C}\mathcal{E}_0A^2(t)+C\big)\|\frac{d\eta^{m+1}}{\ell^\alpha}\|^2_{L^2}\\
&\notag+C\big(\|\frac{d\xi^{m+1}}{\ell^\alpha}\|^2_{L^2}+
\|\frac{d\xi^{m+1}_s}{\ell^{\alpha-1}}\|^2_{L^2}\big)
+(\varepsilon+C\widetilde{C}\mathcal{E}_0)\|\frac{d\xi^{m+1}_{ss}}{\ell^{\alpha-1}}\|^2_{L^2}
+\int\frac{d\eta^{m+1}_s\cdot\partial_s(dF^1_m)}{\ell^{2\alpha-2}}ds\\
\notag&+C\widetilde{C}\mathcal{E}_0A^2(t)\bigg[T\int^t_0\|d\xi^{m+1}\|^2_{H^2_{\alpha+1}}d\tau
+\widetilde{C}\mathcal{E}_0\sqrt{T}\int^T_0\|d\xi^m\|^2_{H^2_{\alpha+1}}d\tau\bigg]
\end{align}
(The term $\int\frac{d\eta^{m+1}_s\cdot\partial_s(dF^1_m)}{\ell^{2\alpha-2}}ds$ will be controlled
below by using the
estimate (\ref{intdFm1_sest}) from Lemma \ref{lemdFm} above, putting $u=d\eta^{m+1}_s$).

We proceed to derive estimates for $\|d\xi^{m+1}_s\|^2_{L^2_{\alpha-1}}$.
Similarly to (\ref{L2deta_s}), using
in addition (\ref{itbdc}), (\ref{smaxdecay}) we have
\begin{align}\label{L2dxi_s}
\notag&\frac{1}{2}\frac{d}{dt}\|d\xi^{m+1}_s\|^2_{L^2_{\alpha-1}}\\
\notag=&\int\frac{d\xi^{m+1}_s}{\ell^{2\alpha-2}}(\partial_sd\xi^{m+1}_t
-n\frac{\psi_{ss}}{\psi}d\xi^{m+1}_s)ds
-(\alpha-1)\int\frac{|d\xi^{m+1}_s|^2}{\ell^{2\alpha-1}}\partial_t\ell ds
+\frac{1}{2}\int\frac{|d\xi^{m+1}_s|^2}{\ell^{2\alpha-2}}\partial_tds\\
\leq&-\int\frac{d\xi^{m+1}_{ss}d\xi^{m+1}_t}{\ell^{2\alpha-2}}ds
+(2\alpha-2)\int\frac{d\xi^{m+1}_sd\xi^{m+1}_t}{\ell^{2\alpha-1}}\ell_s ds
+C\|\frac{d\xi^{m+1}_s}{s\ell^{\alpha-1}}\|^2_{L^2}\\
&-(\alpha-1)\sigma\|\frac{d\xi^{m+1}_s}{\ell^\alpha}\|^2_{L^2(s_{min},s_0)}
\notag+C(\alpha-1)\|\frac{d\xi^{m+1}_s}{\ell^\alpha}\|^2_{L^2}
+C\|\frac{d\xi^{m+1}_s}{s\ell^{\alpha-1}}\|^2_{L^2}
\end{align}
There are essentially
two terms we must estimate in this case. For both of them,
we use the estimates on the coefficients (\ref{coeff}), (\ref{coeffinfty}).
Recall that by (\ref{l_s}) $\partial_s\ell=O(1)$, $t\in[0,T]$. Therefore, we start
first with the term
\begin{align*}
\tag{from the RHS of the equation of $d\xi^{m+1}_t$ (\ref{dpde2})}
\bullet&\;(2\alpha-2)\int\frac{d\xi^{m+1}_sd\xi^{m+1}_t}{\ell^{2\alpha-1}}\ell_s ds\\
=&\;(2\alpha-2)\int\frac{d\xi^{m+1}_s}{\ell^{2\alpha-1}}\ell_s
\bigg[(\frac{\psi_{ss}}{\psi}+(n-1)\frac{\psi_s^2}{\psi^2})
(Bd\eta^{m+1}+Bd\xi^{m+1})
+\frac{\psi_{s}}{\psi}Bd\xi^{m+1}_s\\
\tag{recall the $L^\infty$ estimate on $B$}
&+|B|d\xi^{m+1}_{ss}+\frac{\psi_s}{\psi}Bd\eta^{m+1}_s+dF^m_2
\bigg]ds\\
\leq&\;\alpha^2\|\frac{d\xi^{m+1}_s}{\ell^\alpha}\|^2_{L^2}
+C\|\frac{d\eta^{m+1}}{s^2\ell^{\alpha-1}}\|^2_{L^2}
+\alpha^2\|\frac{d\xi^{m+1}_s}{\ell^\alpha}\|^2_{L^2}
+C\|\frac{d\xi^{m+1}}{s^2\ell^\alpha}\|^2_{L^2}
+\alpha^2\|\frac{d\xi^{m+1}_s}{\ell^\alpha}\|^2_{L^2}\\
&+C\|\frac{d\xi^{m+1}_s}{s\ell^{\alpha-1}}\|^2_{L^2}
+\varepsilon\|\frac{d\xi^{m+1}_{ss}}{\ell^{\alpha-1}}\|^2_{L^2}
+\frac{C}{\varepsilon}\alpha^2\|\frac{d\xi^{m+1}_s}{\ell^\alpha}\|^2_{L^2}
+\alpha^2\|\frac{d\xi^{m+1}_s}{\ell^\alpha}\|^2_{L^2}
+C\|\frac{d\eta^{m+1}_s}{s\ell^{\alpha-1}}\|^2_{L^2}\\
\tag{employ (\ref{l/s}), $0<\varepsilon<1$; see (\ref{ints2})}
&+(2\alpha-2)\int\frac{d\xi^{m+1}_sdF^m_1}{\ell^{2\alpha-1}}\ell_s ds\\
\leq&\;\frac{C}{\varepsilon}\alpha^2\|\frac{d\xi^{m+1}_s}{\ell^\alpha}\|^2_{L^2(s_{min},s_0)}
+\frac{C}{\varepsilon}\alpha^2\|\frac{d\xi^{m+1}_s}{\ell^{\alpha-1}}\|^2_{L^2(s_0,+\infty)}
+C\sigma^2\|\frac{d\eta^{m+1}}{\ell^{\alpha+1}}\|^2_{L^2(s_{min},s_0)}\\
&+C\|\frac{d\eta^{m+1}}{\ell^\alpha}\|^2_{L^2(s_0,+\infty)}
+C\sigma^2\|\frac{d\xi^{m+1}}{\ell^{\alpha+1}}\|^2_{L^2(s_{min},s_0)}
+C\|\frac{d\xi^{m+1}}{\ell^\alpha}\|^2_{L^2(s_0,+\infty)}\\
&+C\sigma\|\frac{d\xi^{m+1}_s}{\ell^\alpha}\|^2_{L^2(s_{min},s_0)}
+C\|\frac{d\xi^{m+1}_s}{\ell^{\alpha-1}}\|^2_{L^2(s_0,+\infty)}
+C\sigma\|\frac{d\eta^{m+1}_s}{\ell^\alpha}\|^2_{L^2(s_{min},s_0)}\\
&+C\|\frac{d\eta^{m+1}_s}{\ell^{\alpha-1}}\|^2_{L^2(s_0,+\infty)}
+\varepsilon\|\frac{d\xi^{m+1}_{ss}}{\ell^{\alpha-1}}\|^2_{L^2}
+(2\alpha-2)\int\frac{d\xi^{m+1}_sdF^m_1}{\ell^{2\alpha-1}}\ell_s ds
\end{align*}
We finish the estimates for the $L^2_{\alpha-1}$ norm of $d\xi^{m+1}_s$
by estimating in a similar manner the first term in the RHS of (\ref{L2dxi_s}). (We remark here that since
the term $d\xi^{m+1}_{ss}$ must necessarily be multiplied by the weight $\frac{1}{\ell^{\alpha-1}}$
in the $L^2$ norms below, the extra powers of $s$ arising in the RHS of
$d\xi^{m+1}_t$ (\ref{dpde2}) must all be absorbed in the terms $d\xi^{m+1}, d\eta^{m+1}$ with at most one derivative;
it is at this point that we make essential use of the estimates on the terms
$d\xi^{m+1}, d\eta^{m+1}$ with critical weights that have {\it already} been established
in \S\ref{L2detadxi}).
\begin{align*}
\tag{plugging in the RHS of the equation of $d\xi^{m+1}_t$ (\ref{dpde2})}
\bullet&\;-\int\frac{d\xi^{m+1}_{ss}d\xi^{m+1}_t}{\ell^{2\alpha-2}}ds\\
=&-\int\frac{d\xi^{m+1}_{ss}}{\ell^{2\alpha-2}}
\bigg[(\frac{\psi_{ss}}{\psi}+(n-1)\frac{\psi_s^2}{\psi^2})
(Bd\eta^{m+1}+Bd\xi^{m+1})+\frac{\psi_{s}}{\psi}Bd\xi^{m+1}_s\\
&+|B|d\xi^{m+1}_{ss}+\frac{\psi_s}{\psi}Bd\eta^{m+1}_s+dF^m_2
\bigg]ds\\
\leq&\;\varepsilon\|\frac{d\xi^{m+1}_{ss}}{\ell^{\alpha-1}}\|^2_{L^2}
+\frac{C}{\varepsilon}\|\frac{d\xi^{m+1}}{s^2\ell^{\alpha-1}}\|^2_{L^2}
+\varepsilon\|\frac{d\xi^{m+1}_{ss}}{\ell^{\alpha-1}}\|^2_{L^2}
+\frac{C}{\varepsilon}\|\frac{d\eta^{m+1}}{s^2\ell^{\alpha-1}}\|^2_{L^2}
+\varepsilon\|\frac{d\xi^{m+1}_{ss}}{\ell^{\alpha-1}}\|^2_{L^2}\\
&+\frac{C}{\varepsilon}\|\frac{d\xi^{m+1}_s}{s\ell^{\alpha-1}}\|^2_{L^2}
-\int\frac{|B||d\xi^{m+1}_{ss}|^2}{\ell^{2\alpha-2}}ds
+\varepsilon\|\frac{d\xi^{m+1}_{ss}}{\ell^{\alpha-1}}\|^2_{L^2}
+\frac{C}{\varepsilon}\|\frac{d\eta^{m+1}_s}{s\ell^{\alpha-1}}\|^2_{L^2}
-\int\frac{d\xi^{m+1}_{ss}dF^m_2}{\ell^{\alpha-2}}ds\\
\leq&\;4\varepsilon\|\frac{d\xi^{m+1}_{ss}}{\ell^{\alpha-1}}\|^2_{L^2}
+\frac{C}{\varepsilon}\sigma^2\|\frac{d\xi^{m+1}}{\ell^{\alpha+1}}\|^2_{L^2(s_{min},s_0)}
+\frac{C}{\varepsilon}\|\frac{d\xi^{m+1}}{\ell^\alpha}\|^2_{L^2(s_0,+\infty)}\\
&+\frac{C}{\varepsilon}\sigma^2\|\frac{d\eta^{m+1}}{\ell^{\alpha+1}}\|^2_{L^2(s_{min},s_0)}
+\frac{C}{\varepsilon}\|\frac{d\eta^{m+1}}{\ell^\alpha}\|^2_{L^2(s_0,+\infty)}
+\frac{C}{\varepsilon}\sigma\|\frac{d\xi^{m+1}_s}{\ell^\alpha}\|^2_{L^2(s_{min},s_0)}\\
&+\frac{C}{\varepsilon}\|\frac{d\xi^{m+1}_s}{\ell^{\alpha-1}}\|^2_{L^2(s_0,+\infty)}
+\frac{C}{\varepsilon}\sigma\|\frac{d\eta^{m+1}_s}{\ell^\alpha}\|^2_{L^2(s_{min},s_0)}
+\frac{C}{\varepsilon}\|\frac{d\eta^{m+1}_s}{\ell^{\alpha-1}}\|^2_{L^2(s_0,+\infty)}\\
&-\frac{1}{2}\|\frac{d\xi^{m+1}_{ss}}{\ell^{\alpha-1}}\|^2_{L^2}
-\int\frac{d\xi^{m+1}_{ss}dF^m_2}{\ell^{\alpha-2}}ds
\tag{by (\ref{l/s}) and (\ref{pB,D}), $c=\frac{1}{2}$, $0<\varepsilon<1$}
\end{align*}

\textbf{Summary}: Adding (\ref{L2deta_s}) to (\ref{L2dxi_s}) and
combining (\ref{partial_sdeta_test}) with the estimates of the
two terms above we deduce (for a.e. $0\leq t\leq T$)
\begin{align}\label{sumdeta_sdxi_s1}
&\notag\frac{1}{2}\frac{d}{dt}(\|d\eta^{m+1}_s\|^2_{L^2_{\alpha-1}}
+\|d\xi^{m+1}_s\|^2_{L^2_{\alpha-1}})\\
&+(\alpha-1)\sigma\big(\|\frac{d\eta^{m+1}_s}{\ell^\alpha}\|^2_{L^2(s_{min},s_0)}
+\|\frac{d\xi^{m+1}_s}{\ell^\alpha}\|^2_{L^2(s_{min},s_0)}\big)\\
\notag\leq&\;C\big(\frac{\alpha^2}{\varepsilon}+A^2(t)\big)(\|d\eta^{m+1}_s\|^2_{L^2_{\alpha-1}}
+\|d\xi^{m+1}_s\|^2_{L^2_{\alpha-1}})
+C\big(\frac{1}{\varepsilon}+\sigma\widetilde{C}\mathcal{E}_0A^2(t)\big)
(\|d\eta^{m+1}\|^2_{L^2_\alpha}+\|d\xi^{m+1}\|^2_{L^2_\alpha})\\
\notag&+\frac{C}{\varepsilon}(\alpha^2+\sigma)\big(\|\frac{d\eta^{m+1}_s}{\ell^\alpha}\|^2_{L^2(s_{min},s_0)}
+\|\frac{d\xi^{m+1}_s}{\ell^\alpha}\|^2_{L^2(s_{min},s_0)}\big)
+C\widetilde{C}\mathcal{E}_0A^2(t)T\int^t_0\|d\xi^{m+1}\|^2_{H^2_{\alpha+1}}d\tau\\
\notag&+(C\widetilde{C}\mathcal{E}_0+10\varepsilon-\frac{1}{2})\|\frac{d\xi^{m+1}_{ss}}{\ell^{\alpha-1}}\|^2_{L^2}
+\frac{C}{\varepsilon}\sigma^2\big(\|\frac{d\eta^{m+1}}{\ell^{\alpha+1}}\|^2_{L^2(s_{min},s_0)}
+\|\frac{d\xi^{m+1}}{\ell^{\alpha+1}}\|^2_{L^2(s_{min},s_0)}\big)\\
\notag&+\int\frac{d\eta^{m+1}_s\partial_s(dF^m_1)}{\ell^{2\alpha-2}}ds
+(2\alpha-2)\int\frac{d\xi^{m+1}_sdF^m_2}{\ell^{2\alpha-1}}\ell_s ds
-\int\frac{d\xi^{m+1}_{ss}dF^m_2}{\ell^{2\alpha-2}}ds\\
&\notag+C(\widetilde{C}\mathcal{E}_0)^2A^2(t)\sqrt{T}\int^T_0\|d\xi^m\|^2_{H^2_{\alpha+1}}d\tau
\end{align}
We next integrate the above  with respect to time
to close the estimate, using Lemma \ref{lemdFm} and the estimates derived in \S\ref{L2detadxi}.
For this purpose, we need to choose certain parameters to be small enough:

Apply (\ref{intdFm2est}) with $u=-\ell^2d\xi^{m+1}$ and let
$\varepsilon,\tilde{\varepsilon},\mathcal{E}_0$ be small such that
\begin{align*}
C\widetilde{C}\mathcal{E}_0+10\varepsilon+\tilde{\varepsilon}+C\widetilde{C}\mathcal{E}_0T\int^t_0A^2(\tau)d\tau<\frac{1}{4}
\end{align*}
and $\alpha,\sigma$ appropriately large satisfying
\begin{align*}
\frac{1}{4}(\alpha-1)>\frac{C}{\varepsilon}&&\frac{1}{4}\sigma>\frac{\alpha^2}{\alpha-1}\frac{C}{\varepsilon},
\end{align*}
in accordance with the considerations after the derivation of (\ref{sumdetadxi1}).
Then, (\ref{sumdeta_sdxi_s1}) gives
\begin{align}\label{sumdeta_sdxi_s2}
&\frac{1}{2}(\|d\eta^{m+1}_s\|^2_{L^2_{\alpha-1}}
+\|d\xi^{m+1}_s\|^2_{L^2_{\alpha-1}})\\
\notag&+\frac{(\alpha-1)\sigma}{2}\int^t_0\big(\|\frac{d\eta^{m+1}_s}{\ell^\alpha}\|^2_{L^2(s_{min},s_0)}
+\|\frac{d\xi^{m+1}_s}{\ell^\alpha}\|^2_{L^2(s_{min},s_0)}\big)d\tau
+\frac{1}{4}\int^t_0\|\frac{d\xi^{m+1}_{ss}}{\ell^{\alpha-1}}\|^2_{L^2}d\tau\\
\notag\leq&\;C\int^t_0\big(\alpha^2+A^2(t)\big)(\|d\eta^{m+1}_s\|^2_{L^2_{\alpha-1}}
+\|d\xi^{m+1}_s\|^2_{L^2_{\alpha-1}})d\tau\\
&\notag+C\big(T+\sigma\widetilde{C}\mathcal{E}_0\int^T_0A^2(\tau)d\tau\big)
\operatornamewithlimits{ess\,sup}_{\tau\in[0,T]}
(\|d\eta^{m+1}\|^2_{L^2_\alpha}+\|d\xi^{m+1}\|^2_{L^2_\alpha})\\
&\notag+C\sigma^2\int^T_0\big(\|\frac{d\eta^{m+1}}{\ell^{\alpha+1}}\|^2_{L^2(s_{min},s_0)}
+\|\frac{d\xi^{m+1}}{\ell^{\alpha+1}}\|^2_{L^2(s_{min},s_0)}\big)d\tau
+\int^t_0\int\frac{d\eta^{m+1}_s\partial_s(dF^m_1)}{\ell^{2\alpha-2}}dsd\tau\\
&+(2\alpha-2)\int^t_0\int\frac{d\xi^{m+1}_sdF^m_2}{\ell^{2\alpha-1}}\ell_s dsd\tau
\notag+C(\widetilde{C}\mathcal{E}_0)^2\sqrt{T}\int^T_0A^2(\tau)d\tau\int^T_0\|d\xi^m\|^2_{H^2_{\alpha+1}}d\tau\\
&\notag+C\bigg(\sigma\int^T_0A^2(\tau)d\tau+\widetilde{C}\mathcal{E}_0\bigg)
\mathcal{E}(d\eta^m,d\xi^m;T)
+C(\widetilde{C}\mathcal{E}_0)^2\sqrt{T}
\int^T_0\|d\xi^{m-1}\|^2_{H^2_{\alpha+1}}d\tau
\end{align}
We invoke 1) the estimates in the previous subsection (\ref{Linftydetadxi}), (\ref{L2L2detadxi}),
for $T>0$ within (\ref{intA(t)small}),
2) (\ref{intdFm1_sest}) for $u=d\eta^{m+1}_s$ and 3) (\ref{intdFm2est}) setting $u=\ell_s\ell d\xi^{m+1}_s$,
$\tilde{\varepsilon}=\alpha$ to derive for $\alpha,\sigma$ large enough)
\begin{align}\label{sumdeta_sdxi_s3}
&\frac{1}{2}(\|d\eta^{m+1}_s\|^2_{L^2_{\alpha-1}}
+\|d\xi^{m+1}_s\|^2_{L^2_{\alpha-1}})\\
\notag&+\frac{(\alpha-1)\sigma}{4}\int^t_0\big(\|\frac{d\eta^{m+1}_s}{\ell^\alpha}\|^2_{L^2(s_{min},s_0)}
+\|\frac{d\xi^{m+1}_s}{\ell^\alpha}\|^2_{L^2(s_{min},s_0)}\big)d\tau
+\frac{1}{4}\int^t_0\|\frac{d\xi^{m+1}_{ss}}{\ell^{\alpha-1}}\|^2_{L^2}d\tau\\
\notag\leq&\;C\int^t_0\bigg[\alpha^2+A^2(t)
+\frac{\sqrt{\widetilde{C}\mathcal{E}_0}}{\varepsilon_1}\big(
\|\frac{\xi^m_{ss}}{\ell^{\alpha-1}}\|_{L^2}+\|\frac{\xi^{m-1}_{ss}}{\ell^{\alpha-1}}\|_{L^2}\big)\bigg]
(\|d\eta^{m+1}_s\|^2_{L^2_{\alpha-1}}
+\|d\xi^{m+1}_s\|^2_{L^2_{\alpha-1}})d\tau\\
&\notag+C\int^t_0\big(\|\frac{\xi^m_{ss}}{\ell^{\alpha-1}}\|_{L^2}+\|\frac{\xi^{m-1}_{ss}}{\ell^{\alpha-1}}\|_{L^2}\big)
\|d\xi^m\|_{H^2_{\alpha+1}}\|d\eta^{m+1}_s\|_{L^2_{\alpha-1}}d\tau\\
&\notag+(C\sigma\widetilde{C}\mathcal{E}_0+\varepsilon_1)
\mathcal{E}(d\eta^m,d\xi^m;T)+C\sigma(\widetilde{C}\mathcal{E}_0)^2\sqrt{T}\mathcal{E}(d\eta^{m-1},d\xi^{m-1};T).
\end{align}
Setting
\begin{align}\label{H_i}
H_1(\tau)=\alpha^2+A^2(t)+\frac{\sqrt{\widetilde{C}\mathcal{E}_0}}{\varepsilon_1}\big(\|\frac{\xi^m_{ss}}{\ell^{\alpha-1}}\|_{L^2}
+&\|\frac{\xi^{m-1}_{ss}}{\ell^{\alpha-1}}\|_{L^2}\big),\\
\notag H_2(\tau)=&\;\big(\|\frac{\xi^m_{ss}}{\ell^{\alpha-1}}\|_{L^2}
+\|\frac{\xi^{m-1}_{ss}}{\ell^{\alpha-1}}\|_{L^2}\big)
\|\frac{d\xi^m_{ss}}{\ell^{\alpha-1}}\|_{L^2},
\end{align}
(\ref{sumdeta_sdxi_s3}) yields
\begin{align}\label{sumdeta_sdxi_s4}
&\frac{1}{2}(\|d\eta^{m+1}_s\|^2_{L^2_{\alpha-1}}
+\|d\xi^{m+1}_s\|^2_{L^2_{\alpha-1}})\\
\notag&+\frac{(\alpha-1)\sigma}{4}\int^t_0\big(\|\frac{d\eta^{m+1}_s}{\ell^\alpha}\|^2_{L^2(s_{min},s_0)}
+\|\frac{d\xi^{m+1}_s}{\ell^\alpha}\|^2_{L^2(s_{min},s_0)}\big)d\tau
+\frac{1}{4}\int^t_0\|\frac{d\xi^{m+1}_{ss}}{\ell^{\alpha-1}}\|^2_{L^2}d\tau\\
\notag\leq&\;C\int^t_0H_1(\tau)(\|d\eta^{m+1}_s\|^2_{L^2_{\alpha-1}}
+\|d\xi^{m+1}_s\|^2_{L^2_{\alpha-1}})d\tau
+C\int^t_0H_2(\tau)\big(\|d\eta^{m+1}_s\|^2_{L^2_{\alpha-1}}+\|d\xi^{m+1}_s\|^2_{L^2_{\alpha-1}}\big)^{\frac{1}{2}}d\tau\\
&\notag+(C\sigma\widetilde{C}\mathcal{E}_0+\varepsilon_1)
\mathcal{E}(d\eta^m,d\xi^m;T)+C\sigma(\widetilde{C}\mathcal{E}_0)^2\sqrt{T}\mathcal{E}(d\eta^{m-1},d\xi^{m-1};T).
\end{align}
It is clear now that we can use some Gronwall type of inequality to get an estimate
from (\ref{sumdeta_sdxi_s4}).
\begin{lemma}\label{extGron}
Let $f:[a,b]\to\R$ be a continuous function which satisfies:
\begin{align*}
\frac{1}{2}f^2(t)\leq\frac{1}{2}f_0^2+\int^t_a\Psi(\tau)f(\tau)d\tau,\;\;\;t\in[a,b],
\end{align*}
where $f_0\in\R$ and $\Psi$ nonnegative continuous in $[a,b]$. Then the estimate
\begin{align*}
\frac{1}{2}|f(t)|\leq\frac{1}{2}|f_0|+\int^t_a\Psi(\tau)d\tau,\;\;\;t\in[a,b]
\end{align*}
holds.
\end{lemma}
Applying the preceding lemma to (\ref{sumdeta_sdxi_s2}) for
\begin{align*}
f^2=\|d\eta^{m+1}_s\|^2_{L^2_{\alpha-1}}
+&\|d\xi^{m+1}_s\|^2_{L^2_{\alpha-1}}\\
\frac{1}{2}f_0^2=&\;(C\sigma\widetilde{C}\mathcal{E}_0+\varepsilon_1)\mathcal{E}(d\eta^m,d\xi^m;T)
+C\sigma(\widetilde{C}\mathcal{E}_0)^2\sqrt{T}\mathcal{E}(d\eta^{m-1},d\xi^{m-1};T)
\end{align*}
and
\begin{align*}
\Psi(\tau)=CH_1(\tau)\big(\|d\eta^{m+1}_s\|^2_{L^2_{\alpha-1}}
+\|d\xi^{m+1}_s\|^2_{L^2_{\alpha-1}}\big)^{\frac{1}{2}}+CH_2(\tau),
\end{align*}
varying $t$ in $[0,T]$, we derive
\begin{align}\label{deta_sdxi_sest1}
\big(\|d\eta^{m+1}_s\|^2_{L^2_{\alpha-1}}
+&\|d\xi^{m+1}_s\|^2_{L^2_{\alpha-1}}\big)^{\frac{1}{2}}\\
\notag\leq&\;C\int^t_0H_1(\tau)\big(\|d\eta^{m+1}_s\|^2_{L^2_{\alpha-1}}
+\|d\xi^{m+1}_s\|^2_{L^2_{\alpha-1}}\big)^{\frac{1}{2}}d\tau+C\int^T_0H_2(\tau)d\tau\\
\notag&+(C\sigma\widetilde{C}\mathcal{E}_0+2\varepsilon_1)^\frac{1}{2}\sqrt{\mathcal{E}(d\eta^m,d\xi^m;T)}
+(C\sigma\sqrt{T})^\frac{1}{2}\widetilde{C}\mathcal{E}_0\sqrt{\mathcal{E}(d\eta^{m-1},d\xi^{m-1};T)}.
\end{align}
By (\ref{H_i}) and (\ref{itenest}), (\ref{intA(t)}
we readily obtain
\begin{align}\label{H_iest}
\notag\int^T_0H_1(\tau)d\tau=&\int^T_0\bigg[\alpha^2+A^2(t)+\frac{\sqrt{\widetilde{C}\mathcal{E}_0}}{\varepsilon_1}
\big(\|\frac{\xi^m_{ss}}{\ell^{\alpha-1}}\|_{L^2}
+\|\frac{\xi^{m-1}_{ss}}{\ell^{\alpha-1}}\|_{L^2}\big)
\bigg]d\tau\leq o(1)+2\sqrt{2}\frac{\widetilde{C}\mathcal{E}_0}{\varepsilon_1}\sqrt{T}\\
\int^T_0H_2(\tau)d\tau=&\int^T_0\big(\|\frac{\xi^m_{ss}}{\ell^{\alpha-1}}\|_{L^2}
+\|\frac{\xi^{m-1}_{ss}}{\ell^{\alpha-1}}\|_{L^2}\big)\|\frac{d\xi^m_{ss}}{\ell^{\alpha-1}}\|_{L^2}d\tau\\
\notag\leq&\;\bigg(\int^T_02\|\frac{\xi^m_{ss}}{\ell^{\alpha-1}}\|^2_{L^2}
+2\|\frac{\xi^{m-1}_{ss}}{\ell^{\alpha-1}}\|^2_{L^2}d\tau\bigg)^{\frac{1}{2}}
\bigg(\int^T_0\|\frac{d\xi^m_{ss}}{\ell^{\alpha-1}}\|^2_{L^2}d\tau\bigg)^{\frac{1}{2}}\\
\notag\leq&\;2\sqrt{2}\sqrt{\widetilde{C}\mathcal{E}_0}\cdot\sqrt{\mathcal{E}(d\eta^m,d\xi^m;T)}
\end{align}
Thus, from (\ref{deta_sdxi_sest1}) and the standard Gronwall inequality it
follows that
\begin{align}\label{deta_sdxi_sest2}
\operatornamewithlimits{ess\,sup}_{t\in[0,T]}\big(\|d\eta^{m+1}_s\|^2_{L^2_{\alpha-1}}
+&\|d\xi^{m+1}_s\|^2_{L^2_{\alpha-1}}\big)^{\frac{1}{2}}\\
\leq&\;
\notag\exp\{o(1)+C\sqrt{T}\frac{\widetilde{C}\mathcal{E}_0}{\varepsilon_1}\}\big[
(C\sigma\widetilde{C}\mathcal{E}_0+2\varepsilon_1)^{\frac{1}{2}}\sqrt{\mathcal{E}(d\eta^m,d\xi^m;T)}\\
&\notag+(C\sigma\sqrt{T})^{\frac{1}{2}}\widetilde{C}\mathcal{E}_0\sqrt{\mathcal{E}(d\eta^{m-1},d\xi^{m-1};T)}\big]
\end{align}
Finally, going back to (\ref{sumdeta_sdxi_s4}) and employing
(\ref{H_iest}), (\ref{deta_sdxi_sest2}) we derive
\begin{align}\label{L2L2deta_sdxi_s}
(\alpha-1)\sigma&\int^T_0\big(\|\frac{d\eta^{m+1}_s}{\ell^\alpha}\|^2_{L^2(s_{min},s_0)}
+\|\frac{d\xi^{m+1}_s}{\ell^\alpha}\|^2_{L^2(s_{min},s_0)}\big)d\tau\\
\notag\leq&\;\big(o(1)+C\frac{\widetilde{C}\mathcal{E}_0}{\varepsilon_1}+C\big)
\exp\{o(1)+C\sqrt{T}\frac{\widetilde{C}\mathcal{E}_0}{\varepsilon_1}\}
\big[(\sigma\widetilde{C}\mathcal{E}_0+\varepsilon_1)\mathcal{E}(d\eta^m,d\xi^m;T)\\
&\notag+\sigma\sqrt{T}(\widetilde{C}\mathcal{E}_0)^2\mathcal{E}(d\eta^{m-1},d\xi^{m-1};T)\big]
\end{align}
and
\begin{align}\label{L2L2dxi_ss}
&\int^T_0\|\frac{d\xi^{m+1}_{ss}}{\ell^{\alpha-1}}\|^2_{L^2}d\tau\\
\notag\leq&\;\big(o(1)+C\frac{\widetilde{C}\mathcal{E}_0}{\varepsilon_1}+C\big)
\exp\{o(1)+C\sqrt{T}\frac{\widetilde{C}\mathcal{E}_0}{\varepsilon_1}\}
\big[(\sigma\widetilde{C}\mathcal{E}_0+\varepsilon_1)\mathcal{E}(d\eta^m,d\xi^m;T)\\
&\notag+\sigma\sqrt{T}(\widetilde{C}\mathcal{E}_0)^2\mathcal{E}(d\eta^{m-1},d\xi^{m-1};T)\big].
\end{align}
\subsection{Proof of Theorem \ref{gensol}}\label{gensolproof}
\noindent

Combining the estimates (\ref{Linftydetadxi}), (\ref{L2L2detadxi}),
and (\ref{deta_sdxi_sest2}), (\ref{L2L2deta_sdxi_s}), (\ref{L2L2dxi_ss})
we conclude that there exist appropriately small $\mathcal{E}_0,\varepsilon_1$ such that
the relation
\begin{align}\label{contr}
\sqrt{\mathcal{E}(d\eta^{m+1},d\xi^{m+1};T)}\leq
\kappa\sqrt{\mathcal{E}(d\eta^m,d\xi^m;T)}
+\kappa\sqrt{\mathcal{E}(d\eta^{m-1},d\xi^{m-1};T)}
\end{align}
holds, for some $0<\kappa<\frac{1}{4}$.
Consequently, the sequence
$\big\{\eta^m,\xi^m\big\}^\infty_{m=0}$ is Cauchy with respect to the norm
$\sqrt{\mathcal{E}(\cdot,\cdot;T)}$. It is clear that the latter
suffices for convergence of the iterates
to a solution $\eta,\xi$ of (\ref{pde}), as required in Theorem \ref{gensol}; see (\ref{energy}) and
Proposition \ref{itsol}. The energy estimates directly imply the uniqueness of the solution in
the function spaces under consideration.

This completes the proof of Theorem \ref{gensol}.$\qquad\blacksquare$

\section{The Linear step in the iteration: Proof of Proposition \ref{itsol}}\label{modprob}
\noindent

In the beginning of \S\ref{genprob} we took for granted that at each step $m+1$, $m=0,1,\ldots$, the
linear system (\ref{itpde}) possessed a (unique) solution with prescribed regularity
and energy bounds, as formulated in Proposition \ref{itsol}. We prove these assertions in this section.
The proof will follow from the study of a general type
of such linear systems. The strategy for deriving
the desired energy estimates in this subsection follows in large part the one in \S\ref{genprob}. Indeed,
as we shall see below many derivations are similar and in fact simpler than
the ones
in the previous subsection.
Whenever this is the case, we will cite the relevant argument in \S\ref{genprob}
for the sake of brevity.
\begin{definition}\label{f,F}
We introduce generic notation of functions
\begin{align}\label{f,Freg}
f,g\in L^\infty(0,T;H^1(s))&& F_1\in L^2(0,T;H^1(s)),\;\;F_2\in L^2(0,T;L^2(s))
\end{align}
satisfying (for a.e. $t$)
\begin{align}\label{f}
\frac{1}{2}+\|f\|_{L^\infty(s)}<\|f\|_{L^\infty(s)}+g(s,t)<C&&
\|\frac{f_s}{\ell^{\alpha-1}}\|^2_{L^2}
+\|\frac{g_s}{\ell^{\alpha-1}}\|^2_{L^2}<\tilde{\varepsilon}<1,
\end{align}
for appropriate positive constants $c,C$, $\tilde{\varepsilon}$ small, and
\begin{align}\label{F}
\int^T_0\|\frac{F_i}{\ell^{\alpha-1}}&\|^2_{L^2}dt<+\infty,\qquad i=1,2\\
\notag&\int\frac{u\cdot\partial_sF_1}{\ell^{2\alpha-2}}ds\leq
C\sigma\|\frac{u}{\ell^\alpha}\|^2_{L^2(s_{min},s_0)}+G_1(t)\|\frac{u}{\ell^{\alpha-1}}\|^2_{L^2}
+G_2(t),
\end{align}
for the general function $u$ (a.e. $0\leq t\leq T$), where $G_1(t),G_2(t)$ are positive, $t$-integrable functions.
\end{definition}

Motivated by (\ref{itpde}), we consider the following (more general, as we will see) linear system.
\begin{align}\label{lpde}
\notag\eta_t=&\;\frac{\psi^2_s}{\psi^2}f\xi+\frac{\psi_s}{\psi}f\xi_s
-\frac{A}{s}f\eta+2n(n-1)\frac{\psi^2_s}{\psi^2}\eta+F_1\\
\xi_t=&\;(\frac{\psi_{ss}}{{\psi}}+(n-1)\frac{\psi^2_s}{\psi^2})f\cdot(\eta+\xi)
+\frac{\psi_s}{\psi}f\xi_s+g\xi_{ss}+\frac{\psi_s}{\psi}f\eta_s+F_2\\
\notag&\eta\bigg|_{t=0}=\eta_0\qquad\xi\bigg|_{t=0}=\xi_0,
\qquad\qquad\xi\bigg|_{(s_{min},t)}=\big(\xi\bigg|_{(s_0,t)},\;{\bf G}\big)=0
\qquad t\geq0.
\end{align}
We summarize our goal with the following theorem.
\begin{theorem}\label{modsol}
Assume
\begin{align*}
\eta_0\in H^1_{\alpha}&&\xi_0\in H^1_{\alpha,0},
\end{align*}
Then, for some small $T>0$ there exist $\alpha,\sigma$ sufficiently
large such that (\ref{lpde}) has a unique solution up to time $T>0$,
in the sense of Definition (\ref{sol}); in particular
\begin{align}\label{modsp}
\eta\in L^\infty(0,T;H^1_\alpha)\cap L^2(0,T;H^1_{\alpha+1})&&
\xi\in L^\infty(0,T;H^1_{\alpha,0})\cap L^2(0,T;H^2_{\alpha+1})\\
\notag\eta_t\in L^2(0,T;L^2_\alpha)&&\xi_t\in L^2(0,T;L^2_{\alpha-1})
\end{align}
Further, the solution satisfies the energy estimate
\begin{align}\label{modenest}
\mathcal{E}(\eta,\xi;T)
\leq\widetilde{C}\bigg[\mathcal{E}_0+\sum\int^T_0\|\frac{F_i}{\ell^{\alpha-1}}\|^2_{L^2}dt
+\int^T_0G_2(t)dt\bigg],
\end{align}
for some positive constant $\widetilde{C}=O\big(\int^T_0G_1(t)dt\big)$.
%
%
\end{theorem}
Our first task is to show that the preceding theorem
actually implies Proposition \ref{itsol}.
Throughout the subsequent estimates, we
will use the symbol $C$ to denote a (sufficiently large) positive constant depending only on $n$.

\subsection{Proposition \ref{itsol} follows from Theorem \ref{modsol}}\label{itsolproof}
\noindent

By the induction hypothesis (see the beginning of the present subsection), we have at our disposal
the pointwise estimate (\ref{pm}) on the $m^{th}$ term of the sequence $(\eta^m,\xi^m)$.
Hence, for sufficiently small initial energy $\mathcal{E}_0$, we verify that
the system (\ref{itpde}) is of the form (\ref{lpde})
(see Definition \ref{f,F}) with
\begin{align}\label{F1}
F_1:=2n(n-1)\bigg[\frac{\psi^2_s}{\psi^2}
\frac{\sum^{2n}_{j=2}\binom{2n}{j}|\xi^m|^{j-2}}{(\xi^m+1)^{2n}}|\xi^m|^2
-\frac{A}{s}\xi^m\frac{\xi^m+2}{(\xi^m+1)^2}-\frac{|\xi^m_s|^2}{(\xi^m+1)^{2n+2}}\bigg]
\end{align}
and
\begin{align}\label{F2}
F_2:=&-(\frac{\psi_{ss}}{\psi}+(n-1)\frac{\psi^2_s}{\psi^2})
\frac{\sum^{2n}_{j=2}\binom{2n}{j}|\xi^m|^{j-2}}{(\xi^m+1)^{2n-1}}|\xi^m|^2
+(n-1)\frac{A}{s}\xi^m\frac{\xi^m+2}{\xi^m+1}\\
&-\frac{|\xi^m_s|^2}{(\eta^m+1)(\xi^m+1)^{2n+1}}
\notag-\frac{1}{2}\frac{\eta^m_s\xi^m_s}{(\eta^m+1)^2(\xi^m+1)^{2n}}.
\end{align}
Indeed, similarly to the derivations of the estimates in Lemma \ref{lemdFm}, using
the bounds  on the coefficients (\ref{coeff}), (\ref{coeff_s}), (\ref{coeffinfty}),
(\ref{coeff_sinfty}), (\ref{A(t)}), the $L^\infty$ estimates (\ref{pm}), (\ref{pm_s}) and
the comparison estimate (\ref{l/s}) we can derive the following estimates (for
a.e. $t\in[0,T]$, $\widetilde{C}\mathcal{E}_0<1$ and the general function $u$):
\begin{align}\label{F12est}
\|\frac{F_1}{\ell^{\alpha-1}}\|^2_{L^2}
\leq&\;
\notag C\bigg[\sigma^2\widetilde{C}\mathcal{E}_0\|\frac{\xi^m}{\ell^\alpha}\|^2_{L^2}
+\sigma A^2(t)\|\frac{\xi^m}{\ell^\alpha}\|^2_{L^2}
+\sqrt{\widetilde{C}\mathcal{E}_0}\|\frac{\xi^m_{ss}}{\ell^{\alpha-1}}\|_{L^2}
\|\frac{\xi^m_s}{\ell^{\alpha-1}}\|^2_{L^2}\bigg]\\
\|\frac{F_2}{\ell^{\alpha-1}}\|^2_{L^2}
\leq&\;C\bigg[\sigma^2\widetilde{C}\mathcal{E}_0\|\frac{\xi^m}{\ell^\alpha}\|^2_{L^2}
+\sigma A^2(t)\|\frac{\xi^m}{\ell^\alpha}\|^2_{L^2}\\
\notag&+\sqrt{\widetilde{C}\mathcal{E}_0}\|\frac{\xi^m_{ss}}{\ell^{\alpha-1}}\|_{L^2}
\big(\|\frac{\xi^m_s}{\ell^{\alpha-1}}\|^2_{L^2}
+\|\frac{\eta^m_s}{\ell^{\alpha-1}}\|^2_{L^2}\big)\bigg]
\end{align}
and
\begin{align}\label{partial_sF1est}
\notag\int\frac{u\cdot\partial_sF_1}{\ell^{2\alpha-2}}ds\leq&\;
C\sigma\|\frac{u}{\ell^\alpha}\|^2_{L^2(s_{min},s_0)}
+C\|\frac{u}{\ell^{\alpha-1}}\|^2_{L^2}
+\frac{C}{\varepsilon}\sqrt{\widetilde{C}\mathcal{E}_0}\|\frac{\xi^m_{ss}}{\ell^\alpha}\|_{L^2}
\|\frac{u}{\ell^{\alpha-1}}\|^2_{L^2}\\
&+C\big(\sigma^2\widetilde{C}\mathcal{E}_0+\sigma A^2(t)\big)\|\frac{\xi^m}{\ell^\alpha}\|^2_{L^2}\\
\notag&+C\big(\sigma\widetilde{C}\mathcal{E}_0+A^2(t)
+\sqrt{\widetilde{C}\mathcal{E}_0}\|\frac{\xi^m_{ss}}{\ell^{\alpha-1}}\|_{L^2}\big)
\|\frac{\xi^m_s}{\ell^{\alpha-1}}\|^2_{L^2}
+\varepsilon\|\frac{\xi^m_{ss}}{\ell^{\alpha-1}}\|^2_{L^2},
\end{align}
$0<\varepsilon<1$.
Integrating up to time $T$ and applying (\ref{itenest}), (\ref{intA(t)}) we confirm the first equation in (\ref{F}):
\begin{align}\label{Fest}
\sum\int^T_0\|\frac{F_i}{\ell^{\alpha-1}}\|^2_{L^2}dt
\leq C\big[\sigma^2(\widetilde{C}\mathcal{E}_0)^2T+\sigma\widetilde{C}\mathcal{E}_0o(1)
+(\widetilde{C}\mathcal{E}_0)^2\sqrt{T}\big],&&
\text{for $T$ sufficiently small.}
\end{align}
and
\begin{align}\label{G1,G2est}
\int^T_0G_1(t)dt&\leq CT+\frac{C}{\varepsilon}\widetilde{C}\mathcal{E}_0\sqrt{T},\\
\notag&\int^T_0G_2(t)dt\leq
C\big[\sigma^2(\widetilde{C}\mathcal{E}_0)^2T+\sigma\widetilde{C}\mathcal{E}_0o(1)
+(\widetilde{C}\mathcal{E}_0)^2\sqrt{T}\big]+\varepsilon\widetilde{C}\mathcal{E}_0\qquad
(\sigma>1).
\end{align}
Thus, the conclusion of Theorem \ref{modsol} yields the desired solution of (\ref{itpde})
 satisfying the assertions in Proposition \ref{itsol} and the energy estimate
\begin{align}\label{verenest}
\notag\mathcal{E}(\eta^{m+1},\xi^{m+1};T)\leq&\;
\widetilde{C}\bigg[\mathcal{E}_0+\sum\int^T_0\|\frac{F_i}{\ell^{\alpha-1}}\|^2_{L^2}dt
+\int^T_0G_2(t)dt\bigg]\\
\leq&\;\widetilde{C}\bigg[\mathcal{E}_0
+o(1)\widetilde{C}\mathcal{E}_0+\varepsilon\widetilde{C}\mathcal{E}_0\bigg]
\leq2\widetilde{C}\mathcal{E}_0,
\end{align}
provided $\varepsilon,T$ are appropriately small.

%

%
\subsection{Plan of the proof of Theorem \ref{modsol}: A new iteration}\label{modsolproof}
\noindent

Consider the system
\begin{align}\label{modlpde}
\notag\eta_t=&\;\frac{\psi^2_s}{\psi^2}f\tilde{\xi}+\frac{\psi_s}{\psi}f\tilde{\xi}_s
+\frac{A}{s}f\eta+2n(n-1)\frac{\psi^2_s}{\psi^2}\eta+F_1\\
\xi_t=&\;(\frac{\psi_{ss}}{{\psi}}+(n-1)\frac{\psi^2_s}{\psi^2})f\cdot(\eta+\xi)
+\frac{f}{\psi^2}\xi+\frac{\psi_s}{\psi}f\xi_s+g\xi_{ss}+\frac{\psi_s}{\psi}f\eta_s+F_2\\
\notag&\eta\bigg|_{t=0}=\eta_0\qquad\xi\bigg|_{t=0}=\xi_0,
\qquad\qquad\xi\bigg|_{(s_{min},t)}=\big(\xi\bigg|_{(s_0,t)},\;{\bf G}\big)=0
\qquad t\geq0,
\end{align}
where
\begin{align}\label{tildexisp}
\tilde{\xi}\in L^\infty(0,T;H^1_{\alpha,0})\cap L^2(0,T;H^2_{\alpha+1})
\end{align}
is a given function satisfying
\begin{align}\label{tildexiest1}
\|\tilde{\xi}\|^2_{L^\infty(0,T;H^1_{\alpha,0})}+\|\tilde{\xi}\|^2_{L^2(0,T;H^2_{\alpha+1})}
\leq\widetilde{C}\bigg[\mathcal{E}_0
+\sum\int^T_0\|\frac{F_i}{\ell^{\alpha-1}}\|^2_{L^2}dt+\int^T_0G_2(t)dt\bigg],
\end{align}
for some positive constant $\widetilde{C}$ to be determined later.
In addition, we assume the estimates
\begin{align}\label{tildexiest2}
\int^T_0\|\frac{\tilde{\xi}}{\ell^{\alpha+1}}\|^2_{L^2(s_{min,s_0})}dt
\leq\frac{\widetilde{C}}{\alpha\sigma^2}\bigg[\mathcal{E}_0
+\sum\int^T_0\|\frac{F_i}{\ell^{\alpha-1}}\|^2_{L^2}dt+\int^T_0G_2(t)dt\bigg]
\end{align}
and
\begin{align}\label{tildexiest3}
\int^T_0\|\frac{\tilde{\xi}_s}{\ell^\alpha}\|^2_{L^2(s_{min,s_0})}dt
\leq\frac{\widetilde{C}}{(\alpha-1)\sigma}\bigg[\mathcal{E}_0
+\sum\int^T_0\|\frac{F_i}{\ell^{\alpha-1}}\|^2_{L^2}dt+\int^T_0G_2(t)dt\bigg].
\end{align}
We {\it claim} that for appropriately large $\alpha,\sigma$ the preceding system has a unique solution
\begin{align}\label{modetaxisp}
\eta\in L^\infty(0,T;H^1_\alpha)\cap L^2(0,T;H^1_{\alpha+1})
&&\xi\in L^\infty(0,T;H^1_{\alpha,0})\cap L^2(0,T;H^2_{\alpha+1})\\
\notag\eta_t\in L^2(0,T;L^2_\alpha)&&\xi_t\in L^2(0,T;L^2_{\alpha-1}),
\end{align}
which satisfies the energy estimates
\begin{align}\label{modetaxiest}
\mathcal{E}(\eta,\xi;T)
\leq\widetilde{C}\bigg[\mathcal{E}_0+\sum\int^T_0\|\frac{F_i}{\ell^{\alpha-1}}\|^2_{L^2}dt
+\int^T_0G_2(t)dt\bigg]
\end{align}
and
\begin{align}\label{modxiest1}
\int^T_0\|\frac{\xi}{\ell^{\alpha+1}}\|^2_{L^2(s_{min,s_0})}dt
\leq\frac{\widetilde{C}}{\alpha\sigma^2}\bigg[\mathcal{E}_0
+\sum\int^T_0\|\frac{F_i}{\ell^{\alpha-1}}\|^2_{L^2}dt+\int^T_0G_2(t)dt\bigg],
\end{align}
\begin{align}\label{modxiest2}
\int^T_0\|\frac{\xi_s}{\ell^\alpha}\|^2_{L^2(s_{min,s_0})}dt
\leq\frac{\widetilde{C}}{(\alpha-1)\sigma}\bigg[\mathcal{E}_0
+\sum\int^T_0\|\frac{F_i}{\ell^{\alpha-1}}\|^2_{L^2}dt+\int^T_0G_2(t)dt\bigg].
\end{align}
Observe that if we can prove this, a standard iteration argument (passing to a subsequence, weak limits etc.) yields
a solution $\eta,\xi$ of the original linear  problem (\ref{lpde}) in the same space
(\ref{modetaxisp}) and satisfying the same estimates as above.
%
%
This reduces the proof of Theorem
 \ref{modsol} to proving  our claim above.
$\qquad\blacksquare$

\subsection{Apriori estimates for $\eta$}\label{modetaest}
\noindent

We proceed to derive estimates for the $L^2_\alpha$ norm of $\eta$.
Utilizing (\ref{partial_tds}), (\ref{l_t}) we derive
\begin{align}\label{L2eta}
\notag\frac{1}{2}\frac{d}{dt}\|\eta\|^2_{L^2_\alpha}
=&\int\frac{\eta\eta_t}{\ell^{2\alpha}}ds
-\alpha\int\frac{\eta^2}{\ell^{2\alpha+1}}\partial_t\ell ds
+\frac{1}{2}\int\frac{\eta^2}{\ell^{2\alpha}}\partial_tds\\
\leq&\int\frac{\eta\eta_t}{\ell^{2\alpha}}ds
-\alpha\sigma\|\frac{\eta}{\ell^{\alpha+1}}\|^2_{L^2(s_{min},s_0)}
+C\alpha\|\frac{\eta}{\ell^{\alpha+1}}\|^2_{L^2(s_{min},s_0)}\\
&\notag+C\alpha\|\frac{\eta}{\ell^\alpha}\|^2_{L^2}
+C\|\frac{\eta}{s\ell^\alpha}\|^2_{L^2}
\end{align}
First, employing (\ref{l/s}) we estimate the last term
\begin{align*}
\|\frac{\eta}{s\ell^\alpha}\|^2_{L^2}\leq
C\sigma\|\frac{\eta}{\ell^{\alpha+1}}\|^2_{L^2(s_{min},s_0)}
+C\|\frac{\eta}{\ell^\alpha}\|^2_{L^2(s_0,+\infty)},
\end{align*}
since $\ell=O(1)$, when $x\in(x_0,+\infty)$; see Definition \ref{weight} and (\ref{ints1}), (\ref{ints2}).
Recall the estimates of the coefficients (\ref{coeff}), (\ref{coeffinfty}), (\ref{A(t)}) .
Going back to the first equation of (\ref{modlpde}) and plugging the RHS into the integral below we obtain
\begin{align*}
\int\frac{\eta\eta_t}{\ell^{2\alpha}}ds=&
\int\frac{\eta}{\ell^{2\alpha}}\bigg[\frac{\psi^2_s}{\psi^2}f\tilde{\xi}
+\frac{\psi_s}{\psi}f\tilde{\xi}_s+\frac{A}{s}f\eta+2n(n-1)\frac{\psi^2_s}{\psi^2}\eta+F_1\bigg]ds\\
\tag{by the pointwise bound on $f$ (\ref{f})}
\leq&\;C\|\frac{\eta}{s\ell^\alpha}\|^2_{L^2}
+\|\frac{\tilde{\xi}}{s\ell^\alpha}\|^2_{L^2}
+\|\frac{\tilde{\xi}_s}{\ell^\alpha}\|^2_{L^2}\\
&+A^2(t)\|\frac{\eta}{\ell^\alpha}\|^2_{L^2}
+\|\frac{F_1}{\ell^{\alpha-1}}\|^2_{L^2}\\
\leq&\;C\sigma\|\frac{\eta}{\ell^{\alpha+1}}\|^2_{L^2(s_{min},s_0)}
+C\|\frac{\eta}{\ell^\alpha}\|^2_{L^2(s_0,+\infty)}
+C\sigma\|\frac{\tilde{\xi}}{\ell^{\alpha+1}}\|^2_{L^2(s_{min},s_0)}\\
\tag{using the comparison estimate (\ref{l/s}) once more}
&+C\|\frac{\tilde{\xi}}{\ell^\alpha}\|^2_{L^2(s_0,+\infty)}\\
&+C\|\frac{\tilde{\xi}_s}{\ell^\alpha}\|^2_{L^2}
+A^2(t)\|\frac{\eta}{\ell^\alpha}\|^2_{L^2}
+\|\frac{F_1}{\ell^{\alpha-1}}\|^2_{L^2},
\end{align*}
The last two estimates and (\ref{L2eta}) yield
\begin{align}\label{summodeta1}
&\frac{1}{2}\frac{d}{dt}\|\eta\|^2_{L^2_\alpha}
+\alpha\sigma\|\frac{\eta}{\ell^{\alpha+1}}\|^2_{L^2(s_{min},s_0)}\\
\notag\leq&\;C(\sigma+\alpha)\|\frac{\eta}{\ell^{\alpha+1}}\|^2_{L^2(s_{min},s_0)}
+\big(C\alpha+A^2(t)\big)\|\eta\|^2_{L^2_\alpha}
+C\sigma\|\frac{\tilde{\xi}}{\ell^{\alpha+1}}\|^2_{L^2(s_{min},s_0)}\\
&\notag+C\|\frac{\tilde{\xi}}{\ell^\alpha}\|^2_{L^2}
+C\|\frac{\tilde{\xi}_s}{\ell^\alpha}\|^2_{L^2}
+\|\frac{F_1}{\ell^\alpha}\|^2_{L^2},
\end{align}
Let
\begin{align*}
\frac{1}{4}\alpha>C&&\frac{1}{4}\sigma>\alpha.
\end{align*}
By our assumptions on $\tilde{\xi}$ (\ref{tildexiest1}), (\ref{tildexiest2}), (\ref{tildexiest3}),
integrating on $[0,t]$, $t\leq T$, we deduce
\begin{align}\label{summodeta2}
&\frac{1}{2}\|\eta\|^2_{L^2_\alpha}
+\frac{\alpha\sigma}{2}\int^t_0\|\frac{\eta}{\ell^{\alpha+1}}\|^2_{L^2(s_{min},s_0)}d\tau\\
\notag\leq&\;\frac{1}{2}\|\eta_0\|^2_{L^2_\alpha}
+\int^t_0\big(C\alpha+A^2(t)\big)\|\eta\|^2_{L^2_\alpha}d\tau
+C\sigma\int^T_0\|\frac{\tilde{\xi}}{\ell^{\alpha+1}}\|^2_{L^2(s_{min},s_0)}d\tau\\
\notag&+C\int^T_0\|\frac{\tilde{\xi}_s}{\ell^\alpha}\|^2_{L^2(s_{min},s_0)}d\tau
+CT\operatornamewithlimits{ess\,sup}_{\tau\in[0,T]}\big(\|\frac{\tilde{\xi}}{\ell^\alpha}\|^2_{L^2}
+\|\frac{\tilde{\xi}_s}{\ell^{\alpha-1}}\|^2_{L^2}\big)
+\int^T_0\|\frac{F_1}{\ell^\alpha}\|^2_{L^2}d\tau\\
\notag\leq&\;\int^t_0\big(C\alpha+A^2(t)\big)\|\eta\|^2_{L^2_\alpha}d\tau
+C\frac{\widetilde{C}}{\alpha\sigma}\bigg[\mathcal{E}_0
+\sum\int^T_0\|\frac{F_i}{\ell^{\alpha-1}}\|^2_{L^2}d\tau
+\int^T_0G_2(\tau)d\tau\bigg],
\end{align}
provided $\widetilde{C}$ is large enough and $T$ small.
Thus, by Gronwall's inequality and (\ref{intA(t)}), $t\in[0,T]$, we conclude ($\alpha$ large)
\begin{align}\label{Linftymodeta}
\|\eta\|^2_{L^\infty(0,T;L^2_\alpha(s))}
\leq\frac{\widetilde{C}}{10}\bigg[\mathcal{E}_0
+\sum\int^T_0\|\frac{F_i}{\ell^{\alpha-1}}\|^2_{L^2}d\tau
+\int^T_0G_2(\tau)d\tau\bigg]
\end{align}
and
\begin{align}\label{L2L2modeta}
\int^T_0\|\frac{\eta}{\ell^{\alpha+1}}\|^2_{L^2(s_{min},s_0)}d\tau
\leq \frac{\widetilde{C}}{\alpha\sigma^2}\bigg[\mathcal{E}_0
+\sum\int^T_0\|\frac{F_i}{\ell^{\alpha-1}}\|^2_{L^2}d\tau
+\int^T_0G_2(\tau)d\tau\bigg].
\end{align}

Now we continue to the $L^2_{\alpha-1}$ energy estimates
of $\eta_s$.
We are going to need certain $L^\infty$ estimates.
Similarly to the derivation of (\ref{pm}), by the energy estimate of $\tilde{\xi}$
(\ref{tildexiest1})
we can obtain the estimates (for a.e. $t\in[0,T]$)
\begin{align}\label{ptildexi}
\|\frac{\tilde{\xi}}{\ell^k}\|^2_{L^\infty(s)}\leq (k+1)C\widetilde{C}
\bigg[\mathcal{E}_0+\sum\int^T_0\|\frac{F_i}{\ell^{\alpha-1}}\|^2_{L^2}dt
+\int^T_0G_2(t)dt\bigg]
\end{align}
and
\begin{align}\label{ptildexi_s}
\|\frac{\tilde{\xi}_s}{\ell^{k-1}}\|^2_{L^\infty(s)}
\leq C\|\frac{\tilde{\xi}_{ss}}{\ell^{\alpha-1}}\|^2_{L^2}
+kC\widetilde{C}
\bigg[\mathcal{E}_0+\sum\int^T_0\|\frac{F_i}{\ell^{\alpha-1}}\|^2_{L^2}dt
+\int^T_0G_2(t)dt\bigg],
\end{align}
for every $k\leq\alpha-1$. Moreover, the same argument combined with the above estimate
(\ref{Linftymodeta}) yields
\begin{align}\label{pmodeta}
\|\frac{\eta}{\ell^k}\|^2_{L^\infty(s)}\leq C\|\frac{\eta_s}{\ell^k}\|^2_{L^2}
+(k+1)C\widetilde{C}\bigg[\mathcal{E}_0+\sum\int^T_0\|\frac{F_i}{\ell^{\alpha-1}}\|^2_{L^2}dt
+\int^T_0G_2(t)dt\bigg],
\end{align}
provided $k\leq\alpha-1$. Employing
(\ref{l_t}), (\ref{[s,t]}) and (\ref{partial_tds}) we derive
\begin{align}\label{L2eta_s}
\notag\frac{1}{2}\frac{d}{dt}\|\eta_s\|^2_{L^2_{\alpha-1}}
=&\int\frac{\eta_s\partial_t\eta_s}{\ell^{2\alpha-2}}ds
-(\alpha-1)\int\frac{\eta^2_s}{\ell^{2\alpha-1}}\partial_t\ell ds
+\frac{1}{2}\int\frac{\eta^2}{\ell^{2\alpha-2}}\partial_tds\\
\leq&\int\frac{\eta_s\partial_s\eta_t}{\ell^{2\alpha-2}}ds
+C\|\frac{\eta_s}{s\ell^{\alpha-1}}\|^2_{L^2}
-(\alpha-1)\sigma\|\frac{\eta_s}{\ell^\alpha}\|^2_{L^2(s_{min},s_0)}\\
\notag&+C\alpha\|\frac{\eta_s}{\ell^\alpha}\|^2_{L^2(s_{min},s_0)}
+C\alpha\|\frac{\eta_s}{\ell^{\alpha-1}}\|^2_{L^2}
+C\|\frac{\eta_s}{s\ell^{\alpha-1}}\|^2_{L^2}
\end{align}
As before, by (\ref{l/s}) it follows
\begin{align*}
\|\frac{\eta_s}{s\ell^{\alpha-1}}\|^2_{L^2}\leq
C\sigma\|\frac{\eta_s}{\ell^\alpha}\|^2_{L^2(s_{min},s_0)}
+C\|\frac{\eta_s}{\ell^\alpha}\|^2_{L^2(s_0,+\infty)}
\end{align*}
In order to estimate the main term we recall additionally the estimates on
the derivatives of the coefficients (\ref{coeff_s}),
(\ref{coeff_sinfty}). We will skip the detailed derivations for the following estimate,
since it is of similar form (only simpler) to the one
carefuly claimed for the term (\ref{partial_sdeta_t}). In the same spirit we
deduce (for a.e. $t\in[0,T]$)
\begin{align*}
\tag{plugging in the RHS of the first equation (\ref{modlpde})}
&\int\frac{\eta_s\partial_s\eta_t}{\ell^{2\alpha-2}}ds\\
&\int\frac{\eta_s}{\ell^{2\alpha-2}}\partial_s\bigg[
\frac{\psi^2_s}{\psi^2}f\tilde{\xi}+\frac{\psi_s}{\psi}f\tilde{\xi}_s
+\frac{A}{s}f\eta+2n(n-1)\frac{\psi^2_s}{\psi^2}\eta+F_1
\bigg]\\
\tag{$0<\varepsilon<1$}
\leq&\;\frac{C}{\varepsilon}\|\frac{\eta_s}{s\ell^{\alpha-1}}\|^2_{L^2}
+\|\frac{\tilde{\xi}}{s^2\ell^{\alpha-1}}\|^2_{L^2}
+\|\frac{\tilde{\xi}}{s}\|^2_{L^\infty(s)}\|\frac{f_s}{\ell^{\alpha-1}}\|^2_{L^2}
+\|\frac{\tilde{\xi}_s}{s\ell^{\alpha-1}}\|^2_{L^2}\\
&+\varepsilon\|\tilde{\xi}_s\|^2_{L^\infty(s)}\|\frac{f_s}{\ell^{\alpha-1}}\|^2_{L^2}
+\varepsilon\|\frac{\tilde{\xi}_{ss}}{\ell^{\alpha-1}}\|^2_{L^2}
+A^2(t)\|\frac{\eta}{s\ell^{\alpha-1}}\|^2_{L^2}
+A^2(t)\|\eta\|^2_{L^\infty(s)}\|\frac{f_s}{\ell^{\alpha-1}}\|^2_{L^2}\\
&+A^2(t)\|\frac{\eta_s}{\ell^{\alpha-1}}\|^2_{L^2}
+\|\frac{\eta}{s^2\ell^{\alpha-1}}\|^2_{L^2}
+\int\frac{\eta_s\cdot\partial_sF_1}{\ell^{2\alpha-2}}ds
\end{align*}
Now we apply the comparison estimate (\ref{l/s}), the $L^\infty$ estimates
(\ref{ptildexi}), (\ref{ptildexi_s}), (\ref{pmodeta}) and the assumptions
(\ref{f}), (\ref{F}) in Definition \ref{f,F} to finally obtain from (\ref{L2eta_s})
\begin{align}\label{summodeta_s1}
&\frac{1}{2}\frac{d}{dt}\|\eta_s\|^2_{L^2_{\alpha-1}}
+(\alpha-1)\sigma\|\frac{\eta_s}{\ell^\alpha}\|^2_{L^2(s_{min},s_0)}\\
\notag\leq&\;(\frac{C}{\varepsilon}\sigma+C\alpha)\|\frac{\eta_s}{\ell^\alpha}\|^2_{L^2(s_{min},s_0)}
+\big[\frac{C}{\varepsilon}+C\alpha+CA^2(t)+G_1(t)\big]\|\eta_s\|^2_{L^2_{\alpha-1}}\\
&+C\sigma^2\|\frac{\eta}{\ell^{\alpha+1}}\|^2_{L^2(s_{min},s_0)}
\notag+\big[C+C\sigma A^2(t)\big]\|\eta\|^2_{L^2_\alpha}
+C\sigma^2\|\frac{\tilde{\xi}}{\ell^{\alpha+1}}\|^2_{L^2(s_{min},s_0)}
+C\|\frac{\tilde{\xi}}{\ell^\alpha}\|^2_{L^2}\\
&+C\sigma\|\frac{\tilde{\xi}_s}{\ell^\alpha}\|^2_{L^2(s_{min},s_0)}
\notag+C\|\frac{\tilde{\xi}_s}{\ell^{\alpha-1}}\|^2_{L^2}
+C\varepsilon\|\frac{\tilde{\xi}_{ss}}{\ell^{\alpha-1}}\|^2_{L^2}\\
&\notag+G_2(t)+\big[C\sigma+CA^2(t)\big]\widetilde{C}
\bigg[\mathcal{E}_0+\sum\int^T_0\|\frac{F_i}{\ell^{\alpha-1}}\|^2_{L^2}dt
+\int^T_0G_2(t)dt\bigg]
\end{align}
Let
\begin{align*}
\frac{1}{4}(\alpha-1)>\frac{C}{\varepsilon}&&\frac{1}{4}\sigma>C\frac{\alpha}{\alpha-1}.
\end{align*}
Then integrating on $[0,t]$, $t\leq T$, we derive
\begin{align}\label{summodeta_s2}
&\frac{1}{2}\|\eta_s\|^2_{L^2_{\alpha-1}}
+\frac{1}{2}(\alpha-1)\sigma\int^t_0\|\frac{\eta_s}{\ell^\alpha}\|^2_{L^2(s_{min},s_0)}d\tau\\
\notag\leq&\;\frac{1}{2}\|\partial_x\eta_0\|^2_{L^2_{\alpha-1}}
+\int^t_0\big[\frac{C}{\varepsilon}+C\alpha+CA^2(t)+G_1(t)\big]\|\eta_s\|^2_{L^2_{\alpha-1}}d\tau\\
\notag&+C\sigma^2\int^T_0\|\frac{\eta}{\ell^{\alpha+1}}\|^2_{L^2(s_{min},s_0)}d\tau
+\bigg(CT+C\int^T_0A^2(t)d\tau\bigg)\|\eta\|^2_{L^\infty(0,T;L^2_\alpha(s))}\\
&\notag+C\sigma^2\int^T_0\|\frac{\tilde{\xi}}{\ell^{\alpha+1}}\|^2_{L^2(s_{min},s_0)}d\tau
+C\sigma\int^T_0\|\frac{\tilde{\xi}_s}{\ell^\alpha}\|^2_{L^2(s_{min},s_0)}d\tau\\
&+CT\operatornamewithlimits{ess\,sup}_{\tau\in[0,T]}\big(\|\frac{\tilde{\xi}}{\ell^\alpha}\|^2_{L^2}
+\|\frac{\tilde{\xi}_s}{\ell^{\alpha-1}}\|^2_{L^2}\big)
\notag+C\varepsilon\int^T_0\|\frac{\tilde{\xi}_{ss}}{\ell^{\alpha-1}}\|^2_{L^2}d\tau
+\int^T_0G_2(\tau)d\tau\\
&\notag+\bigg(CT+C\int^T_0A^2(t)d\tau\bigg)\widetilde{C}
\bigg[\mathcal{E}_0+\sum\int^T_0\|\frac{F_i}{\ell^{\alpha-1}}\|^2_{L^2}d\tau
+\int^T_0G_2(\tau)d\tau\bigg]
\end{align}
Recall our assumptions
on $\tilde{\xi}$ (\ref{tildexiest1}), (\ref{tildexiest2}), (\ref{tildexiest3}).
Utilizing (\ref{intA(t)})
and the derived estimates (\ref{Linftymodeta}), (\ref{L2L2modeta}) we conclude that
\begin{align}\label{summodeta_s3}
&\frac{1}{2}\|\eta_s\|^2_{L^2_{\alpha-1}}
+\frac{1}{2}(\alpha-1)\sigma\int^t_0\|\frac{\eta_s}{\ell^\alpha}\|^2_{L^2(s_{min},s_0)}d\tau\\
\notag\leq&\;\frac{1}{2}\|\partial_x\eta_0\|^2_{L^2_{\alpha-1}}
+\int^t_0\big[\frac{C}{\varepsilon}+C\alpha+CA^2(t)+G_1(t)\big]\|\eta_s\|^2_{L^2_{\alpha-1}}d\tau\\
\notag&+\int^T_0G_2(\tau)d\tau
+\big(o(1)+C\varepsilon+\frac{C}{\alpha}\big)\widetilde{C}
\bigg[\mathcal{E}_0+\sum\int^T_0\|\frac{F_i}{\ell^{\alpha-1}}\|^2_{L^2}d\tau
+\int^T_0G_2(\tau)d\tau\bigg],
\end{align}
 Thus, for $\varepsilon,T$ appropriately small, $\widetilde{C}$
and $\alpha$ large, it follows from Gronwall's inequality that
\begin{align}\label{Linftymodeta_s}
\|\eta_s\|^2_{L^\infty(0,T;L^2_{\alpha-1}(s))}
\leq\frac{\widetilde{C}}{10}\bigg[\mathcal{E}_0
+\sum\int^T_0\|\frac{F_i}{\ell^{\alpha-1}}\|^2_{L^2}d\tau+\int^T_0G_2(\tau)d\tau\bigg]
\end{align}
and
\begin{align}\label{L2L2modeta_s}
\int^T_0\|\frac{\eta_s}{\ell^\alpha}\|^2_{L^2(s_{min},s_0)}d\tau
\leq\frac{\widetilde{C}}{(\alpha-1)\sigma}\bigg[\mathcal{E}_0
+\sum\int^T_0\|\frac{F_i}{\ell^{\alpha-1}}\|^2_{L^2}d\tau+\int^T_0G_2(\tau)d\tau\bigg].
\end{align}
\subsection{Energy estimates for $\xi$: A Galerkin-type argument}\label{Gal}
\noindent

Now that we have solved the first equation of (\ref{modlpde}) for $\eta$ and obtained the required
regularity  estimates for this solution, we plug it into the second equation of the system
(\ref{modlpde}). We initially seek a weak solution of (\ref{modlpde})
\begin{align*}
\xi\in L^\infty(0,T;L^2_\alpha(s))\cap L^2(0,T;H^1_{\alpha+1,0}(s))&&
\ell^2\xi_t\in L^2(0,T;H^{-1}_{\alpha+1}(s)),
\end{align*}
($H^{-1}_{\alpha+1}$ being the dual of $H^1_{\alpha+1,0}$), satisfying
\begin{align}\label{xi_t}
\notag\int^T_0\big(\xi_t,v\big)_{L^2_\alpha}dt=&\int^T_0
\bigg[\big((\frac{\psi_{ss}}{{\psi}}+(n-1)\frac{\psi^2_s}{\psi^2})f\cdot(\eta+\xi),v\big)_{L^2_\alpha}
+\big(\frac{\psi_s}{\psi}f\xi_s,v\big)_{L^2_\alpha}\\
&-\big(g_s\xi_s,v\big)_{L^2_\alpha}-\big(g\xi_s,v_s\big)_{L^2_\alpha}
+2\alpha\big(g\xi_s,v\frac{\ell_s}{\ell}\big)_{L^2_\alpha}\\
\notag&+\big(\frac{\psi_s}{\psi}f\eta_s,v\big)_{L^2_\alpha}+\big(F_2,v\big)_{L^2_\alpha}\bigg]dt\\
\notag\xi\bigg|_{t=0}=\xi_0\qquad&\qquad\qquad\xi\bigg|_{(s_{min},t)}=\big(\xi\bigg|_{(s_0,t)},\;{\bf G}\big)=0
\qquad t\in[0,T],
\end{align}
for all
\begin{align}
v\in L^\infty(0,T;H^1_{\alpha,0}(s))\cap L^2(0,T;H^1_{\alpha+1}(s)),
\end{align}
where by $\big(\cdot,\cdot\big)_{L^2_\alpha}$ we denote the inner product in $H^0_\alpha$
\begin{align}\label{innprod}
\big(v_1,v_2\big)_{L^2_{\alpha}}:=\int\frac{v_1v_2}{\ell^{2\alpha}}ds.
\end{align}
We prove the existence of a (weak) solution to (\ref{xi_t}) in the energy spaces we claimed.
(Uniqueness of the strong solution will follow readily from the energy estimates we establish along the way).
Let $\{u_k(x)\}_{k=1}^\infty$ be an orthonormal basis of $L^2\big((0,+\infty)\big)$, which is also a basis
of $H^1_0\big((0,+\infty)\big)$ $\big[(0,\delta)$, {\bf G}\big]. Then, for each $t\in[0,T]$,
\begin{align}\label{basis}
w_k(s,t):=\ell^\alpha u_k(s-s_{min})&&\bigg[\ell^\alpha u_k\big(\frac{\delta(s-s_{min})}{s_0-s_{min}}\big),\;\;{\bf G}\bigg]
\end{align}
$k=1,2,\ldots$ is an orthonormal basis of $L^2_\alpha(s)$ $\big[$up to uniform rescalling ($t$ fixed), {\bf G}\big]
and a basis of $H^1_{\alpha,0}$. We note that by the comparison (\ref{l/s}) and (\ref{sgeq})
\begin{align}\label{L2L2w}
\int^T_0\int\frac{1}{\ell^2}dsdt\leq C\int^T_0\int\frac{1}{s^2}ds\leq C\int^T_0\frac{1}{s_{min}}ds\leq
C\int^T_0\frac{1}{\sqrt{t}}dt=o(1)<+\infty,
\end{align}
as $T\rightarrow0^+$. From this we derive that the set of functions
\begin{align}\label{L2L2dense}
\text{span}\big\{d_k(t)w_k(s,t)\;\big|\;t\in[0,T],\;k=1,2\ldots\big\},
\end{align}
$d_k(t)$ smooth, is dense in $L^2(0,T;H^1_{\alpha+1,0}(s))$.
Recall that by (\ref{l_s}) we have $\partial_s \ell=O(1)$ and hence, similarly to (\ref{L2L2w})
we verify the asymptotics (as $t\rightarrow0^+$)
\begin{align}\label{wasym1}
\int\frac{w_{k_1}w_{k_2}}{s^2\ell^{2\alpha}}ds=O(\frac{1}{\sqrt{t}})&&
\int\frac{\partial_sw_{k_1}w_{k_2}}{s\ell^{2\alpha}}ds=O(\frac{1}{\sqrt{t}})&&
\int\frac{\partial_sw_{k_1}\partial_sw_{k_2}}{\ell^{2\alpha}}ds=O(\frac{1}{\sqrt{t}}),
\end{align}
without of course  any uniformity in the RHSs with respect to the indices $k_1,k_2\in\{1,2,\ldots\}$.
Further, by the definition of our basis (\ref{basis}) we also find
\begin{align}\label{wasym2}
\int\frac{\partial_tw_{k_1}w_{k_2}}{\ell^{2\alpha}}ds=O(\frac{1}{\sqrt{t}}),
\end{align}
since $\partial_t\ell=O(\ell^{-1})$, as well as $\partial_ts=O(s^{-1}+1)$; see (\ref{l_t}), (\ref{s_t}), (\ref{s_tinfty}).

Now, given a $\nu\in\{1,2,\ldots\}$, we construct Galerkin approximations of the solution of (\ref{xi_t}), which lie
in the span of the first $\nu$ basis elements:
\begin{align}\label{xinu}
\xi^\nu:=\sum_{k=1}^\nu a_k(t)w_k&&a_k(0):=\int\frac{\xi_0w_k(x,0)}{x^{2\alpha}}dx
\end{align}
solving
\begin{align}\label{xigal}
\notag\big(\xi^\nu_t,w_k\big)_{L^2_\alpha}=&\;
\big((\frac{\psi_{ss}}{{\psi}}+(n-1)\frac{\psi^2_s}{\psi^2})f\cdot(\eta+\xi^\nu),w_k\big)_{L^2_\alpha}
+\big(\frac{\psi_s}{\psi}f\xi^\nu_s,w_k\big)_{L^2_\alpha}\\
&-\big(g_s\xi^\nu_s,w_k\big)_{L^2_\alpha}-\big(g\xi^\nu_s,\partial_sw_k\big)_{L^2_\alpha}
+2\alpha\big(g\xi^\nu_s,w_k\frac{\ell_s}{\ell}\big)_{L^2_\alpha}\\
\notag&+\big(\frac{\psi_s}{\psi}f\eta_s,w_k\big)_{L^2_\alpha}
+\big(F_2,w_k\big)_{L^2_\alpha},
\end{align}
for $t\in[0,T]$ and every $k=1,\ldots,\nu$.
\begin{proposition}\label{exgal}
For each $\nu=1,2,\ldots$ there exists a unique function $\xi^\nu$ of the form (\ref{xinu})
satisfying (\ref{xigal}).
\end{proposition}
\begin{proof}
Fix a $k\in\{1,\ldots,\nu\}$. We express (\ref{xigal}) as a linear ODE system in time
and estimate its coefficients by plugging into the equation the formula of $\xi^\nu$ (\ref{xinu})
and of the basic functions (\ref{basis}).
First, employing (\ref{wasym2}) we derive
\begin{align*}
\big(\xi^\nu_t,w_k\big)_{L^2_\alpha}=
a'_k(t)+\sum^\nu_{j=1}a_j(t)O_{j,k}(\frac{1}{\sqrt{t}}).
\end{align*}
Next, utilizing the estimates on the coefficients (\ref{coeff}), (\ref{coeffinfty}), (\ref{f})
and the above asymptotics (\ref{wasym1}) we can write
\begin{align*}
\big((\frac{\psi_{ss}}{{\psi}}+(n-1)\frac{\psi^2_s}{\psi^2})f\cdot\xi^\nu,w_k\big)_{L^2_\alpha}
+\big(\frac{\psi_s}{\psi}f\xi^\nu_s,w_k\big)_{L^2_\alpha}=\sum_{j=1}^\nu a_j(t)O_{j,k}(\frac{1}{\sqrt{t}}).
\end{align*}
Recall our assumption $w_k\in H^1_{\alpha,0}$ and apply (\ref{wasym1}) once more to estimate
\begin{align*}
-\big(g_s\xi^\nu_s,w_k\big)_{L^2_\alpha}-\big(g\xi^\nu_s,\partial_sw_k\big)_{L^2_\alpha}
&+2\alpha\big(g\xi^\nu_s,w_k\frac{\ell_s}{\ell}\big)_{L^2_\alpha}\\
&=\sum_{j=1}^\nu a_k(t)O_{j,k}(1)+\sum_{j=1}^\nu a_k(t)O_{j,k}(\frac{1}{\sqrt{t}}),
\end{align*}
where we used the estimates (\ref{f}) of $g$ and the fact $\partial_s\ell=O(1)$; see (\ref{l_s}).
For the last three terms of the RHS of (\ref{xigal}) we set (for a.e. $t\in[0,T]$)
\begin{align*}
d_k(t):=&\;\big((\frac{\psi_{ss}}{{\psi}}+(n-1)\frac{\psi^2_s}{\psi^2})f\cdot\eta,w_k\big)_{L^2_\alpha}
+\big(\frac{\psi_s}{\psi}f\eta_s,w_k\big)_{L^2_\alpha}+\big(F_2,w_k\big)_{L^2_\alpha}\\
\leq&\;C\|\frac{\eta}{s\ell^\alpha}\|^2_{L^2}+\int\frac{1}{s^2}ds
+C\|\frac{\eta_s}{\ell^\alpha}\|^2_{L^2}+\int\frac{1}{s^2}ds
+C\|\frac{F_2}{\ell^{\alpha-1}}\|^2_{L^2}+\int\frac{1}{\ell^2}ds
\end{align*}
Observe that $d_k(t)$ is integrable in time $t$: Indeed, recall our remark above (\ref{L2L2w})
and our regularity assumption on $F_2$ (\ref{F}). The $\eta$ terms are of course $t-$integrable
from the estimates we derived in \S\ref{modetaest}.

According to the above, (\ref{xigal}) reduces to a linear first order ODE system of the form
\begin{align*}
a'_k(t)=\sum^\nu_{j=1}a_k(t)O_{j,k}(\frac{1}{\sqrt{t}})+
\sum^\nu_{j=1}a_k(t)O_{j,k}(1)+d_k(t)&&k=1,\ldots,\nu
\end{align*}
The coefficients of the preceding system are thus singular at $t=0$,
but they are all integrable on $[0,T]$. This implies local existence
and uniqueness of the system and hence of Galerkin approximations (\ref{xinu}), (\ref{xigal}) of
(\ref{xi_t}) at each step $\nu\in\{1,2,\ldots\}$.
\end{proof}
We proceed to derive energy estimates for the approximate solutions $\xi^\nu$.
\begin{proposition}\label{enestxinu}
There exist $\alpha,\sigma$ appropriately large, such that
for every $\nu=1,2,\ldots$ the following estimates hold:
\begin{align}\label{Linftyxinu}
\|\xi^\nu\|^2_{L^\infty(0,T;L^2_\alpha(s))}
+\int^T_0\|\frac{\xi^\nu_s}{\ell^\alpha}&\|^2_{L^2}dt\\
\notag\leq&\;\frac{\widetilde{C}}{10}\bigg[\mathcal{E}_0
+\sum\int^T_0\|\frac{F_i}{\ell^{\alpha-1}}\|^2_{L^2}dt+\int^T_0G_2(t)dt\bigg]
\end{align}
and
\begin{align}\label{L2L2xinu}
\int^T_0\|\frac{\xi^\nu}{\ell^{\alpha+1}}\|^2_{L^2(s_{min},s_0)}dt
\leq\frac{\widetilde{C}}{\alpha\sigma^2}\bigg[\mathcal{E}_0
+\sum\int^T_0\|\frac{F_i}{\ell^{\alpha-1}}\|^2_{L^2}dt+\int^T_0G_2(t)dt\bigg],
\end{align}
for some $T>0$ independent $\nu$. In addition,
\begin{align}\label{L2xinu_t}
\bigg(\int^T_0\big(\xi_t^\nu,v&\big)_{L^2_\alpha}dt\bigg)^2\\
\notag\leq&\;\frac{\widetilde{C}}{10}\bigg[\mathcal{E}_0
+\sum\int^T_0\|\frac{F_i}{\ell^{\alpha-1}}\|^2_{L^2}dt+\int^T_0G_2(t)dt\bigg]\int^T_0\|v\|^2_{H^1_{\alpha+1},0}dt,
\end{align}
for all $v=\sum_{k=1}^\nu d_k(t)w_k$.
\end{proposition}
\begin{proof}
As in (\ref{L2eta}), by (\ref{partial_tds}) and (\ref{l_t}) we have
\begin{align}\label{L2xinu}
\frac{1}{2}\frac{d}{dt}\|\xi^\nu\|^2_{L^2_\alpha}
\leq&\int\frac{\xi^\nu\xi^\nu_t}{\ell^{2\alpha}}ds
-\alpha\sigma\|\frac{\xi^\nu}{\ell^{\alpha+1}}\|^2_{L^2(s_{min},s_0)}\\
&\notag+C\alpha\|\frac{\xi^\nu}{\ell^{\alpha+1}}\|^2_{L^2(s_{min},s_0)}
+C\alpha\|\frac{\xi^\nu}{\ell^\alpha}\|^2_{L^2}
+C\|\frac{\xi^\nu}{s\ell^\alpha}\|^2_{L^2}
\end{align}
Recall that from (\ref{l/s}) we estimate the last term as
\begin{align*}
\|\frac{\xi^\nu}{s\ell^\alpha}\|^2_{L^2}
\leq C\sigma\|\frac{\xi^\nu}{\ell^{\alpha+1}}\|^2_{L^2(s_{min},s_0)}
+C\|\frac{\xi^\nu}{\ell^\alpha}\|^2_{L^2(s_0,+\infty)},
\end{align*}
using the fact $\ell=O(1)$, $x\in(0,+\infty)$; see Definition \ref{weight}, (\ref{ints1}).
Recall the asymptotics of the coefficients
(\ref{coeff}), (\ref{coeffinfty}) and the poinwise estimate on f (\ref{f}).
Since $\xi^\nu$ is a linear combination (for each fixed $t$)
of the $w_k$'s, $k=1,\ldots,\nu$, by definition
of the inner product in $L^2_\alpha$  and (\ref{xigal}) we derive
\begin{align*}
\bullet&\;\;\int\frac{\xi^\nu\xi^\nu_t}{\ell^{2\alpha}}ds=
\big((\frac{\psi_{ss}}{{\psi}}+(n-1)\frac{\psi^2_s}{\psi^2})f\cdot(\eta+\xi^\nu),\xi^\nu\big)_{L^2_\alpha}
+\big(\frac{\psi_s}{\psi}f\xi^\nu_s,\xi^\nu\big)_{L^2_\alpha}\\
&-\big(g_s\xi^\nu_s,\xi^\nu\big)_{L^2_\alpha}-\big(g\xi^\nu_s,\xi^\nu_s\big)_{L^2_\alpha}
+2\alpha\big(g\xi^\nu_s,\xi^\nu\frac{\ell_s}{\ell}\big)_{L^2_\alpha}
\notag+\big(\frac{\psi_s}{\psi}f\eta_s,\xi^\nu\big)_{L^2_\alpha}
+\big(F_2,\xi^\nu\big)_{L^2_\alpha}\\
\leq&
\;C\|\frac{\xi^\nu}{s\ell^\alpha}\|^2_{L^2}
+C\|\frac{\eta}{s\ell^\alpha}\|^2_{L^2}
+\varepsilon\|\frac{\xi^\nu_s}{\ell^\alpha}\|^2_{L^2}
+\frac{C}{\varepsilon}\|\frac{\xi^\nu}{s\ell^\alpha}\|^2_{L^2}
-\big(g_s\xi^\nu_s,\xi^\nu\big)_{L^2_\alpha}\\
\tag{using the pointwise bounds (\ref{f}) of $g$}
&-\frac{1}{2}\|\frac{\xi^\nu_s}{\ell^\alpha}\|^2_{L^2}
+\varepsilon\|\frac{\xi^\nu_s}{\ell^\alpha}\|^2_{L^2}
+\frac{C}{\varepsilon}\alpha^2\|\frac{\xi^\nu}{\ell^{\alpha+1}}\|^2_{L^2}\\
&+\|\frac{\xi^\nu}{s\ell^\alpha}\|^2_{L^2}
+C\|\frac{\eta_s}{\ell^\alpha}\|^2_{L^2}
+C\|\frac{\xi^\nu}{\ell^{\alpha+1}}\|^2_{L^2}
+\|\frac{F_2}{\ell^{\alpha-1}}\|^2_{L^2}\\
\tag{utilizing (\ref{l/s}), $\alpha\ge1$}
\leq&\;\frac{C}{\varepsilon}(\sigma+\alpha^2)\|\frac{\xi^\nu}{\ell^{\alpha+1}}\|^2_{L^2(s_{min},s_0)}
+\frac{C}{\varepsilon}\alpha^2\|\frac{\xi^\nu}{\ell^\alpha}\|^2_{L^2(s_0,+\infty)}\\
&+C\sigma\|\frac{\eta}{\ell^{\alpha+1}}\|^2_{L^2(s_{min},s_0)}
+C\|\frac{\eta}{\ell^\alpha}\|^2_{L^2(s_0,+\infty)}
+C\|\frac{\eta_s}{\ell^\alpha}\|^2_{L^2(s_{min},s_0)}\\
&+C\|\frac{\eta_s}{\ell^{\alpha-1}}\|^2_{L^2(s_0,+\infty)}
+(2\varepsilon-\frac{1}{2})\|\frac{\xi^\nu_s}{\ell^\alpha}\|^2_{L^2}
+\|\frac{F_2}{\ell^{\alpha-1}}\|^2_{L^2}
-\big(g_s\xi^\nu_s,\xi^\nu\big)_{L^2_\alpha}
\end{align*}
In order to estimate the remaining term
we will need the following $L^\infty$ bound:
As in the derivation of (\ref{pdxi}),
by the boundary condition we imposed on our space we get
\begin{align*}
\|\frac{\xi^\nu}{\ell^k}\|^2_{L^\infty(s)}
\leq C\big(\|\frac{\xi^\nu_s}{\ell^\alpha}\|^2+(k+1)\|\frac{\xi^\nu}{\ell^{\alpha+1}}\|^2\big),
\end{align*}
provided $0\leq k\leq\alpha$. Now we finish the argument by readily estimating
\begin{align*}
-\big(g_s\xi^\nu_s,\xi^\nu\big)_{L^2_\alpha}
\leq&\;\varepsilon\|\frac{\xi^\nu_s}{\ell^\alpha}\|^2_{L^2}
+\frac{1}{\varepsilon}\|\frac{\xi^\nu}{\ell}\|^2_{L^\infty(s)}\|\frac{g_s}{\ell^{\alpha-1}}\|^2_{L^2}\\
\leq&\;(\varepsilon+\frac{C}{\varepsilon}\tilde{\varepsilon})\|\frac{\xi^\nu_s}{\ell^\alpha}\|^2_{L^2},
\end{align*}
where we made use of the $L^2_{\alpha-1}$ estimate on $g$ (\ref{f}).
Combining the above estimates and plugging into (\ref{L2xinu}) we deduce
\begin{align}\label{sumxinu1}
&\frac{1}{2}\frac{d}{dt}\|\xi^\nu\|^2_{L^2_\alpha}
+\alpha\sigma\|\frac{\xi^\nu}{\ell^{\alpha+1}}\|^2_{L^2(s_{min},s_0)}\\
\notag\leq&\;\frac{C}{\varepsilon}(\sigma+\alpha^2)\|\frac{\xi^\nu}{\ell^{\alpha+1}}\|^2_{L^2(s_{min},s_0)}
+\frac{C}{\varepsilon}\|\xi^\nu\|^2_{L^2_\alpha}
+(3\varepsilon+\frac{C}{\varepsilon}\tilde{\varepsilon}-\frac{1}{2})\|\frac{\xi^\nu_s}{\ell^\alpha}\|^2_{L^2}\\
&+C\sigma\|\frac{\eta}{\ell^{\alpha+1}}\|^2_{L^2(s_{min},s_0)}+C\|\frac{\eta_s}{\ell^\alpha}\|^2_{L^2(s_{min},s_0)}
+C\|\eta\|^2_{L^2_\alpha}+C\|\eta_s\|^2_{L^2_{\alpha-1}}
\notag+\|\frac{F_2}{\ell^{\alpha-1}}\|^2_{L^2},
\end{align}
Let $\varepsilon,\tilde{\varepsilon}$ be sufficiently small and $\alpha,\sigma$ large satisfying
\begin{align*}
3\varepsilon+\frac{C}{\varepsilon}\tilde{\varepsilon}<\frac{1}{4}&&
\frac{1}{4}\alpha>\frac{C}{\varepsilon}&&\frac{1}{4}\sigma>\frac{C}{\varepsilon}\alpha.
\end{align*}
Then, integrating on $[0,t]$, $t\leq T$, and invoking the estimates on $\eta$ (\ref{Linftymodeta}),
(\ref{L2L2modeta}), (\ref{Linftymodeta_s}), (\ref{L2L2modeta_s}) we conclude
\begin{align}\label{sumxinu2}
&\frac{1}{2}\|\xi^\nu\|^2_{L^2_\alpha}
+\frac{1}{2}\alpha\sigma\int^t_0\|\frac{\xi^\nu}{\ell^{\alpha+1}}\|^2_{L^2(s_{min},s_0)}d\tau
+\frac{1}{4}\int^t_0\|\frac{\xi^\nu_s}{\ell^\alpha}\|^2_{L^2}d\tau\\
\notag\leq&\;\frac{1}{2}\|\xi^\nu_0\|^2_{L^2_\alpha}+C\int^t_0\|\xi^\nu\|^2_{L^2_\alpha}d\tau
+C\sigma\int^T_0\|\frac{\eta}{\ell^{\alpha+1}}\|^2_{L^2(s_{min},s_0)}
+C\int^T_0\|\frac{\eta_s}{\ell^\alpha}\|^2_{L^2(s_{min},s_0)}d\tau\\
&+C\int^T_0\big(\|\eta\|^2_{L^2_\alpha}+\|\eta_s\|^2_{L^2_{\alpha-1}}\big)d\tau
\notag+\int^T_0\|\frac{F_2}{\ell^{\alpha-1}}\|^2_{L^2}d\tau\\
\notag\leq&\;\frac{1}{2}\|\xi^\nu_0\|^2_{L^2_\alpha}+C\int^t_0\|\xi^\nu\|^2_{L^2_\alpha}d\tau
+C\frac{\widetilde{C}}{\alpha\sigma}\bigg[\mathcal{E}_0
+\sum\int^T_0\|\frac{F_i}{\ell^{\alpha-1}}\|^2_{L^2}d\tau+\int^T_0G_2(\tau)d\tau\bigg]\\
&+CT\frac{\widetilde{C}}{5}\bigg[\mathcal{E}_0
+\sum\int^T_0\|\frac{F_i}{\ell^{\alpha-1}}\|^2_{L^2}d\tau+\int^T_0G_2(\tau)d\tau\bigg]
\notag+\int^T_0\|\frac{F_2}{\ell^{\alpha-1}}\|^2_{L^2}d\tau,
\end{align}
where $\xi^\nu_0:=\xi^\nu(x,0)$. Since $\{w_k\}_{k\in\mathbb{Z}}$ is an orthonormal
basis of $L^2_\alpha$, from (\ref{xinu}) it follows
\begin{align*}
\|\xi^\nu_0\|^2_{L^2_\alpha}\leq\|\xi_0\|^2_{L^2_\alpha}.
\end{align*}
Thus, by Gronwall's inequality, $t\in[0,T]$, we have ($T>0$ small, $\widetilde{C}$ large)
\begin{align}\label{xinuest}
\|\xi^\nu\|^2_{L^\infty(0,T;L^2_\alpha(s))}
+\int^T_0\|\frac{\xi^\nu_s}{\ell^\alpha}&\|^2_{L^2}dt\\
\notag\leq&\;C\frac{\widetilde{C}}{\alpha\sigma}\bigg[\mathcal{E}_0
+\sum\int^T_0\|\frac{F_i}{\ell^{\alpha-1}}\|^2_{L^2}dt+\int^T_0G_2(t)dt\bigg].
\end{align}
For $\alpha$ sufficiently large, together with (\ref{sumxinu2}), we achieve
the required estimates (\ref{Linftyxinu}), (\ref{L2L2xinu}). Finally, setting
$v=\sum^\nu_{k=1}d_k(t)w_k$ and going back
to (\ref{xigal}) we apply the same arguments as above to obtain the estimate (for a.e. $t\in[0,T]$)
\begin{align}
\big(\xi_t^\nu,v\big)_{L^2_\alpha}\leq&\;
C\big(\|\frac{v}{\ell^{\alpha+1}}\|_{L^2}+\|\frac{v_s}{\ell^\alpha}\|_{L^2}\big)
\bigg[\|\frac{\eta}{s^2\ell^{\alpha-1}}\|_{L^2}
+\|\frac{\xi^\nu}{s^2\ell^{\alpha-1}}\|_{L^2}+\|\frac{\xi^\nu_s}{s\ell^{\alpha-1}}\|_{L^2}\\
&+\alpha^2\|\frac{\xi^\nu_s}{\ell^\alpha}\|_{L^2}
+\|\frac{\eta_s}{s\ell^{\alpha-1}}\|_{L^2}+\|\frac{F_2}{\ell^{\alpha-1}}\|_{L^2}\bigg]
\end{align}
Integrate now on $[0,T]$, apply C-S and take squares both sides.
Employing first the comparison estimate (\ref{l/s}), then the estimates on $\eta$ (\ref{Linftymodeta}),
(\ref{L2L2modeta}), (\ref{Linftymodeta_s}), (\ref{L2L2modeta_s}) and the just derived estimates
of $\xi^\nu$ (\ref{xinuest}), (\ref{L2L2xinu}) we arrive at (\ref{L2xinu_t}), for $\alpha,\sigma$ large enough.
\end{proof}
The estimates in Proposition \ref{enestxinu}
suffice to pass to a subsequence (applying a diagonal argument due to (\ref{L2xinu_t}))
and take weak limits in (\ref{xigal}), after integrating on $[0,T]$, to obtain
a weak solution $\xi$ of (\ref{modlpde}), in the sense of (\ref{xi_t}), and
in the desired spaces (\ref{xisp})
\begin{align}\label{xisp}
\xi\in L^\infty(0,T;L^2_\alpha(s))\cap L^2(0,T;H^1_{\alpha+1,0}(s))&&
\ell^2\xi_t\in L^2(0,T;H^{-1}_{\alpha+1}(s)),
\end{align}
satisfying the energy estimates
\begin{align}\label{Linftyxi}
\|\xi\|^2_{L^\infty(0,T;L^2_\alpha(s))}
+\int^T_0\|\frac{\xi_s}{\ell^\alpha}&\|^2_{L^2}dt\\
\notag\leq&\;\frac{\widetilde{C}}{10}\bigg[\mathcal{E}_0
+\sum\int^T_0\|\frac{F_i}{\ell^{\alpha-1}}\|^2_{L^2}dt+\int^T_0G_2(t)dt\bigg]
\end{align}
and
\begin{align}\label{L2L2xi}
\int^T_0\|\frac{\xi}{\ell^{\alpha+1}}\|^2_{L^2(s_{min},s_0)}dt
\leq\frac{\widetilde{C}}{\alpha\sigma^2}\bigg[\mathcal{E}_0
+\sum\int^T_0\|\frac{F_i}{\ell^{\alpha-1}}\|^2_{L^2}dt+\int^T_0G_2(t)dt\bigg].
\end{align}
\subsection{Improved regularity and energy estimates for $\xi$}\label{xiimpr}
\noindent
We now show that $\xi$ is a strong solution of (\ref{modlpde}).
Consider a positive time $0<t_0<T$. Looking at the second equation of (\ref{modlpde})
for $t\in[t_0,T]$, we notice that those coefficients which involve $\psi$ and its
derivatives are smooth and bounded, while $f,g\in L^\infty(0,T;H^1(s))$ by consideration (\ref{f,Freg}).
Moreover, from \S\ref{modetaest} we have
$\eta\in L^\infty(0,T;H^1)$ and by assumption $F_i\in L^2(0,T;L^2(s))$, $i=1,2$.
Hence, by standard theory of parabolic equations \cite[\S7.1 Theorem 5]{E},
the solution $\xi$ of (\ref{modlpde}) that we established above (\ref{xisp}),
having ``initial data'' $\xi(s,t_0)\in H^1(s)$ (for a.e. $0<t_0<T$), attains
interior regularity
\begin{align*}
\xi\in L^\infty(t_0,T;H^1_0(s))\cap L^2(t_0,T;H^2(s))&&\xi_t\in L^2(t_0,T;L^2(s))
\end{align*}
Since $t_0\in(0,T)$ is arbitrary, the following derivations make sense. Recall
that for fixed $t>0$, the weight $\ell^2$ is bounded below; see
Definition \ref{weight}.

We proceed to derive the regularity
\begin{align*}
\xi\in L^\infty(0,T;H^1_{\alpha,0})\cap L^2(0,T;H^2_{\alpha+1})&&
\xi_t\in L^2(0,T;L^2_{\alpha-1})
\end{align*}
and the remaining energy estimates we claimed in \S\ref{modsolproof}
for our solution $\xi$ (\ref{xisp}) of (\ref{modlpde}).
Similarly to (\ref{L2dxi_s}),
utilizing the boundary condition in (\ref{modlpde}) and (\ref{l_t}),
(\ref{partial_tds}), (\ref{[s,t]}), (\ref{l_s}) we derive (for a.e. $t\in[0,T]$)
\begin{align}\label{L2xi_s}
\notag\frac{1}{2}\frac{d}{dt}\|\xi_s\|^2_{L^2_{\alpha-1}}
=&\int\frac{\xi_s}{\ell^{2\alpha-2}}(\partial_s\xi_t
-n\frac{\psi_{ss}}{\psi}\xi_s)ds
-(\alpha-1)\int\frac{\xi^2_s}{\ell^{2\alpha-1}}\partial_t\ell ds
+\frac{1}{2}\int\frac{\xi^2_s}{\ell^{2\alpha-2}}\partial_tds\\
\leq&-\int\frac{\xi_{ss}\xi_t}{\ell^{2\alpha-2}}ds
+(2\alpha-2)\int\frac{\xi_s\xi_t}{\ell^{2\alpha-1}}\ell_sds
+C\|\frac{\xi_s}{s\ell^{\alpha-1}}\|^2_{L^2}\\
&-(\alpha-1)\sigma\|\frac{\xi_s}{\ell^\alpha}\|^2_{L^2(s_{min},s_0)}
\notag+C(\alpha-1)\|\frac{\xi_s}{\ell^\alpha}\|^2_{L^2}
+C\|\frac{\xi_s}{s\ell^{\alpha-1}}\|^2_{L^2}
\end{align}
As usual, we replace the weight $s$ with $\ell$ employing
the comparison estimate (\ref{l/s})
\begin{align*}
\|\frac{\xi_s}{s\ell^{\alpha-1}}\|^2_{L^2}\leq
C\sigma\|\frac{\xi_s}{\ell^\alpha}\|^2_{L^2(s_{min},s_0)}+
C\|\frac{\xi_s}{\ell^{\alpha-1}}\|^2_{L^2(s_0,+\infty)}
\end{align*}
There are two terms we need estimate. For what follows we recall
the estimates on the coefficients (\ref{coeff}), (\ref{coeffinfty}). Also,
note that by (\ref{l_s}) we get the crude (but sufficient) estimate $\partial_s\ell=O(1)$.
\begin{align*}
\tag{plugging in the RHS of the second equation of (\ref{modlpde})}
\bullet&\;\;(2\alpha-2)\int\frac{\xi_s\xi_t}{\ell^{2\alpha-1}}\ell_sds\\
=&\;(2\alpha-2)\int\frac{\xi_s}{\ell^{2\alpha-1}}\ell_s\bigg[
(\frac{\psi_{ss}}{{\psi}}+(n-1)\frac{\psi^2_s}{\psi^2})f\cdot(\eta+\xi)
+\frac{\psi_s}{\psi}f\xi_s+g\xi_{ss}+\frac{\psi_s}{\psi}f\eta_s+F_2
\bigg]ds\\
\leq&\;\alpha^2\|\frac{\xi_s}{\ell^\alpha}\|^2_{L^2}
+C\|\frac{\eta}{s^2\ell^{\alpha-1}}\|^2_{L^2}
+C\|\frac{\xi}{s^2\ell^{\alpha-1}}\|^2_{L^2}
+C\|\frac{\xi_s}{s\ell^{\alpha-1}}\|^2_{L^2}
+\varepsilon\|\frac{\xi_{ss}}{\ell^{\alpha-1}}\|^2_{L^2}
+\frac{C}{\varepsilon}\alpha^2\|\frac{\xi_s}{\ell^\alpha}\|^2_{L^2}\\
\tag{using the $L^\infty$ estimate (\ref{f})}
&+C\alpha^2\|\frac{\xi_s}{\ell^\alpha}\|^2_{L^2}
+\|\frac{\eta_s}{s\ell^{\alpha-1}}\|^2_{L^2}
+\|\frac{F_2}{\ell^{\alpha-1}}\|^2_{L^2}\\
\tag{by (\ref{l/s}), $\alpha\ge1$}
\leq&\;\frac{C}{\varepsilon}\alpha^2\|\frac{\xi_s}{\ell^\alpha}\|^2_{L^2}
+C\sigma^2\|\frac{\eta}{\ell^{\alpha+1}}\|^2_{L^2(s_{min},s_0)}
+C\|\frac{\eta}{\ell^{\alpha+1}}\|^2_{L^2(s_0,+\infty)}\\
&+C\sigma^2\|\frac{\xi}{\ell^{\alpha+1}}\|^2_{L^2(s_{min},s_0)}
+C\|\frac{\xi}{\ell^{\alpha+1}}\|^2_{L^2(s_0,+\infty)}
+\varepsilon\|\frac{\xi_{ss}}{\ell^{\alpha-1}}\|^2_{L^2}
+C\sigma\|\frac{\eta_s}{\ell^\alpha}\|^2_{L^2(s_{min},s_0)}\\
\tag{$0<\varepsilon<1$}
&+C\|\frac{\eta_s}{\ell^\alpha}\|^2_{L^2(s_0,+\infty)}
+\|\frac{F_2}{\ell^{\alpha-1}}\|^2_{L^2}
\end{align*}
Next, we estimate the term
\begin{align*}
\tag{substituting from (\ref{modlpde})}
\bullet\;\;\;&-\int\frac{\xi_{ss}\xi_t}{\ell^{2\alpha-2}}ds\\
=&\;-\int\frac{\xi_{ss}}{\ell^{2\alpha-2}}\bigg[
(\frac{\psi_{ss}}{{\psi}}+(n-1)\frac{\psi^2_s}{\psi^2})f\cdot(\eta+\xi)
+\frac{\psi_s}{\psi}f\xi_s+g\xi_{ss}+\frac{\psi_s}{\psi}f\eta_s+F_2
\bigg]ds\\
\tag{employing (\ref{f}) again for $f,g$}
\leq&\;\varepsilon\|\frac{\xi_{ss}}{\ell^{\alpha-1}}\|^2_{L^2}
+\frac{C}{\varepsilon}\|\frac{\eta}{s^2\ell^{\alpha-1}}\|^2_{L^2}
+\frac{C}{\varepsilon}\|\frac{\xi}{s^2\ell^{\alpha-1}}\|^2_{L^2}\\
&+\frac{C}{\varepsilon}\|\frac{\xi_s}{s\ell^{\alpha-1}}\|^2_{L^2}
-\frac{1}{2}\|\frac{\xi_{ss}}{\ell^{\alpha-1}}\|^2_{L^2}
+\varepsilon\|\frac{\xi_{ss}}{\ell^{\alpha-1}}\|^2_{L^2}
+\frac{C}{\varepsilon}\|\frac{\eta_s}{s\ell^{\alpha-1}}\|^2_{L^2}
+\frac{1}{\varepsilon}\|\frac{F_2}{\ell^{\alpha-1}}\|^2_{L^2}\\
\tag{applying (\ref{l/s})}
\leq&\;(2\varepsilon-\frac{1}{2})\|\frac{\xi_{ss}}{\ell^{\alpha-1}}\|^2_{L^2}
+\frac{C}{\varepsilon}\sigma^2\|\frac{\eta}{\ell^{\alpha+1}}\|^2_{L^2(s_{min},s_0)}
+\frac{C}{\varepsilon}\|\frac{\eta}{\ell^{\alpha+1}}\|^2_{L^2(s_0,+\infty)}\\
&+\frac{C}{\varepsilon}\sigma^2\|\frac{\xi}{\ell^{\alpha+1}}\|^2_{L^2(s_{min},s_0)}
+\frac{C}{\varepsilon}\|\frac{\xi}{\ell^{\alpha+1}}\|^2_{L^2(s_0,+\infty)}
+\frac{C}{\varepsilon}\sigma\|\frac{\xi_s}{\ell^\alpha}\|^2_{L^2(s_{min},s_0)}\\
&+\frac{C}{\varepsilon}\|\frac{\xi_s}{\ell^{\alpha-1}}\|^2_{L^2(s_0,+\infty)}
+\frac{C}{\varepsilon}\sigma\|\frac{\eta_s}{\ell^\alpha}\|^2_{L^2(s_{min},s_0)}
+\frac{C}{\varepsilon}\|\frac{\eta_s}{\ell^{\alpha-1}}\|^2_{L^2(s_0,+\infty)}
+\frac{1}{\varepsilon}\|\frac{F_2}{\ell^{\alpha-1}}\|^2_{L^2},
\end{align*}
Plugging the above estimates into (\ref{L2xi_s}) we deduce
\begin{align}\label{summodxi_s1}
&\frac{1}{2}\frac{d}{dt}\|\xi_s\|^2_{L^2_{\alpha-1}}
+(\alpha-1)\sigma\|\frac{\xi_s}{\ell^\alpha}\|^2_{L^2(s_{min},s_0)}\\
\notag\leq&\;\frac{C}{\varepsilon}(\alpha^2+\sigma)\|\frac{\xi_s}{\ell^\alpha}\|^2_{L^2(s_{min},s_0)}
+\frac{C}{\varepsilon}\alpha^2\|\xi_s\|^2_{L^2_{\alpha-1}}
+(3\varepsilon-\frac{1}{2})\|\frac{\xi_{ss}}{\ell^{\alpha-1}}\|^2_{L^2}\\
\notag&+C\sigma^2\big(\|\frac{\eta}{\ell^{\alpha+1}}\|^2_{L^2(s_{min},s_0)}
+\|\frac{\xi}{\ell^{\alpha+1}}\|^2_{L^2(s_{min},s_0)}\big)
+C\sigma\big(\|\frac{\eta_s}{\ell^\alpha}\|^2_{L^2(s_{min},s_0)}
+\|\frac{\xi_s}{\ell^\alpha}\|^2_{L^2(s_{min},s_0)}\big)\\
\notag&+C\big(\|\eta\|^2_{L^2_\alpha}+\|\xi\|^2_{L^2_\alpha}+\|\eta_s\|^2_{L^2_{\alpha-1}}\big)
+(\frac{1}{\varepsilon}+1)\|\frac{F_2}{\ell^{\alpha-1}}\|^2_{L^2}
\end{align}
Let $\varepsilon$ be small and $\alpha,\sigma$ large such that
\begin{align*}
3\varepsilon<\frac{1}{4}&&\frac{1}{4}(\alpha-1)>\frac{C}{\varepsilon}&&
\frac{1}{4}\sigma>\frac{C}{\varepsilon}\frac{\alpha^2}{\alpha-1}
\end{align*}
Then, integrating on $[0,t]$, $t\leq T$, and invoking the derived estimates in
the previous subsections for $\eta$ (\ref{Linftymodeta}), (\ref{L2L2modeta}),
(\ref{Linftymodeta_s}), (\ref{L2L2modeta_s}) and $\xi$ (\ref{Linftyxinu}),
(\ref{L2L2xinu}) we obtain
\begin{align}\label{summodxi_s2}
&\frac{1}{2}\|\xi_s\|^2_{L^2_{\alpha-1}}
+\frac{1}{2}(\alpha-1)\sigma\int^t_0\|\frac{\xi_s}{\ell^\alpha}\|^2_{L^2}d\tau
+\frac{1}{4}\int^t_0\|\frac{\xi_{ss}}{\ell^{\alpha-1}}\|^2_{L^2}d\tau\\
\notag\leq&\;\frac{1}{2}\|\partial_x\xi_0\|^2_{L^2_{\alpha-1}}
+C\alpha^2\int^t_0\|\xi_s\|^2_{L^2_{\alpha-1}}d\tau\\
&\notag+C(\frac{1}{\alpha}+T)\widetilde{C}\bigg[\mathcal{E}_0
+\sum\int^T_0\|\frac{F_i}{\ell^{\alpha-1}}\|^2_{L^2}d\tau+\int^T_0G_2(\tau)d\tau\bigg]
+C\int^T_0\|\frac{F_2}{\ell^{\alpha-1}}\|^2_{L^2}d\tau
\end{align}
Thus, Gronwall's inequality, $t\in[0,T]$, we finally conclude (for $T>0$ small, $\alpha$ large)
\begin{align}\label{Linftymodxi_s}
\|\xi_s\|^2_{L^\infty(0,T;L^2_{\alpha-1}(s))}
+&\int^T_0\|\frac{\xi_{ss}}{\ell^{\alpha-1}}\|^2_{L^2}d\tau\\
\notag\leq&\;\frac{\widetilde{C}}{10}\bigg[\mathcal{E}_0
+\sum\int^T_0\|\frac{F_i}{\ell^{\alpha-1}}\|^2_{L^2}d\tau+\int^T_0G_2(\tau)d\tau\bigg]
\end{align}
and
\begin{align}\label{L2L2modxi_s}
\int^T_0\|\frac{\xi_s}{\ell^\alpha}\|^2_{L^2(s_{min},s_0)}d\tau
\leq\frac{\widetilde{C}}{(\alpha-1)\sigma}\bigg[\mathcal{E}_0
+\sum\int^T_0\|\frac{F_i}{\ell^{\alpha-1}}\|^2_{L^2}d\tau+\int^T_0G_2(\tau)d\tau\bigg].
\end{align}
This concludes our proof of the linear step.

\end{document}